\documentclass[10pt,a4paper,reqno]{amsart}
\usepackage{graphicx,multirow,array,amsmath,amssymb,enumerate,tikz-cd,mathtools}

\newtheorem{thm}{Theorem}[section]
\newtheorem{lem}[thm]{Lemma}
\newtheorem{cor}[thm]{Corollary}
\newtheorem{prop}[thm]{Proposition}
\newtheorem{claim}[thm]{Claim}

\newtheorem{notation}[thm]{Notation}
\newtheorem{rmk}[thm]{Remark}

\theoremstyle{definition}
\newtheorem{exam}[thm]{Example}

\numberwithin{equation}{section}
\numberwithin{figure}{section}

\newcommand{\Z}{\mathbb{Z}}
\newcommand{\C}{\mathbb{C}}
\newcommand{\CP}{\mathbb{CP}}
\newcommand{\CPbar}{\overline{\mathbb{CP}}}
\newcommand{\ep}{\epsilon}
\newcommand{\arrowr}{\rightarrow}

\begin{document}

\title[Splitting singular fibers]
{Splitting singular fibers with periodic monodromies and their monodromy factorization}

\author[H. Kim]{Hyunggi Kim}
\address{Research Institute of Mathematics, Seoul National University, Korea}
\email{hyunggi@snu.ac.kr}


\maketitle

\begin{abstract}
A Lefschetz fibration is a smooth 4-manifold admitting a surface bundle structure over a surface except at finitely many singular fibers, whose singularities are only of nodal type.
From the structure of the singular fibers, the monodromy of each singular fiber is given by a right-handed Dehn twist along a curve, called a vanishing cycle, in the fiber.
Collecting all monodromy data from the singular fibers, we obtain a monodromy factorization into right-handed Dehn twists, which completely determines the Lefschetz fibration.

In this paper, we construct a fibration with one singular fiber whose monodromy is periodic (that is, the monodromy homeomorphism is isotopic to a periodic map).
Following the idea of Matsumoto, we give a splitting of the singular fiber into Lefschetz fibers and determine their vanishing cycles for a collection of periodic monodromies.
We describe the construction of the splitting singular fibers and the procedure for reading vanishing cycles using two branched covering structures of the fibers.
We also give splittings of singular fibers into Lefschetz fibers related to a family of periodic actions and determine their vanishing cycles.
\end{abstract}

\section{Introduction} \label{sec:intro}
One method for studying smooth 4-manifolds is to find a projection map onto a surface and examine its surface bundle structure.
When the map has critical points, each fiber containing critical points is called a \textit{singular fiber}.
The simplest singularity of a singular fiber is a nodal singularity; that is,
if a local coordinate $(z_1, z_2)\in \C^2$ of the total space is chosen, then the map is given by $f(z_1,z_2)=z_1^2+z_2^2$.
Such a map which has only nodal singularities is called a \textit{Lefschetz fibration}.
Lefschetz fibrations have become important objects in the study of symplectic 4-manifolds, following the works of \cite{D} and \cite{Go}.

A singular fiber of a Lefschetz fibration is closely related to a right-handed Dehn twist along a simple closed curve in a surface, known as a \textit{vanishing cycle}.
By collecting all information from the singular fibers, we obtain a \textit{monodromy representation}.
Conversely, starting from a monodromy representation, we can construct a Lefschetz fibration,
and under some appropriate equivalence relations, this correspondence is bijective.

Though a monodromy factorization is a convenient tool for creating new examples using various relations in mapping class groups,
finding a monodromy factorization from a given Lefschetz fibration is not easy, because there is no general method for transporting vanishing cycles corresponding to each singularity to a single fiber,
called a \textit{reference fiber}.

There have been several studies attempting to find a monodromy factorization from a given Lefschetz fibration.
Matsumoto \cite{M1} found vanishing cycles in the genus-3 Lefschetz fibration on the Fermat surface of degree 5,
and later Hamada and Hayano \cite{HH} gave a system of vanishing cycles for a genus-3 Lefschetz pencil on the four-torus constructed by Smith.
Both papers mainly use complex hypersurfaces in $\CP^3$, so they possess canonical coordinates for calculating vanishing cycles.
Ishizaka \cite{I} used a hyperelliptic branched cover along a complex curve in $\CP^1 \times D^2$ defined by an equation,
and found vanishing cycles for some building blocks of Lefschetz fibrations of fiber genus $g$ with hyperelliptic monodromy.
Each of these papers utilizes a canonical coordinate system from the ambient space or the product structure,
along with its lifting under the double branched cover provided by the construction.
If a given Lefschetz fibration has no canonical coordinate system, then it is quite difficult to trace the vanishing cycles.

More recently, Sakallı and Van Horn-Morris \cite{SV1,SV2} studied splittings of singular fibers in genus-two fibrations arising from explicit algebraic equations. 
By constructing suitable perturbations of these fibrations, they obtain genus-two Lefschetz fibrations and determine the corresponding vanishing cycles and monodromy factorizations. 
In \cite{SV1}, the fibrations are given by explicit hyperelliptic equations, and the vanishing cycles are studied through the associated double branched-cover descriptions. 
In \cite{SV2}, they consider another family of algebraic fibrations equipped with an involution inherent in the defining polynomial, and analyze the resulting double branched-cover structure over the quotient fibration.

An exceptional method was provided by Matsumoto \cite{M3}.
He constructed a genus-2 Lefschetz fibration whose total space is diffeomorphic to $S^2\times T^2 \# 4\CPbar^2$.
He first considered two involutions $\sigma$ and $\tau$ on a genus-2 surface $\Sigma_2$ and $S^2$, respectively,
and the resolution of the quotient space $\Sigma_2 \times S^2/ \sigma \times \tau$.
Composing this with the projection map, we obtain a map $f: X \arrowr S^2$ with two singular fibers.
He split the singular (but not Lefschetz-type) fibers into four Lefschetz singular fibers
and identified the vanishing cycles by using the method of introducing four local charts for each singular fiber.
Since all critical points occur in only one of these charts,
the main calculation can be performed in some open subset of $\C^2$.

This paper is inspired by \cite{M3}.
First, instead of two involutions, consider a particular degree-$(2n+1)$ map $\sigma_n$ on $\Sigma_{3n+1}$ and $2\pi/n$ rotation map $\tau_n$ on $S^2$.
Similarly, by taking the quotient and the resolution, we obtain a complex surface $X$ and a map $\phi: X \arrowr S^2$.
Then the total space $X$ is diffeomorphic to $S^2 \times T^2 \# (4n+5)\CPbar^2$.
In a general setting, we can determine the diffeomorphism type of the total space, as stated in Theorem~\ref{thm:total}.
Now, take a perturbation $\phi_\ep: X_\ep \arrowr S^2$ of $\phi: X \arrowr S^2$ so that
$\phi_\ep$ has only Lefschetz singular fibers.
Finally, we calculate the monodromy factorization of this Lefschetz fibration.

\begin{thm}\label{thm:main}
For each $n\ge 1$, there is a genus $(3n+1)$-Lefschetz fibration on $(S^2 \times T^2) \# (4n+5)\CPbar^2$ over $S^2$ whose monodromy factorization is given by
\begin{align}
(t_{d} t_{d'})^2 \left(\prod_{j=12}^7 \prod_{i=n}^1 t_{c_{i,j}} \right)  \prod_{i=1}^n (t_{d_i} t_{d'_i})^2 \cdot t_{\delta} \left(\prod_{j=6}^1 \prod_{i=n}^1 t_{c_{i,j}} \right)  =1 
\end{align}
where $t_c$ denotes the right-handed Dehn twist along a curve $c$.
The exact description of vanishing cycles $c_{i,j}$, $\delta, d, d', d_i,$ and $d'_i$ is described in \eqref{eq:vc_comb1}, \eqref{eq:dd'}, and \eqref{eq:vc_comb2}.
\end{thm}

The vanishing cycles can be simply drawn in $\C$, using the hyperelliptic quotient $q: \Sigma_{3n+1,1} \arrowr D^2\cong \C$.
\begin{figure}[h]
\centering
\includegraphics[scale=0.8]{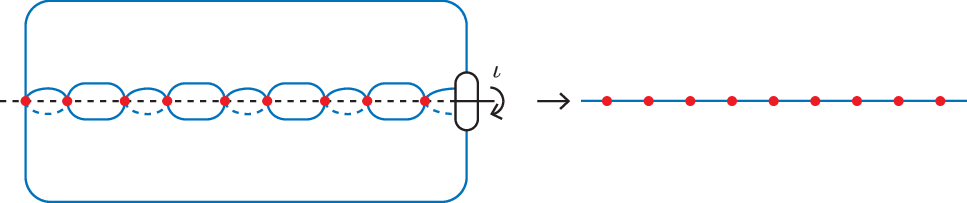}
\caption{Hyperelliptic quotient map. Blue curves map to the straight line under the map $\iota$.}
\label{fig:hyperelliptic}
\end{figure}

Note that under the quotient $q$, the curves $d_i, d'_i$ are mapped to a single curve $\tilde{d}_i$, and $d,d'$ to a single curve $\tilde{d}$.
Moreover, since $c_{i,j}$ is hyperelliptic curve, $q$ maps from $c_{i,j}$ to a curve $t^{-1}_{\tilde{d}}(\tilde{c}_{i,j})\subset \C$.

\begin{figure}[h]
\centering
\includegraphics{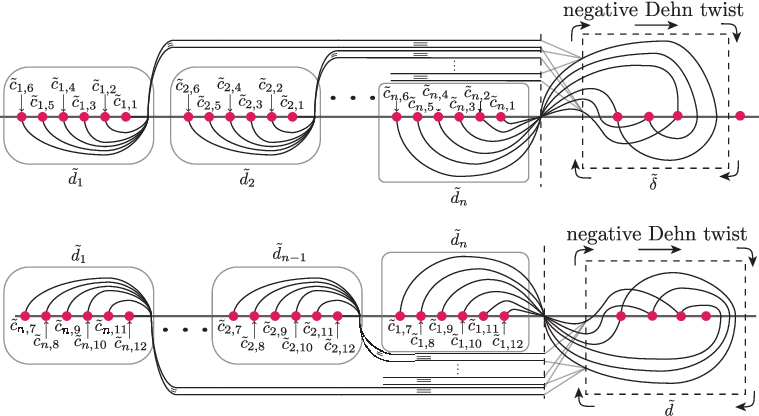}
\caption{Vanishing cycles $\tilde{c}_{i,j}$, $\tilde{d}_i$, $\tilde{\delta}$ and $\tilde{d}$ after the quotient by $\iota$.}
\label{fig:vc}
\end{figure}


The contents of this paper are as follows.
In Section~\ref{sec:back}, we provide some basic properties of Lefschetz fibrations, periodic maps, cyclic quotient singularities of complex surfaces, and their resolutions,
called the \textit{Hirzebruch–Jung resolution} (or simply \textit{HJ-resolution}).

For a periodic action $\sigma$ on $\Sigma_g$ of order $n$ and the $2\pi/n$-rotation action $\tau$ on $S^2$ with two fixed points,
we have the quotient space $\Sigma_g \times S^2/ \sigma\times \tau$.
It has cyclic quotient singularities; thus, by taking the HJ-resolution, we obtain the map $\phi: X \arrowr S^2$.
It is shown that the total space $X$ is diffeomorphic to a blow-up of the product space $\Sigma_h \times S^2$,
where $\Sigma_h\cong \Sigma_g/\sigma$.

In the map $\phi$, there are two singular fibers resulting from the two fixed points on $S^2$.
From a local perspective, we have the map $\phi_i:X_i \arrowr D^2$, with one singular fiber.
In Sections~\ref{sec:splitgen1} and~\ref{sec:splitgen2}, we construct a local chart for each $X_i$ and define $\phi_i$ on each chart.
We also describe a perturbation of each $\phi_i$ so that it becomes a Lefschetz fibration.
Similar to \cite{M3}, all critical points lie in one chart; therefore, by analyzing the movement in this main chart around the singular fibers,
we can find the corresponding vanishing cycles.

\section{Backgrounds}\label{sec:back}

\subsection{Periodic actions on Riemann surfaces} \label{sec:period}
Let $\Sigma_g$ be a closed oriented surface of genus $g$.
A simple example of an element in the mapping class group is a \textit{periodic map}; that is, an orientation-preserving map
$\sigma:\Sigma_g \rightarrow \Sigma_g$ such that $\sigma^n=\text{id}$ for some positive integer $n$.
Such an integer $n$ is called the \textit{period} of $\sigma$.
Two periodic maps $\sigma$ and $\sigma'$ on $\Sigma_g$ are called \textit{conjugate} 
if there exists an orientation-preserving map $\psi$ on $\Sigma_g$ such that $\sigma'=\psi \circ \sigma \circ \psi^{-1}$.
In this section, we review the classification of periodic maps on $\Sigma_g$ up to conjugation, 
as described by Nielsen~\cite{Ni}.

A point $p \in \Sigma_g$ is called a \textit{multiple point} if there exists an integer $m$ with $0<m<n$ such that 
$m$ is the smallest positive integer satisfying $\sigma^m(p)=p$.
Note that the set of multiple points is discrete; hence, the set $M_\sigma$ of multiple points of $\sigma$ is finite.
Since the restriction \[\sigma: \Sigma_g-M_\sigma \rightarrow \Sigma_g -M_\sigma\] is a free action, the quotient map 
\[\pi_\sigma: \Sigma_g \rightarrow \Sigma_g/\sigma\] induces a covering map
\[\pi_\sigma: \Sigma_g -M_\sigma \rightarrow \Sigma_g/\sigma - B_\sigma,\] where $B_\sigma=\pi_\sigma(M_\sigma)$.
Each point in $B_\sigma$ is called the \textit{branch point} of $\pi_\sigma$.
If we fix a base point $x\in \Sigma_g/\sigma -B_\sigma$ and a lift $\tilde{x}\in \pi_\sigma^{-1}(x)$, then $\pi_\sigma$ induces a homomorphism
\[\omega_\sigma: \pi_1(\Sigma_g/\sigma -B_\sigma,x) \arrowr \Z_n\] defined as follows:
for any class $\left[\alpha\right]\in \pi_1(\Sigma_g/\sigma -B_\sigma,x)$, let $\tilde{\alpha}$ be the lift of $\alpha$ via $\pi_\sigma$ starting at $\tilde{x}$.
Then the endpoint $\tilde{\alpha}(1)\in \pi_\sigma^{-1}(x)$ satisfies $\tilde{\alpha}(1)=\sigma^k(\tilde{x})$ for a unique integer $k$ ($0\le k \le n-1$). 
We define $\omega_\sigma(\left[\alpha\right]):=k \pmod n$.
Since $\omega_\sigma: \pi_1(\Sigma_g/\sigma -B_\sigma,x) \arrowr \Z_n$ is a homomorphism into the abelian group $\Z_n$, it induces a homomorphism
\[\omega_\sigma: H_1(\Sigma_g/\sigma -B_\sigma) \arrowr \Z_n.\]

For each point $q_i$ in $B_\sigma:=\{q_1,\dots,q_b\}$, where $b=| B_\sigma |$, 
we assign a loop $\alpha_i$ bounding a disk $D_i$ in $\Sigma_g/\sigma$ such that $D_i \cap B_\sigma=\{q_i\}$, oriented in the counter-clockwise direction.
The following proposition states that $n,b, \{\omega_\sigma(\alpha_1),\dots,\omega_\sigma(\alpha_b)\}$ are essential data for classifying $\sigma$.

\begin{prop}[\cite{Ni}]
Two periodic maps $\sigma$ and $\sigma'$ on $\Sigma_g$ are conjugate if and only if the following conditions are satisfied:
\begin{enumerate}[(i)]
\item The period of $\sigma$ and $\sigma'$ are equal,
\item $| B_\sigma |= | B_{\sigma'} |=b$, and
\item After reindexing the elements of $B_{\sigma'}$, 
$\omega_\sigma(\alpha_i)=\omega_{\sigma'}(\alpha_i)$ for each $1\le i \le b$.
\end{enumerate}
\end{prop}

Let $\theta_i=\omega_\sigma(\alpha_i)$. 
The ordered set $(\theta_1,\dots,\theta_b)$ is called the \textit{valency} of $\sigma$.
By the above Proposition, the data $\left[g,n;\theta_1,\dots,\theta_b\right]$ determines a periodic map up to conjugation.
Conversely, we can construct a periodic map from the data $\left[g,n;\theta_1,\dots,\theta_b\right]$ 
using the existence theorem by Hurwitz~\cite{Hu}.

\begin{prop} \label{prop:condition}
There exists a periodic map $\sigma$ on $\Sigma_g$ with period $n$ and valency $(\theta_1,\dots,\theta_b)$ 
if and only if the following conditions are satisfied:
\begin{enumerate}[(i)]
\item $\theta_1+\cdots+\theta_b \equiv 0 \mod n$,
\item Let $n_i=\gcd\{n,\theta_i\}$. Then there exists a non-negative integer $h$ such that
\begin{equation} \label{eq:genus}
2g-2=n(2h-2)+\sum_{i=1}^b (n-n_i),
\end{equation}

\item If $h=0$ in \eqref{eq:genus}, then $\gcd\{\theta_1,\dots,\theta_b\} \equiv 1 \pmod n$.
\end{enumerate}
\end{prop}

Note that the condition (i) follows from the fact that $\omega_\sigma$ is a homomorphism and the sum $\alpha_1+\cdots+\alpha_b$ is null-homologous. 
Condition (ii) follows from the Riemann-Hurwitz formula, which implies $\Sigma_g/\sigma \cong \Sigma_h$.
Condition (iii) follows from the surjectivity of $\omega_\sigma$. 
Another remark is that $n_i$ represents the number of distinct points in the fiber $\pi_\sigma^{-1}(q_i)$. 
Thus, if $n_i=1$, then $\pi_\sigma^{-1}(q_i)$ is a single fixed point of $\sigma$.

In the following, we use the addition notation introduced by~\cite{AI}: 
\[\sigma=(n, \theta_1/n+\cdots+\theta_b/n)=(n,r_1/n_1+\cdots+r_b/n_b),\] 
where $r_i=\theta_i/\gcd\{\theta_i,n\}$ and $n_i=n/\gcd\{\theta_i,n\}$.
Note that the rational number $r_i/n_i$ implies that 
$\sigma^{n/n_i}$ acts by a rotation of angle $2\pi r_i/n_i$ locally at each point of $\pi_\sigma^{-1}(q_i)$.

\begin{exam}
Consider the periodic action $\sigma$ of type $(3,1/3+1/3+1/3)$ on $\Sigma_4$.
Such a map exists by the Hurwitz theorem, but it also can be geometrically realized:
consider the 16-gon presentation of $\Sigma_4$, and define the action $\sigma$ as follows.

\begin{figure}[h]
\centering
\includegraphics{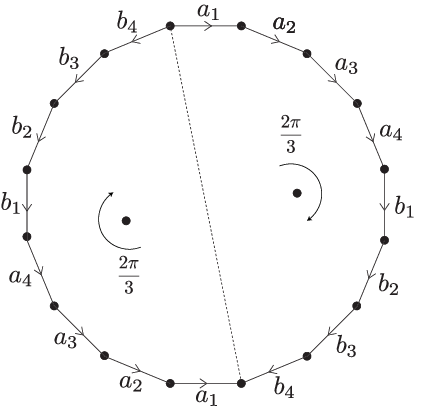}
\caption{Order 3 action on the genus 4 surface $\Sigma_4$}
\label{fig:sigma}
\end{figure}

Then we can easily see that there are three fixed points, and that locally around each fixed point, $\sigma$ acts as a $2\pi/3$-rotation.
\end{exam}

\subsection{Cyclic quotient singularities and Hirzebruch-Jung resolutions} \label{sec:cqsing}
In Section~\ref{sec:period}, we reviewed the properties of periodic maps.
In particular, some power of periodic map $\sigma$ acts as a rotation near a multiple point.
Although the quotient space $\Sigma_g/\sigma$ is also a Riemann surface, the 4-dimensional analogue may be a singular manifold.
That is, if $\sigma$ and $\sigma'$ are rotation actions on $\C$ fixing the origin, then 
the quotient space $\C^2/(\sigma\times \sigma')$ has a singularity at the origin, called a \textit{cyclic quotient singularity}.
In this section, we review such singularities and their resolutions.
The main reference is ~\cite{BHPV}, Section 3.5.

First, we define the $A_{n,q}$-singularity and construct its Hirzebruch-Jung resolution.
For $0<q<n$ with $\gcd(n,q)=1$, consider a subspace
\[W=\{(w,z_1,z_2)\in \C^3 \mid w^n=z_1 z_2^{n-q}\}.\]
Then $W$ can be regarded as an $n$-fold branched cover of $\C^2$ with branch set $z_1 z_2^{n-q}=0$ (counted with multiplicity),
and it has a singularity at $(0,0,0)$, called an $A_{n,q}$-\textit{singularity}.

The minimal resolution of an $A_{n,q}$-singularity is given by a Hirzebruch-Jung string.
If $n/q=\left[e_1,\dots,e_r\right]$ denotes the negative continued fraction expansion, 
let $U_i$ be a small neighborhood of the zero section $C_i$ in $\mathcal{O}_{\CP^1}(-e_i)$.
Consider the linear plumbing $X$ of all $U_i$, 
and take $C_0, C_{r+1}$ to be smooth curves which intersect $C_1, C_r$ transversely at one point, respectively.

Recall from Section 3.5 in~\cite{BHPV} that $\phi$ is a holomorphic function with divisor $(\phi)=\sum_{r=0}^{r+1}n_i C_i$ if and only if
$(\phi)\cdot C_k=0$ for each $k=1,\dots,r$. Hence we have the recursive relation

\begin{equation}\label{eq:rec}
n_{k-1}-e_k n_k +n_{k+1}=0.
\end{equation} 

Let $\mu_k$ and $\nu_k$ be the solutions of $n_k$ in~\eqref{eq:rec} with the initial data $\mu_0=0,\mu_1=1$, and $\nu_0=1,\nu_1=1$. 
Define two holomorphic functions $g$ and $h$ by 
\begin{align}\label{eq:holo}
(g)=\sum_{i=0}^{r+1}\mu_i C_i, \hskip 0.5em (h)=\sum_{i=0}^{r+1}\nu_i C_i
\end{align}

Also, define $f$ by
\[(f)=\sum_{i=0}^{r+1}\lambda_i C_i\]
with $\lambda_{0}=n, \lambda_1=q$, which implies $\lambda_{r+1}=0$.
Then, by the induction, we have the following relation between $\mu_k, \nu_k,$ and $\lambda_k$ for each $k=0,\dots, r+1$:

\begin{equation}
\lambda_k +(n-q)\mu_k=n\nu_k.
\end{equation}

This implies

\begin{equation}
(f\cdot g^{n-q})=(f)+(n-q)(g)=n(h)=(h^n),
\end{equation}

and hence $f\cdot g^{n-q}/h^n$ is a function in $\Gamma(\mathcal{O}_X^*)$, which can be absorbed into $f$.
Therefore, we have $h^n=f\cdot g^{n-q}$.
Define a map $X \arrowr W$ as $(w,z_1,z_2)=(h,f,g)$.

\begin{prop}[\cite{BHPV}]
The map $\rho:X\arrowr W$ defined above is the minimal resolution.
\end{prop}

Let $\gamma=\pi\circ \rho: X \arrowr Z=\{(z_1,z_2)\in \C^2 \mid |z_1|<1, |z_2|<1\}$, where $\pi:W\arrowr Z$ is the projection map.
If we write \[X^*=X-\cup_{i=0}^{r+1}C_i, ~W^*= W-\{z_1z_2=0\},\]
then \[\gamma: X^*\arrowr W^*\] is the covering map of degree $n$.
Let $\delta_i\in \pi_1(X^*)$ be a small positively oriented loop around $C_i$,
and let $w_1=(1,0), w_2=(0,1)$ be the positively oriented loop around the $z_i$-axis in $\Z^2 \cong \pi_1(W^*)$.
Then $\delta_0,\delta_1$ map to $(\lambda_0,\mu_0)=(n,0), (\lambda_1,\mu_1)=(q,1)$, respectively.
Since $(n,0),(q,1)$ generates a subgroup of index $n$, it follows that $\gamma_* \pi_1 (X^*)<\Z^2$ is that subgroup.

We now return to the cyclic quotient singularity.
Consider the $\Z_n$-action on $\C^2$ defined by 
\begin{equation}
k\cdot (z_1,z_2)=(e^{\frac{2\pi i q_1 k}{n}}z_1, e^{\frac{2\pi i q_2 k}{n}}z_2) \hskip 1em (k\in\Z_n)
\end{equation}
where $0<q_i<n$.
If $\gcd\{n,q_1,q_2\}=c\neq 1$, then the action has order $n/c$, i.e., it is a $\Z_{n/c}$-action. 
Hence we may assume $\gcd\{n,q_1,q_2\}=1$.

Write $q_i/n=p_i/n_i$ with $\gcd\{p_i, n_i\}=1$, and let $m=\gcd\{n_1, n_2\}$.
Let $p'_i$ be an integer such that $p_i p'_i \equiv 1 \pmod m$, with $0<p'_i<m$,
and let $q$ be an integer such that $q\equiv p_1 p'_2 \pmod m$, with $0<q<m$.

\begin{prop}[\cite{BHPV}]
The image of $(0,0)\in \C^2$ in the quotient $\C^2/\Z_n$ is an $A_{m,q}$-singularity.
\end{prop}

\begin{proof}
Let $\gamma_0:\C^2 \arrowr \C^2/\Z_n$ be the quotient map by the $\Z_n$-action defined above.
Consider a $\Z_{n_i}$-action on $\C$ defined by the standard $2\pi i/n_i$-rotation fixing the origin.
Note that the quotient map \[\gamma:\C^2 \arrowr \C^2/(\Z_{n_1}\times \Z_{n_2})\]
can be identified with the covering $z_1=u_1^{n_1}, z_2=u_2^{n_2}$ of $\C^2$ branched along $\{z_1z_2=0\}$.

Now $\Z_n$ can be embedded in $\Z_{n_1}\times \Z_{n_2}$ as the subgroup generated by $(p_1, p_2)\in \Z_{n_1}\times \Z_{n_2}$. 
Hence, there is a factorization $\gamma=\delta \circ \gamma_0$, where
\[\delta: \C^2/\Z_n \arrowr \C^2/(\Z_{n_1}\times \Z_{n_2}).\]
Our goal is to find the subgroup \[\Delta\subset \Z\times \Z=\pi_1((\C^2)^*)\] corresponding to $\delta$, where $(\C^2)^*=\C^2-\{z_1z_2=0\}$.

Since the loops in $(\C^2)^*$ corresponding to $(n_1, 0), (0,n_2)$, and $(p_1, p_2)$ lift to loops under $\delta$, 
$\Delta$ is generated by $(n_1, 0), (0,n_2)$, and $(p_1, p_2)$.
Since $\gcd\{n_1,n_2p_1\}=m$, there exist integers $a$ and $b$ such that
\[a(n_1,0)+b(n_2p_1, n_2p_2)=(m,bn_2 p_2)=(m,0)+bp_2(0,n_2),\]
showing that $(m,0)\in \Delta$. 
Similarly, we have $(0,m)\in \Delta$.
Moreover, since $p'_2(p_1,p_2)\equiv (q,1) \mod (m,m)$ by assumption, this implies that $(q,1)\in \Delta$.
Conversely, the subgroup of $\Z\times \Z$ generated by $(m,0), (q,1)$ has index $m$, proving the claim.
\end{proof}

\begin{exam} \label{ex:Anq}
Let the $\Z_n$-action on $\C^2$ be given by 
\begin{equation}
k\cdot (z_1,z_2)=(e^{\frac{2\pi i p k}{n}}z_1, e^{\frac{2\pi i k}{n}}z_2) \hskip 1em (k\in\Z_n).
\end{equation}

This is the special case of the action described above, with $q_2=p_2=1$.
Then $\gcd\{n_1, n\}=n_1=m$, and since $p_2'=1$, we have $q=p$.
Thus, in this case, the quotient $\C^2/\Z_n$ has an $A_{m,p}$-singularity. 


Similarly, consider a $\Z_n$-action given by 
\begin{equation}
k\cdot (z_1,z_2)=(e^{\frac{2\pi i p k }{n}}z_1, e^{\frac{2\pi i (n-1) k}{n}}z_2) \hskip 1em (k\in\Z_n).
\end{equation}

Then $n-1\equiv m-1 \pmod m$, and $q \equiv p(m-1)\equiv m-p \pmod m$. 
So the quotient has an $A_{m,m-p}$-singularity.
\end{exam}

\subsection{Construction and Remarks on the total spaces}
Let $\sigma$ be a periodic action on $\Sigma_g$ with the data 
\[\left(n,\frac{\theta_1}{n}+\cdots+\frac{\theta_b}{n}\right)=\left(n, \frac{r_1}{n_1}+\cdots +\frac{r_b}{n_b}\right),\]
and let $h$ be the integer from~\eqref{eq:genus}, so that $\Sigma_g/\sigma \cong \Sigma_h$.
Note that $\sigma^{n/n_i}$ acts as a rotation by $2r_i \pi /n_i$ near each multiple point.
Let $q_i$ denote the branch point of \[\Sigma_g \arrowr \Sigma_g/\sigma\cong\Sigma_h\] corresponding to $r_i/n_i$.

Note that $\sigma^{-1}=\sigma^{n-1}$ is the periodic map with the data
\[\left(n,\frac{n-\theta_1}{n}+\cdots+\frac{n-\theta_b}{n}\right)=\left(n, \frac{n_1-r_1}{n_1}+\cdots +\frac{n_b-r_b}{n_b}\right),\]
and the quotient space $\Sigma_g/\sigma^{-1}$ is homeomorphic to $\Sigma_h$.
For our setting, we use the action under $\sigma^{-1}$ instead of $\sigma$.

Let $\tau$ be the rotation action by $2\pi/n$ on $D^2$ fixing the origin.
Consider the quotient space \[(\Sigma_g \times D^2)/(\sigma^{-1}\times \tau).\]
Then it is a singular space, and the singular points are $q_i\times 0$.
Moreover, it has a cyclic quotient singularity at each $q_i \times 0$, which is of type $A_{n_i, n_i-r_i}$ by Example~\ref{ex:Anq}.
Similarly, if $\tau'$ is the rotation action by $2\pi (n-1)/n$ on $D^2$ fixing the origin,
then the quotient space $(\Sigma_g \times D^2)/(\sigma^{-1}\times \tau')$ has $A_{n_i, r_i}$-singularity at $q_i \times 0$.

We define a projection map \[\pi_2: (\Sigma_g \times D^2)/(\sigma^{-1}\times \tau) \arrowr D^2/\tau \cong D^2,\] 
and take a HJ-resolution \[\rho: X \arrowr (\Sigma_g \times D^2)/(\sigma^{-1}\times \tau).\]  
Then the composition \[\phi=\pi_2 \circ \rho: X \arrowr D^2\] is a smooth fibration over $D^2$ such that 
\[\phi: X-\phi^{-1}(0) \arrowr D^2-0\] is a smooth fiber bundle with a fiber $\Sigma_g$, 
and $\phi^{-1}(0)$ is the singular fiber of $\phi$.
Then the monodromy of the smooth fibration $\phi: X \arrowr D^2$ is given by $\sigma$.

%

We can also construct the fibration over $S^2$ on a closed manifold.
Let $\tau$ be the $2\pi /n$-rotation action on $S^2$ with two fixed points $\{S,N\}$ ($S$ and $N$ denote the south and north poles of the sphere, respectively).
Then this action is of type $(n,1/n+(n-1)/n)$.
Similarly, consider the quotient space \[(\Sigma_g \times S^2)/(\sigma^{-1} \times \tau)\] and the projection map 
\[\pi_2: (\Sigma_g \times S^2)/(\sigma^{-1} \times \tau) \arrowr S^2/\tau \cong S^2.\]
It has $2b$ cyclic quotient singularities at $q_i\times N$ and $q_i\times S$ of type $A_{n_i, n_i-r_i}$ and $A_{n_i, r_i}$, respectively.
Take a HJ-resolution \[\rho: X \arrowr (\Sigma_g \times S^2)/(\sigma^{-1} \times \tau),\] 
and then we obtain a smooth fibration \[\phi:=\pi_2 \circ \rho: X \arrowr S^2.\]
Its generic fiber is homeomorphic to $\Sigma_g$ and $\phi$ has two singular fibers $\phi^{-1}(S)$ and $\phi^{-1}(N)$.

\begin{rmk}
In fact, this fibration $\phi: X \arrowr S^2$ is a special case of a standard isotrivial fibration, studied in complex geometry:
the desingularization $S$ of $T=(C_1\times C_2)/G$, where $C_i$ are smooth projective curves and $G$ acts on $C_1\times C_2$ diagonally such that $G$ acts faithfully on each component $C_i$.
See \cite{Se}. For a topological perspective, see \cite{P}. 
\end{rmk}

Now we compute the multiplicities and the self intersection numbers of each component in the singular fibers.

\begin{exam} \label{ex:sigma}
Consider $\sigma=(3,1/3+1/3+1/3)$ on $\Sigma_4$. 
Then $(\Sigma_4 \times S^2)/(\sigma^{-1}\times \tau)$ has $6$ cyclic quotient singularities, 
$3$ of which are of type $A_{3,2}$ and the other $3$ of which are of type $A_{3,1}$.

Remark from Section~\ref{sec:cqsing} that the resolution of a singularity of type $A_{3,1}$ consists of a single sphere of self-intersection $-3$, and the resolution of a singularity of type $A_{3,2}$ consists of two spheres of self-intersection $-2$.

Hence, the two singular fibers $\phi^{-1}(N)$ and $\phi^{-1}(S)$ of the fibration \[\phi:X \arrowr S^2 \] consist of
\begin{align*}
\phi^{-1}(N)&=3 T_1+ \sum_{i=1}^3 (d_{i,1} D_{i,1}+d_{i,2} D_{i,2}) \\
\phi^{-1}(S)&=3 T_2+ \sum_{i=1}^3 e_{i} E_{i}.
\end{align*}
Here, $T_i$ is the proper transform of $\pi_2^{-1}(N) , \pi_2^{-1}(S)$ under $\rho$, so that $T_i$ is homeomorphic to a torus.
$D_{i,j}$ and $E_i$ are the spheres coming from the HJ-resolution, so that $[D_{i,j}]^2=-2$ and $[E_i]^2=-3$.

To figure out the multiplicities $d_{i,j}$ and $e_i$, we evaluate each sphere $D_{i,j}$ and $E_i$ with $\phi^{-1}(N)$ and $\phi^{-1}(S)$, respectively.
Since \[0=\phi^{-1}(N)\cdot D_{i,2}=-2 d_{i,2}+ d_{i,1}\] and \[0=\phi^{-1}(N)\cdot D_{i,1}=d_{i,2}-2 d_{i,1}+3,\]
we obtain $d_{i,1}=2$ and $d_{i,2}=1$.
Using the same method, we obtain $e_{i}=1$.

Moreover, evaluating $T_1$ with $\phi^{-1}(N)$, we have
\[0=3[T_1]^2+d_{1,1}+d_{2,1}+d_{3,1}=3[T_1]^2+6,\]
so that $[T_1]^2=-2$. This integer is the negative of the sum $2/3+2/3+2/3$, which is the type of $\sigma^{-1}$.
Similarly, we can deduce $[T_2]^2=-1$.

Figure~\ref{fig:sigmafiber} shows the configuration of each singular fiber $\phi^{-1}(N)$ and $\phi^{-1}(S)$, whose monodromies are $\sigma$ and $\sigma^{-1}=\sigma^2$, respectively.

\begin{figure}[h]
\centering
\includegraphics{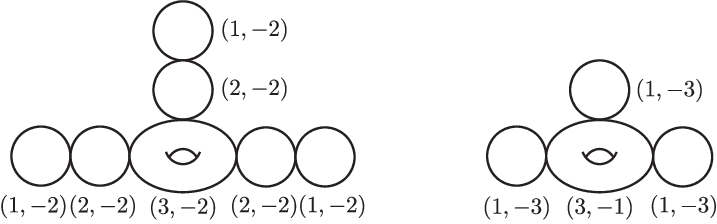}
\caption{Configuration of singular fibers whose monodromies are $\sigma$(left), $\sigma^{-1}$(right).}
\label{fig:sigmafiber}
\end{figure}
\end{exam}

\begin{notation}
Before describing the statement regarding singular fibers, we introduce some notation.
\begin{enumerate}
\item (Two fibrations) 
\begin{align*}
&\pi_1:(\Sigma_g \times S^2 )/(\sigma^{-1} \times \tau) \arrowr \Sigma_g/\sigma^{-1} \cong \Sigma_h, \\ 
&\pi_2: (\Sigma_g \times S^2)/(\sigma^{-1} \times \tau)\arrowr S^2/\tau\cong S^2. \\ 
&\phi=\pi_2 \circ \rho: X \arrowr S^2, \text{ and} \\ 
&\phi'= \pi_1 \circ \rho: X \arrowr \Sigma_h.
\end{align*}

\item (Negative continued fraction) 
\begin{align*}
&n_i/(n_i-r_i)=\left[a_{i,1},\dots, a_{i,k_i}\right], \text{ and} \\ 
&n_i/r_i=\left[b_{i,1},\dots, b_{i,k'_i}\right]
\end{align*}
are negative continued fractions of $n_i/(n_i-r_i)$ and $n_i/r_i$, where $a_{i,j}, b_{i,j}\ge 2$. 

\item (Singular fibers of $\phi$)
Let $F_1=(\pi_2)^{-1}(N)$, $F_2=(\pi_2)^{-1}(S)$, and let $\tilde{F}_i$ be the proper transform of $F_i$ under $\rho$. 
Also define $F'_1$ and $F'_2$ by
\begin{align}
&F_1'=\phi^{-1}(N)=n\tilde{F}_1+\sum_{i,j} d_{i,j} D_{i,j}, \\
&F_2'=\phi^{-1}(S)=n\tilde{F}_2+\sum_{i,j} e_{i,j} E_{i,j},
\end{align}
where $d_{i,j}, e_{i,j}$ are positive integers, and $D_{i,j}, E_{i,j}$ are spheres of self-intersection $-a_{i,j}$ and $-b_{i,j}$, respectively, 
such that $D_{i,1}$ and $E_{i,1}$ intersect with $\tilde{F}_1$ and $\tilde{F}_2$ transversely at a point, respectively.

\item (Singular fibers of $\phi'$)
Let $G_i=(\pi_1)^{-1}(q_i)$, and let $\tilde{G}_i$ be the proper transform of $G_i$ under $\rho$. 
Also let
\begin{equation} \label{eq:S^2_fibers}
G_i'=(\phi')^{-1}(q_i)={n_i}\tilde{G}_i+\sum_{j=1}^{k_i} d'_{i,j} D_{i,j}+\sum_{j=1}^{k_i'} e'_{i,j} E_{i,j},
\end{equation}
so that $D_{i,k_i}$ and $E_{i,k'_i}$ intersect with $G_i$ for each $i=1,\dots,b$.
\end{enumerate}
\end{notation}

The following theorems generalize the process of the calculation described in Example~\ref{ex:sigma}.

\begin{prop}\label{prop:square1}
\begin{enumerate}[(i)]
\item $d_{i,j}$ is the solution of the system of linear equations
\begin{align}\label{eq:mult1}
\begin{bmatrix}
a_{i,1} & -1 & 0 & 0  & \cdots \\ 
-1 & a_{i,2} & -1 & 0 & \cdots \\ 
0 &-1 & a_{i,3} & -1 & \cdots \\ 
\vdots & \vdots &\vdots & \ddots  & \vdots \\ 
0 & \cdots & 0 & -1 & a_{i, k_{i}} \end{bmatrix}
\begin{bmatrix}
 d_{i,1} \\ d_{i,2} \\ d_{i,3} \\ \vdots \\ d_{i,k_i}
\end{bmatrix}
=\begin{bmatrix}
n \\ 0 \\ 0 \\ \vdots \\ 0
\end{bmatrix}
\end{align}
for each $i=1,\dots,b$. 
For $e_{i,j}$, it is the solution of the system of linear equations described in \eqref{eq:mult1}, replacing $a_{i,j}$ by $b_{i,j}$ and $d_{i,j}$ by $e_{i,j}$.

\item $d'_{i,j}$ is the solution of the system of linear equations
\begin{align}\label{eq:mult2}
\begin{bmatrix}
a_{i,1} & -1 & 0 & 0  & \cdots \\ 
-1 & a_{i,2} & -1 & 0 & \cdots \\ 
0 &-1 & a_{i,3} & -1 & \cdots \\ 
\vdots & \vdots &\vdots & \ddots  & \\ 
0 & \cdots & 0 & -1 & a_{i, k_{i}} \end{bmatrix}
\begin{bmatrix}
 d'_{i,1} \\ d'_{i,2} \\ d'_{i,3} \\ \vdots \\ d'_{i,k_i}
\end{bmatrix}
=\begin{bmatrix}
0 \\ 0 \\ 0 \\ \vdots \\ n_i
\end{bmatrix}
\end{align} 
for each $i$. 
For $e'_{i,j}$, it is the solution of the system of linear equations described in \eqref{eq:mult2}, replacing $a_{i,j}$ by $b_{i,j}$ and $d'_{i,j}$ by $e'_{i,j}$.
\end{enumerate}
\end{prop}

\begin{proof}
Fix $i$. Evaluating the intersection of $F'_1$ with $D_{i,j}$ ($j=1,\dots, k_i$) yields the system of $k_i$ linear equations
\begin{equation} \label{eq:recursion}
\begin{cases}
n-a_{i,1}d_{i,1}+d_{i,2}=0&  \\
d_{i,j}-a_{i,j+1}d_{i,j+1}+d_{i,j+2}=0 & (j=1,\dots,k_i-2) \\
d_{i,k_i -1}-a_{i,k_i} d_{i,k_i}=0. &
\end{cases}
\end{equation}
The first equation comes from the fact that $\tilde{F}$ intersects the sphere $D_{i,1}$ once transversely.
The other systems of equations are just analogous to this process.
\end{proof}

Note that $d_{i,j}, d'_{i,j}, e_{i,j}, e'_{i,j}$ need not be integers when we only consider the system of linear equations.
However, Proposition~\ref{prop:integers} says they are integers, and hence they are actually multiplicities of each component.


\begin{lem} \label{lem:recursive}
Let $a_1,\dots , a_k$ be positive integers.
If a sequence $\{P_n\}$ and $\{Q_n\}$ satisfy the recursive formula
\begin{align} 
P_0&=1, P_1=a_1, \text{ and } P_{j+1}=a_{j+1}P_j-P_{j-1}  (j=1,\dots, k-1)  \label{eq:recursive1}\\
Q_0&=0, Q_1=1, \text{ and } Q_{j+1}=a_{j+1}Q_j-Q_{j-1}  (j=1,\dots, k-1) \label{eq:recursive2},
\end{align}
then \[\frac{P_k}{Q_k}=[a_1,\dots, a_k],\]
where $[a_1,\dots, a_k]$ is the negative continued fraction of $a_1,\dots, a_k$.
\end{lem}

\begin{proof}
Consider $\phi_j(z)=\frac{-1}{z+a_j}$, $g_1(z)=a_1+z$, and $g_j(z)=g_{j-1}(\phi_j(z))$ $(j\ge2)$.
Since the composition with $\phi_j$ is a M\"{o}bius transformation, 
$g_j$ can be written as 
\begin{equation}\label{eq:recursion-Mt}
g_j(z)=\frac{A_j+C_j z}{B_j+D_j z},
\end{equation}
with $A_1=a_1, B_1=1, C_1=1, D_1=0$, and $\frac{A_k}{B_k}=g_k(0)=\left[a_1,\dots, a_k\right]$.
From~\eqref{eq:recursion-Mt}, we obtain the recursive relations of $A_j, B_j, C_j, D_j$ for each $j=1,\dots, k-1$:
\begin{equation}
\begin{cases}
&A_{j+1}=a_{j+1} A_{j}-C_{j},\\
&B_{j+1}=a_{j+1} B_{j}-D_{j},\\
&C_{j}=A_{j-1},\\
&D_{j}=B_{j-1}
\end{cases}
\end{equation}
when we set $A_0=1, B_0=0$.
Substituting $C_j$ and $D_j$ into the recursive formula about $A_j$ and $B_j$, respectively, 
we obtain $A_{j+1}=a_{j+1} A_{j}-A_{j-1}$, and $B_{j+1}=a_{j+1} B_{j}-B_{j-1}$ ($j=2,\dots, k-1$).
This is the same as the recursive formula of $\{P_n\}$ and $\{Q_n\}$,
and hence $A_j=P_j$ and $B_j=Q_j$ for $j=1,\dots, k$.

Therefore, we obtain 
\[\frac{P_k}{Q_k}=\frac{A_k}{B_k}=[a_1,\dots, a_k],\]
which proves the claim.
\end{proof}

Now we prove that $d_{i,j}, d'_{i,j}, e_{i,j}$, and $e'_{i,j}$ are related to the recursive formulas \eqref{eq:recursive1} and \eqref{eq:recursive2}.
Note that the number of exceptional spheres in each resolution of $q_i \times N$ and $q_i \times S$ is given as $k_i$ and $k'_i$, respectively.

\begin{prop} \label{prop:integers}
$d_{i,j}, d'_{i,j}, e_{i,j}$, and $e'_{i,j}$ are integers.
In particular, $d_{i,1}$, $e_{i,1}$, $d'_{i,k_i}$ and $e'_{i,k'_i}$ are given by
\begin{align*}
d_{i,1}&=n\cdot \frac{n_i-r_i}{n_i}, \\
e_{i,1}&=n\cdot \frac{r_i}{n_i}, \\
d'_{i,k_i}&={n_i-r'_i}, \\
e'_{i,k'_i}&={r'_i},
\end{align*}
where $r'_i$ is the integer such that $0< r'_i < n_i$ and $r_i r'_i\equiv 1 \pmod{n_i}$.
\end{prop}

\begin{proof}
Fix $i$. 
Consider \[d_{j-1}:=\frac{1}{n}d_{i,j}\]
for $j=1,\dots, k$.
If we write $a_j:=a_{i,j}$, $k:=k_i$, equations~\eqref{eq:recursion} can be rewritten as
\begin{equation} \label{eq:recursion-simple}
\begin{cases}
d_{1}=a_{1}d_{0}-1&  \\
d_{j+1}=a_{j+1}d_{j}-d_{j-1} & (j=1,\dots,k-2) \\
a_{k} d_{k-1}-d_{k-2}=0 &
\end{cases}
\end{equation}

Since all $d_j$ can be written in terms of $d_0$ and a constant term, 
denote \[d_j=P_j d_0-Q_j \hspace{0.5em}(j=0,\dots,k-1),\] where $P_0=1, Q_0=0$, and $P_1=a_1, Q_1=1$.
Then $P_j$ and $Q_j$ satisfies the recursive formula \eqref{eq:recursive1} and \eqref{eq:recursive2}.
Hence if we show $d_0$ is an integer, then we prove the claim.

The last equation of~\eqref{eq:recursion-simple} implies
\[(a_kP_{k-1}-P_{k-2})d_0=(a_k Q_{k-1}-Q_{k-2}).\]

When we define \[P_k:=a_k P_{k-1}-P_{k-2}, Q_k:= a_k Q_{k-1}-Q_{k-2},\] 
then we have \[d_0=Q_k/P_k=\frac{1}{[a_1,\dots, a_k]}=\frac{1}{[a_{i,1},\dots, a_{i,k}]}=\frac{n_i-r_i}{n_i}\] by Lemma~\ref{lem:recursive}.
Since $d_{i,1}=n d_0$, we have \[d_{i,1}=\frac{n}{n_i}(n_i-r_i),\]
and all $d_{i,j}$ are integers.

The same argument implies that
\[e_{i,1}=n \cdot \frac{1}{\left[b_{i,1},\dots, b_{i,k'_i}\right]}=n\cdot \frac{r_i}{n_i}.\]

For $d'_{i,j}$, set
 \[d'_{j-1}:=\frac{1}{n_i}d'_{i,k-j+1} ~(j=1,\dots,k)\] 
Then we have the recursive formula \[d'_{j+1}=a_{k-j}d'_{j}-d'_{j-1} ~(j=1,\dots, k-1).\]
In this case, we have a coefficient sequence $\{a_{k-j}\}$ of reverse order, not $\{a_j\}$.
By Lemma~\ref{lem:recursive}, we have \[d'_0=1/\left[a_k,\dots, a_1\right].\] 
It is a well-known fact that the reverse continued fraction of $n/q$ is $n/q'$, where $qq'\equiv 1 \mod n$.
Since \[(n_i-r_i)(n_i-r'_i) \equiv 1 \pmod{n_i},\] where $r_i r'_i\equiv 1 \pmod{n_i}$,
we have \[d'_{0}=\frac{n_i-r'_i}{n_i},\] and hence
\[d'_{i,k_i}=n_i-r'_i.\]

Similarly, applying the same argument for $e'_{i,k'_i}$, we have \[e'_{i,k'_i}=\frac{n_i}{\left[b_{i,k'_i},\dots, b_{i,1}\right]}={r'_i}.\]
\end{proof}

\begin{cor}
\begin{enumerate}[(i)]
\item The self-intersection numbers of the proper transforms $\tilde{F}_1$ and $\tilde{F}_2$ are given by
\begin{align*}
&[\tilde{F}_1]^2=-\sum_{i=1}^b\frac{n_i-r_i}{n_i}=:-c,\\ 
&[\tilde{F}_2]^2=-\sum_{i=1}^b \frac{r_i}{n_i}=-b+c, \end{align*} 
so that $[\tilde{F}_1]^2+[\tilde{F}_2]^2=-b$.

\item $[\tilde{G}_i]^2=-1$ for all $i=1,\dots, b$.
\end{enumerate}
\end{cor}

\begin{proof}
Since $F'_1 \cdot \tilde{F}_1=0$, we have \[n [\tilde{F}_1]^2 +\sum_{i=1}^b d_{i,1}=0.\]
By Proposition~\ref{prop:integers}, \[d_{i,1}=n\cdot \frac{n_i-r_i}{n_i},\] and hence
\[[\tilde{F}_1]^2 =-\sum_{i=1}^b \frac{n_i-r_i}{n_i}:=-c.\]
(Note that $c$ is an integer by condition (i) of Proposition~\ref{prop:condition}.)

By the same calculation, we obtain $[\tilde{F}_2]^2$ as
\[[\tilde{F}_2]^2=-\frac{1}{n}\sum_{i=1}^b e_{i,1}=-\sum_{i=1}^b\frac{r_i}{n_i}=-b+c.\]

Consider the case $\tilde{G}_i$. Since $\tilde{G}_i$ transversely intersects once with $D_{i,k_i}$ and $E_{i,k'_i}$, 
the equation $G'_i \cdot \tilde{G}_i=0$ implies that
\[[\tilde{G}_i]^2=\frac{-d'_{i,k_i}-e'_{i,k'_i}}{n_i}=-\frac{n_i-r'_i}{n_i}-\frac{r'_i}{n_i}=-1,\]
proving the claim.
\end{proof}

From the fact $[\tilde{G_i}]^2=-1$, we can prove that the total space of $\phi:X \arrowr S^2$ is diffeomorphic to a ruled surface after some blow-downs.
\begin{thm}\label{thm:total}
The complex surface $X$ can be blown down to $\Sigma_h \times S^2$, where $\Sigma_h$ is homeomorphic to $\Sigma_g / \sigma$.
In particular, $X \cong (\Sigma_h \times S^2) \# (\sum_{i=1}^b (k_i+k'_i)) \overline{\CP^2}$.
\end{thm}

\begin{proof}
Consider the map $\phi': X \arrowr \Sigma_h$.
It is a sphere bundle except $b$ singular fibers $(\phi')^{-1}(q_i)$. 
Each $(\phi')^{-1}(q_i)$ is written as in \eqref{eq:S^2_fibers}, 
so it is a linear chain of $k_i+k'_i+1$ spheres, with self-intersection $-a_{i,1},\dots, -a_{i,k_i}, -1, -b_{i,k'_i},\dots, -b_{i,1}$.

It is a well-known fact that if $n/q=[a_1,\dots, a_k]$ and $n/(n-q)=[b_1,\dots, b_{k'}]$, then \[[a_1,\dots, a_k,1,b_{k'},\dots,b_1]=0.\]
This implies a linear chain of spheres with self-intersection \[-a_1,\dots,-a_k,-1, -b_{k'},\dots,-b_1\] can be blown down to a single sphere of self-intersection 0.
Since $G'_i$ constitutes such a chain, this implies that each singular fiber of the fibration $\phi':X\arrowr \Sigma_h$ 
can be blown down to a sphere of self-intersection $0$ in the sphere fibration $\tilde{\phi'}: Y \arrowr \Sigma_h$.

Remark that in the construction of HJ-resolution written in Section~\ref{sec:cqsing}, 
$z_1=0$ in $W$ corresponds to a section of $\pi_1:(\Sigma_g \times S^2)/(\sigma^{-1}\times \tau) \arrowr \Sigma_h$, 
and $z_2=0$ corresponds to a fiber.
Since the multiplicities $\mu_j$ in \eqref{eq:holo} correspond to $d'_{i,j}$ and $e'_{i,j}$, $d'_{i,1}=e'_{i,1}=\mu_1=1$.
This implies that each fiber $\tilde{\phi'}^{-1}(q_i)$ has a multiplicity $1$, and hence $\tilde{\phi'}$ is a sphere bundle.

Finally, we claim that the intersection form of $Y$ is even.
Since $[\tilde{F}_1]^2+[\tilde{F}_2]^2=-b$ and the number of singular fibers of $\phi'$ is $b$, 
we can choose the last blow-down sphere for each singular fiber of $\phi'$ such that we have two sections of self intersection 0, depicted as Figure~\ref{fig:bd}.

\begin{figure}[h]
\centering
\includegraphics{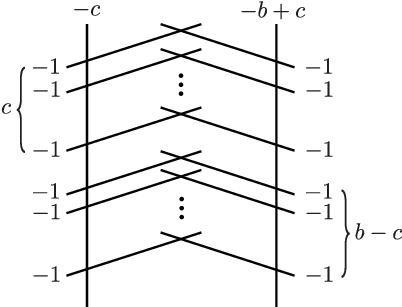}
\caption{The last blow down procedure. Blow-down $c$ spheres at the left and $b-c$ sphere at the right.}
\label{fig:bd}
\end{figure}
 
Therefore, we have a configuration of a sphere bundle over $\Sigma_h$ with two sections of self-intersection 0, which is a spin manifold.
Hence it is diffeomorphic to $S^2 \times \Sigma_h$.
\end{proof}

Note that this property mainly relies on the specific action $(n, 1/n+n-1/n)$ on $S^2$.
In~\cite{P}, it is remarked that the total space $X$ is a ruled surface.

\begin{exam} \label{ex:total}
For the action given in Example~\ref{ex:sigma} (that is, $\sigma=(3,1/3+1/3+1/3)$), there are a total of 9 spheres in $\phi^{-1}(N)$ and $\phi^{-1}(S)$.
So by Theorem~\ref{thm:total}, we have \[X\cong (T^2 \times S^2)\#9\CPbar^2.\]
\end{exam}

\subsection{Main results}
Consider an action \[\sigma_{2n+1}=\left(2n+1, \frac{1}{2n+1}+\frac{n}{2n+1}+\frac{n}{2n+1}\right)\] acting on $\Sigma_{3n+1}$.
Since $\gcd(n,2n+1)=1$, there are three multiple points for this action.
Consider the quotient $(\Sigma_{3n+1} \times S^2) / (\sigma^{-1}_{2n+1} \times \tau_{2n+1})$
and the HJ resolution 
\[\rho_{n}: X_{n} \arrowr (\Sigma_{3n+1} \times S^2) / (\sigma^{-1}_{2n+1} \times \tau_{2n+1}),\] 
where $\tau_{2n+1}$ is the action on $S^2$ of order $2n+1$ with two fixed points.
Composing with the projection map $\pi_2: (\Sigma_{3n+1} \times S^2) / (\sigma^{-1}_{2n+1} \times \tau_{2n+1}) \arrowr S^2$, 
we have the fibration \[\phi_{n}=\pi_2 \circ \rho_n: X_{n} \arrowr S^2\]
with two singular fibers $F^1_{n}=\phi^{-1}_{n}(N), F^2_{n}=\phi^{-1}_{n}(S)$, 
whose monodromies are $\sigma^{-1}_{2n+1}=\sigma^{2n}_{2n+1}$ and $\sigma_{2n+1}$, respectively.

To figure out two singular fibers, consider the negative continued fraction for $\frac{2n+1}{n}$ and $\frac{2n+1}{n+1}$:
\begin{align}
\frac{2n+1}{n}&=[3, \underbrace{2,2, \dots, 2}_{n-1}], \\
\frac{2n+1}{n+1}&=\left[2, n+1\right].
\end{align}
Hence, the singular fiber $F^1_{n}, F^2_{n}$ with the monodromies
\[\sigma^{-1}_{2n+1}=\left(2n+1, \frac{2n}{2n+1}+\frac{n+1}{2n+1}+\frac{n+1}{2n+1}\right)\] and
\[\sigma_{2n+1}=\left(2n+1, \frac{1}{2n+1}+\frac{n}{2n+1}+\frac{n}{2n+1}\right)\] are given as Figure~\ref{fig:sfiber_n}.
Also, since there are $4n+5$ spheres in the two singular fibers, the total space $X_{n}$ is diffeomorphic to $(T^2 \times S^2) \# (4n+5)\CPbar^2$.

\begin{figure}[ht]
\centering
\includegraphics[scale=0.6]{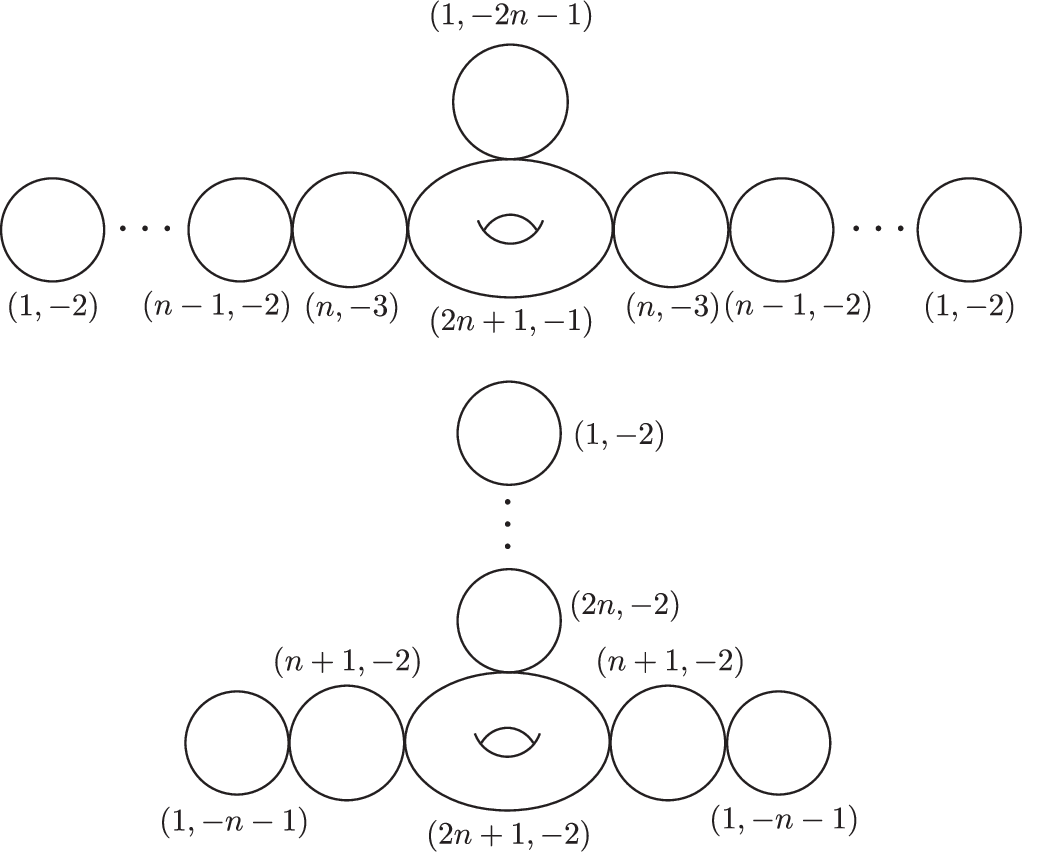}
\caption{Singular fibers $F^1_{n}$ (above), and $F^2_{n}$ (below).}
\label{fig:sfiber_n}
\end{figure}

Restricting the map $\phi_{n}$ to each fibered neighborhood of $F_{n}^1, F_{n}^2$, 
we have two fibrations \[\phi_{n,i}: X_{n,i}\arrowr D^2\] for $i=1,2$ 
whose monodromies are given by $\sigma_{2n+1}^{-1}$, $\sigma_{2n+1}$, respectively.
By perturbing such maps, 
we obtain Lefschetz fibrations whose monodromy factorizations give factorizations of words of $\sigma_{2n+1}^{-1}$ and $\sigma_{2n+1}$, proved in Sections~\ref{sec:splitgen1} and \ref{sec:splitgen2}.

\begin{thm}\label{thm3}
For each $i=1,2$ there is a perturbation $(\phi_{n,i})_\ep :X_{n,i} \arrowr D^2$ such that $(\phi_{n,i})_\ep$ becomes a Lefschetz fibration whose monodromy factorization is given by
\begin{align}
\sigma_{2n+1}^{-1}&=t_{\delta}\prod_{i=1}^n (t_{d_i} t_{d'_i}) \left(\prod_{j=6}^1 \prod_{i=n}^1 t_{c_{i,j}} \right), \quad (i=1) \label{eq:genfact_1}\\
\sigma_{2n+1}&= (t_{d} t_{d'})^2 \left(\prod_{j=12}^7 \prod_{i=n}^1 t_{c_{i,j}} \right) \prod_{i=1}^n (t_{d_i} t_{d'_i}). \quad (i=2) \label{eq:genfact_2}
\end{align}
\end{thm}

Combining two factorizations, we have a monodromy factorization of a genus-$(3n+1)$ Lefschetz fibration over $S^2$ as in Theorem~\ref{thm:main}.

%
%
%

\section{Construction of $X_1$ and its perturbation}\label{sec:splitgen1}
The contents of Sections~\ref{sec:splitgen1} and \ref{sec:splitgen2} are organized as follows.
First, for each singular fiber, we construct the regular neighborhood $X_{n,j}$ of $F^j_n$ by using several local charts. 
On each chart, we define the map $\phi_{n,j}$, which provides the fibration structure.
We show that this map is well-defined with respect to each transition map.
After establishing these constructions, we define a smooth perturbation $(X_{n,j})_\epsilon$ of the total space $X_{n,j}$ 
by introducing smooth perturbations of the transition maps. 
Correspondingly, we also define a perturbation $(\phi_{n,j})_\epsilon$ of $\phi_{n,j}$ on each chart.

Next, in order to show that
\[(\phi_{n,j})_\epsilon : (X_{n,j})_\epsilon \to D^2\]
is a Lefschetz fibration, we observe that it suffices to work on a single chart. By computing derivatives on that chart, we determine the critical points and singular values of the map. Since it is hard to obtain explicit solutions, we instead consider approximate solutions.
Moreover, except for the fiber $(\phi_{n,j})_\epsilon^{-1}(0)$, we can show that all critical points are nondegenerate. 
For the zero-fiber, it has only nodal singularities by the construction of the defining function. 
Therefore, it follows that the map defined above is indeed a Lefschetz fibration.

To find the vanishing cycles of the Lefschetz fibrations,
consider the two projection maps
\[ \Phi : \C_{(x,y)}^2 \to \C_{(X,y)}^2 \quad\text{and}\quad \pi_y : \mathbb{C}_{(X,y)}^2 \to \mathbb{C}_y\]
defined by $\Phi(x,y)=(x^2,y)=:(X,y)$ and $\pi_y(X,y)=y$,
where $\C_{(x,y)}^2, \C_{(X,y)}^2, \C_y$ denote complex spaces together with the corresponding coordinate variables on each space. 
If we simply write $(\phi_{n,j})_\epsilon$ as $\phi_n$,
$\phi_n:\C_{(x,y)}^2 \arrowr \C_t$ is transformed under the map $\Phi$ into 
\[\psi_n: \C_{(X,y)}^2 \arrowr \C_t.\]
Denote \[F_n^t=\phi_n^{-1}(t)\cap U, \quad f_n^t=\psi_n^{-1}(t)\cap \Phi(U), \quad\text{and}\quad \Phi_n^t=\Phi|_{F_n^t}: F_n^t \arrowr f_n^t.\]
$\Phi_n^t$ is the hyperelliptic double branched cover.
Moreover, after some extension \[\hat{\pi}_y: \hat{f}_n^t \arrowr \hat{\C}_y\] of $\pi_y: f_n^t \arrowr \C_y$, 
$\hat{\pi}_y$ is the $(2n+1)$-fold branched cover.
Observe that the two branch points of $\Phi_n^t$ merge into a single point which becomes a critical point of $\phi_n$, 
as $t$ moves to a singular value,
and that the monodromy around such a singular value is given by the Dehn twist along an appropriate path connecting these two points.
After fixing a regular value ${t_O}$ and a Hurwitz system in $\C_t$,
we first determine the path in $\C_y$, 
and we lift it through the two successive branched coverings to obtain the corresponding vanishing cycle.
In this process, the movement of the branch points of $\pi_y$ plays a crucial role.

Finally, the monodromy factorization obtained through the procedure is determined up to fiber automorphism, 
equivalently, up to global conjugation from the viewpoint of words in the mapping class group.
So this global conjugation must be identified to complete the proof.
Since the monodromies of two singular fibers we constructed should be given by $\sigma_{2n+1}^{-1}$ and $\sigma_{2n+1}$, 
we use relations in the mapping class group to obtain expressions for the word factorization of $\sigma_{2n+1}^{-1}$ and $\sigma_{2n+1}$, respectively, which are globally conjugate to those derived geometrically above.

\subsection{Construction of local charts}
We now construct a system of local charts to realize a regular neighborhood of $F^1_{n}$.
Let $X_{n,1}= \phi^{-1}_{n}(D^2_+)$ be the inverse image of the upper half-disk of $S^2$, and
let $\phi_{n,1}: X_{n,1} \arrowr D^2$ denote the restriction of $\phi_{n}$.
The following description is summarized in the following proposition.

\begin{prop} \label{prop:charts1}
There is a system of local charts
\[X_{n,1} \cong U\cup \cup_{i=1,2}(U_{i,1}\cup \cdots U_{i,n}) \cup U_{3,1} \cup U_{3,2},\]
and a map $\phi_{n,1}|_{U_{i,j}}:U_{i,j} \arrowr \C_t$ compatible with the transition maps,
such that it can be perturbed to a map
\[(\phi_{n,1})_\ep: (X_{n,1})_\ep \arrowr \C_t,\]
so that $(\phi_{n,1})_\ep$ is a Lefschetz fibration.
The transition maps and $(\phi_{n,1})_\ep$ on each chart of $(X_{n,1})_\ep$ are described in ~\eqref{eq:charte1} and ~\eqref{eq:phie1}.

Here, $(X_{n,1})_\ep$ is constructed by smoothly perturbing transition maps, so $(X_{n,1})_\ep$ is diffeomorphic to $X_{n,1}$ and $(\phi_{n,1})_\ep$ is a smooth perturbation of $\phi_{n,1}$.
\end{prop}

For a sufficiently small $r_n>0$, we define the open subset 
\[U=\{(x,y)\mid |f(x,y)|=|x^2-y^3-1|<r_n\}\subset \C^2\]
of $\C^2$.
The precise choice of $r_n$ will be specified later.
Define the restriction $\phi_{n,1}|_U: U \arrowr \C$ by 
\[\phi_{n,1}|_U=y^n f^{2n+1}(x,y).\]
Then the zero-fiber $\phi_{n,1}^{-1}(0)$ consists of a punctured torus $T^0$ of multiplicity $2n+1$ and two disks $D_1, D_2$ of multiplicity $n$, each intersecting $T^0$ transversely at a point.
We attach two 2-handles 
\[U_{i,1}=\{(\mathfrak{s}_1, \mathfrak{t}_1)\in \C^2 \mid |\mathfrak{s}_1|<\delta_1, |\mathfrak{t}_1|<\delta_1 \}\] 
for $i=1,2$ to $U$ along the attaching circles $\partial D_i$ with framing $-3$.
Equivalently, using local coordinates $(s,t)$ near $(1,0)\in U$ defined by $s=y$ and $t=f(x,y)$, 
the attaching map for $U_{i,1}$ is given by $(\mathfrak{s}_1,\mathfrak{t}_1)=(t^{-1}, st^3)$.
The map $\phi_{n,1}$ then extends over $U_{i,1}$ as 
\[\phi_{n,1}|_{U_{i,1}}=\mathfrak{s}_1^{n-1}\mathfrak{t}_1^n.\]
Subsequent charts $U_{i,j}=\{(\mathfrak{s}_j, \mathfrak{t}_j) \}$ for $j=2,\dots, n$ are attached to $U_{i,j-1}$ via the transition maps 
\[(\mathfrak{s}_j, \mathfrak{t}_j)=(\mathfrak{t}_{j-1}^{-1}, \mathfrak{s}_{j-1}\mathfrak{t}_{j-1}^2),\] 
so that the map further extends as \[\phi_{n,1}|_{U_{i,j}}=\mathfrak{s}_j^{n-j}\mathfrak{t}_j^{n-j+1}.\]
In particular, on the last chart, we have \[\phi_{n,1}|_{U_{i,n}}=\mathfrak{t}_n.\]

Near the point at infinity in the $U$-chart, introducing the homogeneous $z$-coordinate allows us to rewrite $f(x,y)=x^2 -y^3 -1$ in the projective form:
\begin{align}\label{eq:infinite}
\tilde{f}(x,y,z):=\frac{x^2}{z^2}-\frac{y^3}{z^3}-1.
\end{align}

Recall that as $z\arrowr 0$ along the curve $\tilde{f}(x,y,z)=0$, $(x,y)$ approaches $(1,0)$.
By setting \[u=\frac{z}{x}, v=\frac{y}{x},\] the expression \eqref{eq:infinite} becomes 
\[\tilde{f}(x,y,z)=\frac{1}{u^{2}}-\frac{v^3}{u^3}-1:=g(u,v).\]

We now attach a chart $U_{3,1}=\{(\xi_1, \eta_1)\}$ to $U$ via the relations: 
\begin{align} \label{eq:infchart}
\xi_1=v, \eta_1=\frac{\sqrt[2n+1]{1-(1+u^2)g(u,v)}^n}{v} g(u,v).
\end{align}
Note that the $(2n+1)$-th root in~\eqref{eq:infchart} is well-defined, since $u$ and $g(u,v)$ can be chosen to be arbitrarily small.
We then define \[\phi_{n,1}|_{U_{3,1}}=\xi_1 \eta_1^{2n+1}.\]
Finally, attach a chart $U_{3,2}=\{(\xi_2, \eta_2)\}$ to $U_{3,1}$ with the transition map 
\[(\xi_2, \eta_2)=(\eta_1^{-1}, \xi_1 \eta_1^{2n+1}),\] so that 
\[\phi_{n,1}|_{U_{3,2}}=\eta_2.\]
We thus clearly obtain the identification 
\[X_{n,1} \cong U\cup \cup_{i=1,2}(U_{i,1}\cup \cdots U_{i,n}) \cup U_{3,1} \cup U_{3,2}.\]

Now consider a perturbation $(X_{n,1})_{\ep}$.
On the main chart $U$, the perturbed map is defined as
\[(\phi_{n,1})_{\ep}|_U=f(x,y)(y^n f(x,y)^{2n}-{\ep^{2n}}).\]
From this, the transition maps are modified as follows:
\begin{equation} \label{eq:charte1}
    \begin{aligned}
    U_{i,1}\cap U&: (\mathfrak{s}_1, \mathfrak{t}_1)=(t^{-1}, t(st^2-{\ep^{2}})), \\
    U_{i,j+1} \cap U_{i,j}&:(\mathfrak{s}_{j+1}, \mathfrak{t}_{j+1})
    =(\mathfrak{t}^{-1}_{j}, \mathfrak{t}_{j}(\mathfrak{s}_j \mathfrak{t}_j +(z_n^{j-1}-z_n^{j}){\ep^{2}})), \\
    U_{3,1} \cap U&: (\xi_1,\eta_1)=\left(v, \frac{\sqrt[2n+1]{1-(1+u^2)g(u,v)}^n}{v} g(u,v)\right),\\
    U_{3,2}\cap U_{3,1}&: (\xi_2, \eta_2)=(\eta_1^{-1}, \xi_1 \eta_1(\eta_1^{2n}-{\ep^{2n}})),
    \end{aligned}
\end{equation}
where $i=1,2$ and $j=1,\dots, n-1$.
This allows us to extend $(\phi_{n,1})_{\ep}$ as follows:
\begin{equation} \label{eq:phie1}
    \begin{aligned}
    (\phi_{n,1})_{\ep}|_U&=f(x,y)(y^n f(x,y)^{2n}-{\ep^{2n}})\\
    (\phi_{n,1})_{\ep}|_{U_{i,j}}&=
    \mathfrak{t}_j \prod_{k=j}^{n-1}(\mathfrak{s}_j \mathfrak{t}_j+(z_n^{j-1}-z_n^k){\ep^{2}}), \text{ } (i=1,2, j=1,\dots, n-1)\\
    (\phi_{n,1})_{\ep}|_{U_{i,n}}&=\mathfrak{t}_n, \quad (i=1,2)\\
    (\phi_{n,1})_{\ep}|_{U_{3,1}}&
    =\xi_1 \eta_1(\eta_1^{2n}-{\ep^{2n}})=\xi_1 \eta_1 \prod_{i=0}^{2n-1}(\eta_1 -z_{2n}^i {\ep}),\\
    (\phi_{n,1})_{\ep}|_{U_{3,2}}&=\eta_2.
    \end{aligned}
\end{equation}

\begin{figure}
\centering
\includegraphics[scale=0.5]{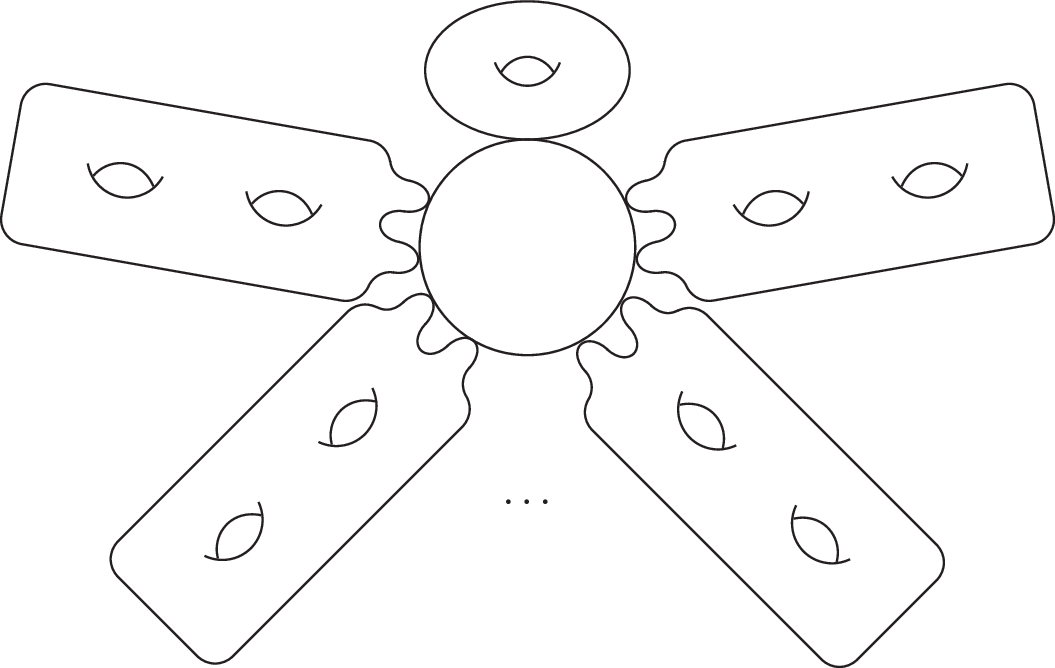}
\caption{Configuration of $((\phi_{n,1})_{\ep_n})^{-1}(0).$}
\label{fig:perturbfiber1}
\end{figure}

\subsection{Critical points and singular values} \label{sec:crit,sing}
By the definition of $(\phi_{n,1})_{\ep}$, critical points occur only in the $U$-chart.
For simplicity, we denote $(\phi_{n,1})_{\ep} |_U$ and $f(x,y)^n$ by $\phi_n$ and $f^n$ in this section, respectively.
Also, put $\ep=\ep_n$, where $\ep_n=(10n)^{-1}$.		 
\begin{prop}
There are $6n$ critical points in $U$ away from the zero-fiber, and all critical points are of the form $(0,y)$.
For $n=3k$, every three critical points $(0,y_{i,j})$, $j=0,1,2$, lie on a single singular fiber over $t_i$, for $i=0,\dots, 2n-1$.
Approximated critical points $(0,y_{i,j})$ and singular values $t_i$ are given by
\begin{align}\label{eq:y_i,j, 1}
y_{i,j}=
\begin{cases} z_6^{2j+1} (1+\frac{c_n}{3} z_{4n} z_{2n}^i {\ep_n}) & (k \text{ odd}) \\ 
z_6^{2j+1} (1+\frac{c_n}{3} z_{2n}^i {\ep_n}) & (k \text{ even})
\end{cases}
\end{align}
and
\begin{align} \label{eq:t_i, 1}
t_{i}=
\begin{cases} - \frac{2nc_n}{2n+1}  z_{4n}z_{2n}^i {\ep_n^{2n+1}} & (k \text{ odd})\\  
-\frac{2nc_n}{2n+1}  z_{2n}^i {\ep_n^{2n+1}}, & (k \text{ even})
\end{cases} 
\end{align}
where \[c_n=\frac1{\sqrt[2n]{2n+1}}.\]

For $n \neq 3k$, every critical point lies on distinct singular fibers, whose approximated expressions for the critical points
$(0,y_i)$ and singular values $t_i$, $i=0,\dots, 6n-1$, are given by
\begin{align}
y_i=
\begin{cases}
z_6^{2j(i)+1}(1+\frac{c_n}{3} z_{12n} z_{6n}^i \ep_n )& (n \text{ odd}) \\
z_6^{2j(i)+1}(1+\frac{c_n}{3} z_{6n}^i \ep_n), & (n \text{ even}) 
\end{cases}
\end{align}
where $j(i)\in\{0,1,2\}$ is the unique solution of the modular equation
\begin{align} 
2nj \equiv \begin{cases} i+m+1 \pmod 3 & \quad \text{if}  \quad n=2m+1\\
 i+m \pmod 3 & \quad \text{if}  \quad n=2m, \end{cases}
\end{align}
and
\begin{align}\label{eq:t_i, 2}
t_{i}=\begin{cases} - \frac{2nc_n}{2n+1}  z_{12n}z_{6n}^i {\ep_n^{2n+1}} & (n \text{ odd})\\ 
 -\frac{2nc_n}{2n+1}  z_{6n}^i {\ep_n^{2n+1}}. & (n \text{ even})\end{cases}
\end{align}
\end{prop}

\begin{figure}[h]
\includegraphics[scale=0.6]{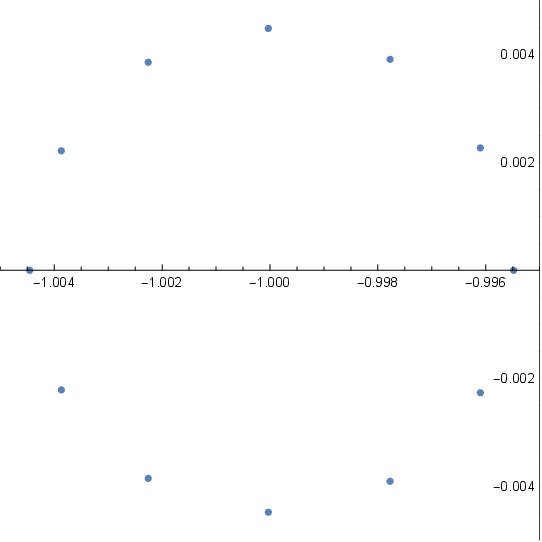}
\includegraphics[scale=0.6]{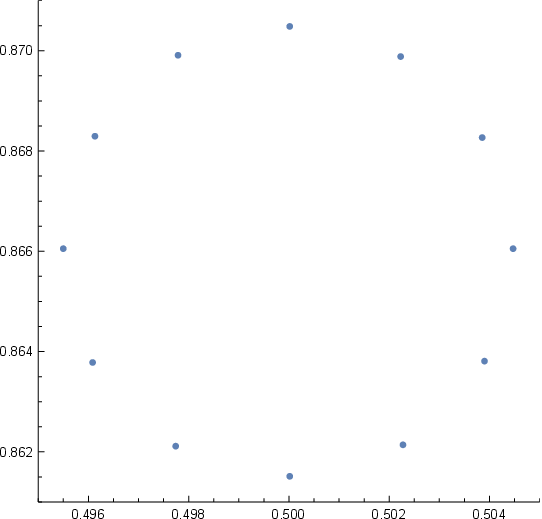}
\includegraphics[scale=0.6]{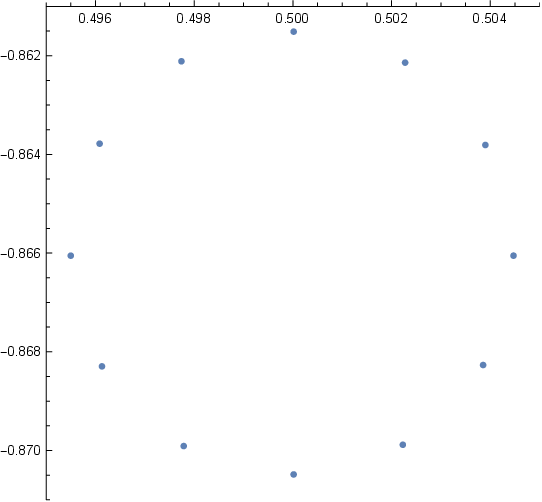}
\caption{Critical points $y_{i,1}$ (top-left), $y_{i,0}$ (top-right), and $y_{i,2}$ (bottom), for $n=6$}
\label{fig:critsing6}
\end{figure}

\begin{figure}[ht]
\includegraphics[scale=0.6]{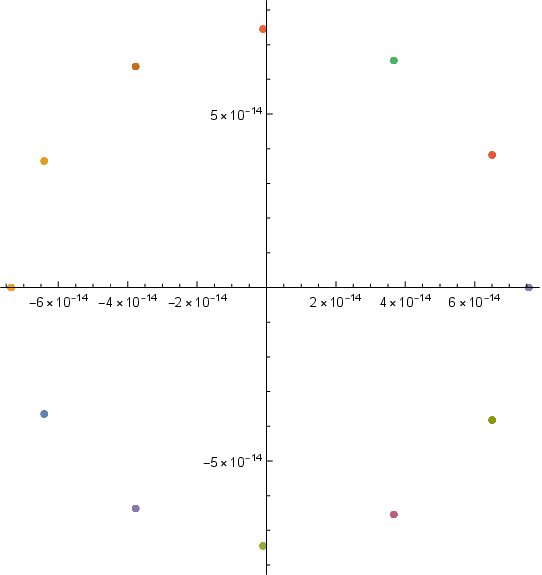}
\includegraphics[scale=0.6]{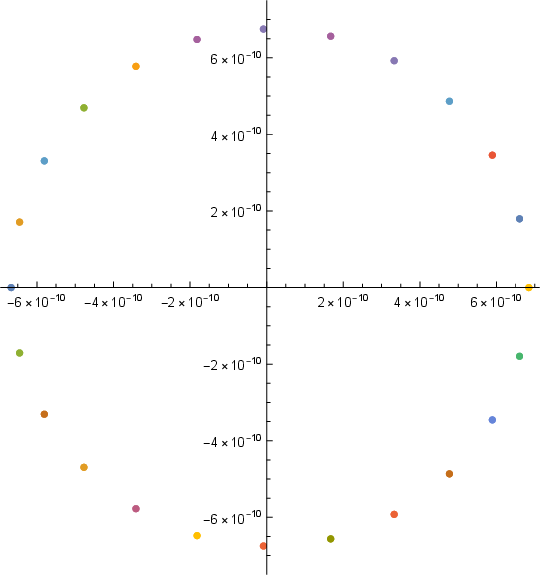}
\caption{Singular values $t_i$ for $n=6$ (left) and $n=4$ (right)}
\label{fig:critsing9}
\end{figure}

\begin{proof} 
To find critical points, we set $\frac{\partial \phi_n}{\partial x}=\frac{\partial \phi_n}{\partial y}=0$. Since
\begin{align*}
\frac{\partial \phi_n}{\partial x} &= 2x((2n+1)y^n f^{2n} -{\ep_n^{2n}}), \\
\frac{\partial \phi_n}{\partial y} &= n y^{n-1} f^{2n+1} + (-3y^2)(2n+1)y^n f^{2n}+3y^2 {\ep_n^{2n}} \\
&=y^{n-1}f^{2n}(nx^2-(7n+3)y^3-n)+3y^2 {\ep_n^{2n}},
\end{align*}
if $x\neq0$, then $y^n f^{2n}={\ep_n^{2n}}/(2n+1)$. Substituting this into $\frac{\partial \phi_n}{\partial y}$, we obtain
\begin{align*}
=&\frac{{\ep_n^{2n}}}{(2n+1)y}(nx^2 -(7n+3)y^3 -n)+3y^2 {\ep_n^{2n}} \\
=&\frac{{\ep_n^{2n}}}{(2n+1)y}(nx^2-(7n+3)y^3 -n +(6n+3)y^3) \\
=&\frac{n{\ep_n^{2n}}}{(2n+1)y}f=0,
\end{align*}
which implies $f=0$, contradicting with the assumption that $(2n+1)y^n f^{2n}={\ep_n^{2n}}\neq 0$.

Thus, all critical points are of the form $(0,y)$, where $y$ is a solution to the equation
\begin{align} \label{eq:criteq'}
y^{n-1}f^{2n}((7n+3)y^3+n)=3y^2 {\ep_n^{2n}},
\end{align}
or equivalently, when $y\neq 0$
\begin{align} \label{eq:criteq}
f^{2n}=\frac{3{\ep_n^{2n}}y^3}{y^n ((7n+3)y^3+n)}=\frac{3{\ep_n^{2n}}(f+1)}{y^n ((7n+3)f+6n+3)}.
\end{align}

Note that the equation \eqref{eq:criteq} has $7n$ solutions, but $n$ of these are located near the origin.
These solutions are discarded because the domain of $U$ is defined such that $|-y^3-1|$ is sufficiently small.
Furthermore, all roots of \eqref{eq:criteq} away from the origin are simple roots.

To find the roots of \eqref{eq:criteq}, we divide our analysis into two cases $n=3k$ and $n\neq 3k$.
First, assume $n=3k$.
We can then write $y^n=y^{3k}=(-f-1)^k=(-1)^k(f+1)^k$.
Thus, we have 
\begin{align} \label{eq:f^2n}
f^{2n}=\frac{3{\ep_n^{2n}}(f+1)}{(-1)^k(f+1)^k ((7n+3)f+6n+3)}.
\end{align}
To solve the equation for $y$, we introduce linear approximations.
In this section, we focus on finding the approximated solutions;
we refer the reader to the appendix for the justification of error control.

Since $|f|$ is sufficiently small, we can approximate \eqref{eq:f^2n} as
\[ f^{2n} \approx (-1)^k \frac{1}{2n+1} {\ep_n^{2n}},\]
so that 
\begin{align} \label{eq:f_i, 1}
f\approx f_i = \begin{cases} c_n \cdot z_{4n}z_{2n}^i {\ep_n}  & (k \text{ odd}),\\  c_n \cdot z_{2n}^i {\ep_n} & (k \text{ even}),\end{cases}
\end{align}
where $z_m=\exp(2\pi i /m)$ and
\[c_n=\sqrt[2n]{1/2n+1}.\]
Since $f=-1-y^3$, using the approximation $(1+x)^{1/3}\approx 1+1/3\cdot x$, 
we have the expression of $y_{i,j}$ and $t_i=\phi_n(0,y_{i,j})$ given in \eqref{eq:y_i,j, 1} and \eqref{eq:t_i, 1}.

Now assume $n\neq 3k$.
Cubing both sides of Equation \eqref{eq:criteq} yields
\begin{align}\label{eq:criteq,1}
f^{6n}=\frac{3^3{\ep_n^{6n}}(f+1)^3}{y^{3n} ((7n+3)f+6n+3)^3}=\frac{3^3{\ep_n^{6n}}(f+1)^3}{(-1)^n(f+1)^n ((7n+3)f+6n+3)^3},
\end{align}
so the solution $f$ of \eqref{eq:criteq,1} is given by
\begin{align}\label{eq:f_i, 2}
f\approx f_i= \begin{cases} c_n\cdot z_{12n} z_{6n}^i {\ep_n} & (n \text{ odd}) \\ c_n \cdot z_{6n}^i {\ep_n} & (n \text{ even}) \end{cases}
\end{align}
for $0\le i \le 6n-1$,
and similarly obtain $y_{i,j}$.
But, by Lemma~\ref{lem:j(i)}, only $y_{i,j(i)}=:y_i$ are the true solution of ~\eqref{eq:criteq}.
Therefore, each critical point $(0,y_i)$ lies on distinct singular fiber over $t_i$.
\end{proof}

\begin{lem}\label{lem:j(i)}
There exists a unique function $j: \{0,\dots, 6n-1 \} \arrowr \{0,1,2 \}$ such that $y_{i,j(i)}$ is a root of \eqref{eq:criteq}.
\end{lem}

\begin{proof}
Assume $n=2m+1$.
Substituting $y_{i,j}$ into \eqref{eq:criteq}, the left-hand side $f^{2n}$ becomes 
\[1/(2n+1){\ep_n^{2n}} \cdot z_6 z_3 ^{i}.\]
Meanwhile, the right hand side can be approximated as
\begin{align*}
&\frac{3 {\ep_n^{2n}} (1+ f_i)}{z_6^n z_3^{nj} (1+1/3 f_i)^n ((7n+3)f_i +(6n+3))} \\
\approx& \frac{{\ep_n^{2n}}}{(2n+1) z_6^{2m+1} z_3^{nj}} \\
=&\frac{1}{2n+1} {\ep_n^{2n}} \cdot z_6^5 z_3^{2m+2nj} \\
=&\frac{1}{2n+1} {\ep_n^{2n}} \cdot z_6 z_3^{2m+2nj+2}.
\end{align*}
Equating this with $f^{2n}$, we obtain the congruence
\[i \equiv 2m+2nj+2 \pmod 3,\]
which simplifies to
\begin{align} \label{eq:mod3eq}
2nj \equiv i+m+1 \pmod 3.
\end{align}
Since $n$ is not a multiple of $3$, there is a unique solution $j(i)$ to \eqref{eq:mod3eq}.
The case $n=2m$ is similar.
\end{proof}

\subsection{Vanishing cycles, $n=3k$ case} \label{sec:vc_3k}
In this subsection, we assume $n=3k$. 
Define $\Phi:\C_{(x,y)}^2 \arrowr \C_{(X,y)}^2$ by 
\[\Phi(x,y)=(x^2,y)=:(X,y).\]
Then $\phi_n$ is transformed into the map $\psi_n: \C^2_{(X,y)} \arrowr \C_t$ defined by
\[\psi_n(X,y):=(X-y^3-1)(y^n(X-y^3-1)^{2n}-{\ep_n^{2n}}),\]
and the restriction $\Phi_n^t: \phi_n^{-1}(t) \arrowr \psi_n^{-1}(t)$ of $\Phi$ is a hyperelliptic branched cover branched along $x=0$.
A generic fiber $\phi_n^{-1}(t)$ is homeomorphic to a once-punctured surface of genus $3n+1$.
Since the number of branch points of $\Phi_n^t$ is $6n+3$, the fiber $\psi_n^{-1}(t)$ is homeomorphic to a disk.

While the fiber $\psi_n^{-1}(t)$ is smooth for $t\neq 0$, 
the fiber $\psi_n^{-1}(t_{i})$ is locally written as $X=y^2$, tangent to the $y$-axis for each singular value $t_{i}$ of $\phi_n$.
Thus, as $t$ approaches $t_{i}$, two branch points of $\Phi$ merge into a point.
Therefore, to find vanishing cycles for $\phi_n$, we first fix a reference fiber $\phi_n^{-1}(t_O)$ and trace the trajectories of the branch points that merge as $t$ moves from $t_O$ to each singular value.

We choose $t_O={\ep_n^{4n}}$ sufficiently close to the origin so that the branch points of $\Phi_n^{t_O}$ are close to those of $\Phi_n^0$.
Note that the branch points of $\Phi_n^0$ are given by the solutions of the following equation
\[(-y^3-1)(y^n(-y^3-1)^{2n}-{\ep_n^{2n}})=0,\]
that is, \[y=y_j^0:=z_6^{2j+1},\] and
\begin{align}\label{eq:refpoints}
y \approx y_{i,j}^0:= \begin{cases} z_6^{2j+1}(1+\frac{1}{3}z_{4n}z_{2n}^i {\ep_n}) & (k \text{ odd}) \\
 z_6^{2j+1}(1+\frac{1}{3}z_{2n}^i {\ep_n})& (k \text{ even}). \end{cases}
\end{align}
Comparing this with \eqref{eq:y_i,j, 1}, we see that the coefficient $c_n/3$ is replaced by $1/3$.

We now construct a path system that induces the corresponding monodromy factorization.
Choosing the principal branch of the argument such that $\arg(z) \in \left[-\pi , \pi \right)$,
we define $\gamma_{i}(s)$ by
\begin{align} \label{eq:path_system}
\gamma_{i}(s)=
\begin{cases}
{\ep_n^{4n}} \exp{ 2s\arg(t_i) \sqrt{-1}}  & (0\le s\le 1/2) \\
((2s-1)|t_i|+ (2-2s) {\ep_n^{4n}})\exp{ \arg(t_i) \sqrt{-1}}& (1/2 \le s \le 1)
\end{cases}
\end{align}
Since ${\ep_n^{4n}}$ is chosen to be sufficiently close to 0, we can consider the branch points originating from $y_{i,j}^0$ and $y_j^0$ to be nearly stationary for $0\le s \le 1/2$ along the path $\gamma_{i}$ in $\C_t$.
Furthermore, the branch point starting from $y_{i,j}^0$ remains essentially stationary for $0 \le s \le 1$ along $\gamma_{i}$.
Therefore, we must identify another branch point that merges with $y_{i,j}^0$.
\begin{claim}\label{claim:branch_move}
For $n=3k$, the branch points $y_j^0$ and $y_{i,j}^0$ merge at $y_{i,j}$ for each $i,j$ as we move in the fibers over the curve ${\gamma}_{i}$.
\end{claim}

\begin{proof}
Define a curve $y_{i,j}(s)$ by
\[y_{i,j}(s):=
\begin{cases} z_{6}^{2j+1}(1+ s \frac{c_n}{3}z_{4n}z_{2n}^i {\ep_n}) & (k \text{ odd}) \\
z_{6}^{2j+1}(1+ s \frac{c_n}{3} z_{2n}^i {\ep_n}) & (k \text{ even}) 
\end{cases}\]
so that $y_{i,j}(0)=y_j^0$ and $y_{i,j}(1)=y_{i,j}$.
We need to show that the image $\phi_n(0,y_{i,j}(s))$ is sufficiently close to the straight-line segment in $\C_t$.
For simplicity of notation, assume $k$ is even.
Since $-y_{i,j}(s)^3-1 \approx s \cdot c_n z_{2n}^i {\ep_n}$,
the value $\phi_n(0,y_{i,j}(s))$ is given by
\begin{align*}
\gamma_{i,j}(s)=\phi_n(0,y_{i,j}(s))
&\approx (s\cdot c_n z_{2n}^i {\ep_n}) \left( (1+s n \frac{c_n}{3} z_{2n}^i {\ep_n}) (s^{2n} \frac{1}{2n+1}{\ep_n^{2n}}-{\ep_n^{2n}})\right) \\
&=-s\cdot c_n  \left(  \frac{2n+(1-s^{2n})}{2n+1}\right)z_{2n}^i {\ep_n^{2n+1}}-s^2 n \frac{c_n^2}{3} \left(\frac{2n+(1-s^{2n})}{2n+1}\right) z_n^i {\ep_n^{2n+2}}\\
&\approx -s\cdot c_n  \left(  \frac{2n+(1-s^{2n})}{2n+1}\right)z_{2n}^i {\ep_n^{2n+1}}.
\end{align*}
Thus, the two paths $\gamma_{i,j}$ and ${{\gamma}}_{i}$ are homotopic relative endpoints in the complement of the other singular values.
\end{proof}

\begin{claim} \label{claim:subsurface}
For $0\le i \le k-1$ and $0 \le j \le 2$, 
the points $y^0_{i,j}$, $y^0_{i+k,j+2}$, $y^0_{i+2k,j+1}$, $y^0_{i+3k,j}$, $y^0_{i+4k,j+2}$, and $y^0_{i+5k,j+1}$
lie on the surface $yf^2=z_{n}^l \ep_n^2$ for some $l$.
\end{claim}

\begin{proof}
Note that $y_{i,j}^0 f(0,y_{i,j}^0)^2$ is evaluated as
\[y_{i,j}^0 f(0,y_{i,j}^0)^2 =\begin{cases} z_6^{2j+1}(1+ \frac{1}{3} z_{2n}^i {\ep_n})(z_n^i {\ep_n^{2}}) &\approx z_{n}^{i+kj+k/2} {\ep_n^{2}}\\
 z_6^{2j+1}(1+ \frac{1}{3} z_{4n}z_{2n}^i {\ep_n})(z_{2n}z_n^i {\ep_n^{2}}) &\approx z_{n}^{i+kj+(k+1)/2} {\ep_n^{2}}, \end{cases}\]
so the corresponding branch point lies on the surface $yf^2=z_n^{m_{i,j}} {\ep_n^{2}}$, where $m_{i,j}=i+kj+\left[  \frac{k+1}{2} \right]$.
Since $i+kj \equiv (k+i)+k(j+2) \equiv (k+2i)+k(j+1) \pmod{n=3k}$, they share the same value $m_{i,j}$ modulo $n$.
\end{proof}

To find vanishing cycles, we utilize another projection map $\pi_y: \psi^{-1}_n(t) \arrowr \C_y$.
This is generically a $(2n+1)$-fold cover, since $\psi_n(X,y)-t=0$ is a polynomial of degree $(2n+1)$ with respect to $X$.
To locate the ramification points of $\pi_y$, we must compute the discriminant $\Delta_y(t)$ of $\psi_n(X,y)-t=0$ with respect to $X$.

\begin{lem} \label{lem:discriminant_1}
The discriminant $\Delta_y(t)$ is given as
\begin{align}
\Delta_y(t)=(-1)^{\frac{n(n-1)}{2}} y^{2n^2-n}((2n)^{2n}{\ep_n^{2n(2n+1)}}-(2n+1)^{2n+1}y^n t^{2n}).
\end{align}
\end{lem}

\begin{proof}
Letting $a=-y^3-1$, we have $\psi_n(X,y)-t=(X+a)(y^n (X+a)^{2n}-{\ep_n^{2n}})-t$.
Since translation along the $X$-axis does not affect the discriminant, we can simply write 
\[F(X):=X(y^n X^{2n}-b)-t=y^n X^{2n+1}-bX-t=0,\]
where $b={\ep_n^{2n}}$.
Then $F'(X)=(2n+1)y^n X^{2n}-b$.
It is a well-known fact that the formula for the discriminant is given by 
\[\Delta_y(t)=(-1)^{\frac{n(n-1)}{2}} y^{-n} R(F,F'),\] 
where $R(F,F')$ is the resultant of $F$ and $F'$.
Recall that for polynomials $f(x)=\sum_{i=0}^n a_i x^i$ and $g(x)=\sum_{i=0}^m b_i x^i$, 
the resultant $R(f,g)$ is the determinant of the $(n+m)\times (n+m)$ Sylvester matrix:
\begin{align}
Syl(f,g)=
\begin{bmatrix} \label{eq:syl}
a_n & a_{n-1} &a_{n-2} & \cdots &a_{1}  & a_{0} & \cdots & 0 \\
0 & a_{n} &a_{n-1} & \cdots &a_{2}  & a_{1} & \cdots & 0 \\
\vdots & \vdots &\ddots & \vdots & \vdots &\vdots & \ddots &\vdots \\
0&0&\cdots&a_n&a_{n-1}&\cdots&\cdots&a_0 \\
b_m&b_{m-1}&\cdots&b_1&b_0&0&\cdots &0 \\
0&b_{m}&\cdots&b_2&b_1&b_0&\cdots &0 \\
\vdots &\vdots &\ddots&\vdots &\vdots &\vdots &\ddots &\vdots \\
0&0&\cdots&\cdots&b_m&b_{m-1}&\cdots&b_0
\end{bmatrix}
\end{align}
Substituting $f=F, g=F'$ into \eqref{eq:syl} yields the $(4n+1)\times (4n+1)$ matrix
\begin{align*}
Syl(F,F')=
\begin{bmatrix}
y^n & 0 & \cdots & \cdots & -b  & -t & \cdots & 0 \\
0 & y^n &0 & \cdots & 0  & -b & \cdots & 0 \\
\vdots & \vdots &\ddots & \vdots & \vdots &\vdots & \ddots &\vdots \\
0&0&\cdots&y^n&0&\cdots&\cdots&-t \\
(2n+1)y^n&0&\cdots&0&-b&0&\cdots &0 \\
0&(2n+1)y^n&\cdots&0&0&-b&\cdots &0 \\
\vdots &\vdots &\ddots&\vdots &\vdots &\vdots &\ddots &\vdots \\
0&0&\cdots&\cdots&(2n+1)y^n&0&\cdots&-b
\end{bmatrix}.
\end{align*}
Applying elementary row operations, we obtain
\begin{align*}
|Syl(F,F')|&=
\begin{vmatrix}
y^n & 0 & \cdots & \cdots & -b  & -t & \cdots & 0 \\
0 & y^n &0 & \cdots & 0  & -b & \cdots & 0 \\
\vdots & \vdots &\ddots & \vdots & \vdots &\vdots & \ddots &\vdots \\
0&0&\cdots&y^n&0&\cdots&\cdots&-t \\
0&0&\cdots&0&2nb&(2n+1)t&\cdots &0 \\
\vdots &\vdots &\ddots&\vdots &\vdots &\vdots &\ddots &\vdots \\
0&0&\cdots&\cdots&0&\cdots &2nb&(2n+1)t \\
0&0&\cdots&\cdots&(2n+1)y^n&0 &\cdots & -b
\end{vmatrix}\\
&=y^{2n^2} 
\begin{vmatrix}
2nb & (2n+1)t &0 & \cdots & 0  & 0  \\
0&2nb &(2n+1)t&\cdots&\cdots&0 \\
\vdots &\vdots &\ddots&\vdots &\vdots &\vdots \\
0&0&\cdots&\cdots &2nb&(2n+1)t \\
(2n+1)y^n&0&\cdots&\cdots &0& b
\end{vmatrix} \\
&=y^{2n^2}((2n)^{2n}b^{2n+1}-(2n+1)^{2n+1}y^n t^{2n}),
\end{align*}
so the discriminant $\Delta_y(t)$ is given by
\begin{align*}
\Delta_y(t)=(-1)^{\frac{n(n-1)}{2}} y^{2n^2-n}((2n)^{2n}({\ep_n^{2n}})^{2n+1}-(2n+1)^{2n+1}y^n t^{2n}),
\end{align*}
which proves the claim.
\end{proof}

By the lemma above, $y=0$ is the solution to $\Delta_y(t)=0$. However, this root is discarded because $y^n$ is the leading coefficient.
Thus, the nontrivial solutions to $\Delta_y(t)=0$, which correspond to the branch points $b_l(t)$ of $\pi_y$, are given by
\[b_l(t)^n=\frac{(2n)^{2n} {\ep_n^{2n(2n+1)}}}{(2n+1)^{2n+1} t^{2n}},\]
or equivalently,
\begin{align}
b_l(t)=z_n^l C_n \frac{{\ep_n^{2(2n+1)}}}{t^2},
\end{align}
where \[C_n=\left(\frac{2n}{2n+1}\right)^2 \cdot \frac{1}{(2n+1)^{1/n}}.\]
Since $b_l(t)$ is a simple root of $\Delta_y(t)=0$, 
there is a unique ramification point $B_l$ $(l=0,\dots, n-1)$ with multiplicity $2$ corresponding to $y=b_l(t)$.
We have identified the branch points of $\pi_y$, but $\pi_y$ fails to be a branched covering at $y=0$.
To resolve this, we extend the map $\pi_y$ to be a branched covering.

\begin{lem}
The extension \[\hat{\pi}_y: \hat{f}_n^t \arrowr \C_y\] 
of the map $\pi_y$ is a $(2n+1)$-fold branched cover, whose branch points are $y=0$ and $b_l(t)$.
The ramification points over $y=0$ in $\hat{f}_n^t$ are $O_{k}$ ($k=0,\dots n-1$) with multiplicity $2$, and $O_n$ with multiplicity $1$. 
Therefore, $\hat{f}_n^t \cong \C$ by Proposition~\ref{prop:condition}.
\end{lem}
\begin{proof}
Define \[\hat{f}_n^t =\Phi(U\cup \cup_{j=1}^n U_{1,j})\cap \psi_n^{-1}(t).\]
The $y$-coordinate transforms into \[s=\mathfrak{s}_1^2(\mathfrak{s}_1\mathfrak{t}_1+{\ep_n^{2}})\] in $U_{1,1}=\{(\mathfrak{s}_1, \mathfrak{t}_1)\}$,
and inductively we obtain 
\[y=\mathfrak{s}_n^2 (\mathfrak{s}_n \mathfrak{t}_n +z_n^{n-1}{\ep_n^{2}}) \prod_{k=1}^{n-1} (\mathfrak{s}_n \mathfrak{t}_n -(z_n^{k-1}-z_n^{n-1}){\ep_n^{2}})^2.  \]
If we fix the fiber over $t_O$, then $\mathfrak{t}_n=t_O$, yielding
\[y=\mathfrak{s}_n^2 (\mathfrak{s}_n t_O +z_n^{n-1}{\ep_n^{2}}) \prod_{k=1}^{n-1} (\mathfrak{s}_n t_O -(z_n^{k-1}-z_n^{n-1}){\ep_n^{2}})^2. \]
If $y=0$, we find $n-1$ points \[O_{k-1}: \mathfrak{s}_n=(z_n^{k-1}-z_n^{n-1}){\ep_n^{2}}/t_O\] with multiplicity 2 ($k=1,\dots, n-1$),
\[O_{n-1}: \mathfrak{s}_n=0\] with multiplicity 2, 
and \[O_n: \mathfrak{s}_n=- z_n^{n-1} {\ep_n^{2}}/t_O\] with multiplicity 1. 
Note that the point $ O_n$ and $O_k$ ($k\neq n$) essentially correspond to $y=0$ in $U$ and $\mathfrak{s}_k=0$ in $U_{1,k}$, respectively.
Consequently, we can extend $\pi_y $ to the $(2n+1)$-fold branched cover \[\hat{\pi}_y: \hat{f}_n^{t_O} \arrowr \C_y,\] 
with ramification point $O_k$ $(k=0,\dots, n)$ over $y=0$.  
\end{proof}

We now lift the trajectory of the path $y_{i,j}(s)$ in $\C_y$ to $\hat{f}_n^t$.
Note that two points $y_{i,j}^0$ and $y_j^0$ lie on different sheets of $\pi_y$.
This implies that as we move along $\gamma_i$, the sheet containing $y_{i,j}(s)$ must change; in other words, 
the path intersects $b_l(\gamma_i(s))$ for some $s$. 
Furthermore, this $s$ must be close to $1$ so that the norm $|b_l(\gamma_i(s))|$ is close to $1$,
and $l$ will precisely match the value given in Claim~\ref{claim:subsurface}.

\begin{claim}
For each $i,j$, there is a unique integer $l=l(i,j)$ such that $b_l(\gamma_i(1))$ is sufficiently close to $y_{i,j}(1)$, where $l(i,j)$ is the function given in Claim~\ref{claim:subsurface}.
\end{claim}
\begin{proof}
Since
\[b_l(t_{i})=\begin{cases} z_{2n}^{2l-(2i+1)}& (n \text{ odd}) \\ z_n^{l-i} & (n \text{ even}) \end{cases},\] 
there exists an $l=l(i,j)$ such that $z_{2n}^{2l-(2i+1)} =z_6^{2j+1} $ for odd $n $, and $z_n^{l-i}=z_6^{2j+1}$ for even $n$.
In particular, when $n=3k$ is odd, $z_6^{2j+1}=z_{6k}^{k(2j+1)}$, which gives $2l \equiv 2i+1+k(2j+1) \pmod{2n}$, 
so that $l \equiv i+jk+(k+1)/2 \pmod{n}$.
When $n=3k$ is even, $z_{3k}^{(k/2)(2j+1)}=z_{3k}^{l-i}$, so that $l \equiv i+jk+k/2 \pmod{n}$. 
Comparing this with the function provided in Claim ~\ref{claim:subsurface}, we verify that they yield the same values.
\end{proof}

Let $c_{i,j}$ denote the vanishing cycle in ${F}_n^{t_O}$ corresponding to $y_{i,j}$, that is, $c_{i,0}, c_{i,1}, c_{i,2}$ are the vanishing cycles corresponding to the singular value $t_i$.
Under the map $\Phi$, each $c_{i,j}$ is projected to a path $\tilde{c}_{i,j}$ in $f_n^{t_O}$, and in particular in $\hat{f}_n^{t_O}$, connecting $y_{i,j}$ and $y_j$.
Under the map $\hat{\pi}_y$ the path $\tilde{c}_{i,j}$ further projects to a curve $\tilde{\tilde{c}}_{i,j}$ in $\C_y$.
We know that $\tilde{\tilde{c}}_{i,j}$ is formed by tracking the branch point $b_l$
over the path $\overline{\gamma}$ (the reverse path of $\gamma_i$, from the singular value $t_i$ to $t_O$).
Thus, $\tilde{\tilde{c}}_{i,j}$ is isotopic to the concatenation: 
\[b_{l(i,j)}(\overline{\gamma}_i) * b_{l(i,j)}(\gamma_i)* y_{i,j}.\]
Here, $b_l(\gamma_i(s))$ moves in a counter-clockwise (clockwise, resp.) direction for $0\le s \le 1/2$ when $0\le i \le n-1$ ($n\le i \le 2n-1$, resp.), and moves radially for $1/2\le s \le 1$.
Figure~\ref{fig:move_branchpt} illustrates the movement of $b_{l(i,1)}(\overline{\gamma}_i)$. 
For $j=0$ ($2$, resp.), the curve $b_{l(i,j)}(\overline{\gamma}_i)$ is obtained by rotating the configuration in Figure~\ref{fig:move_branchpt} by $-2\pi/3$($2\pi/3$, resp.) in the positive direction.

\begin{figure}[ht]
\centering
\includegraphics{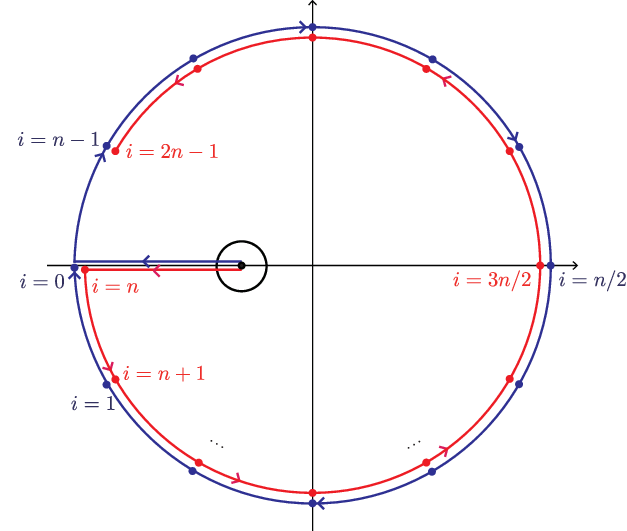}
\caption{Movement of $b_{l(i,1)}(\overline{\gamma}_i)$. The blue arrow represents the curve for $0\le i\le n-1$, and the red for $n \le i \le 2n-1$}
\label{fig:move_branchpt}
\end{figure}

Since the points $a_j=(0,y_j)$ lie in the same sheet of $\hat{\pi}_y$, and considering that each $b_{l(i,j)}(\overline{\gamma}_i)$ starts from near $a_j$, 
the branched cover structure of $\hat{\pi}_y:\hat{f}_n^{t_O} \arrowr \C_y$ is depicted in Figure~\ref{fig:branched_cover_n}.

\begin{figure}[ht]
\includegraphics[scale=0.6]{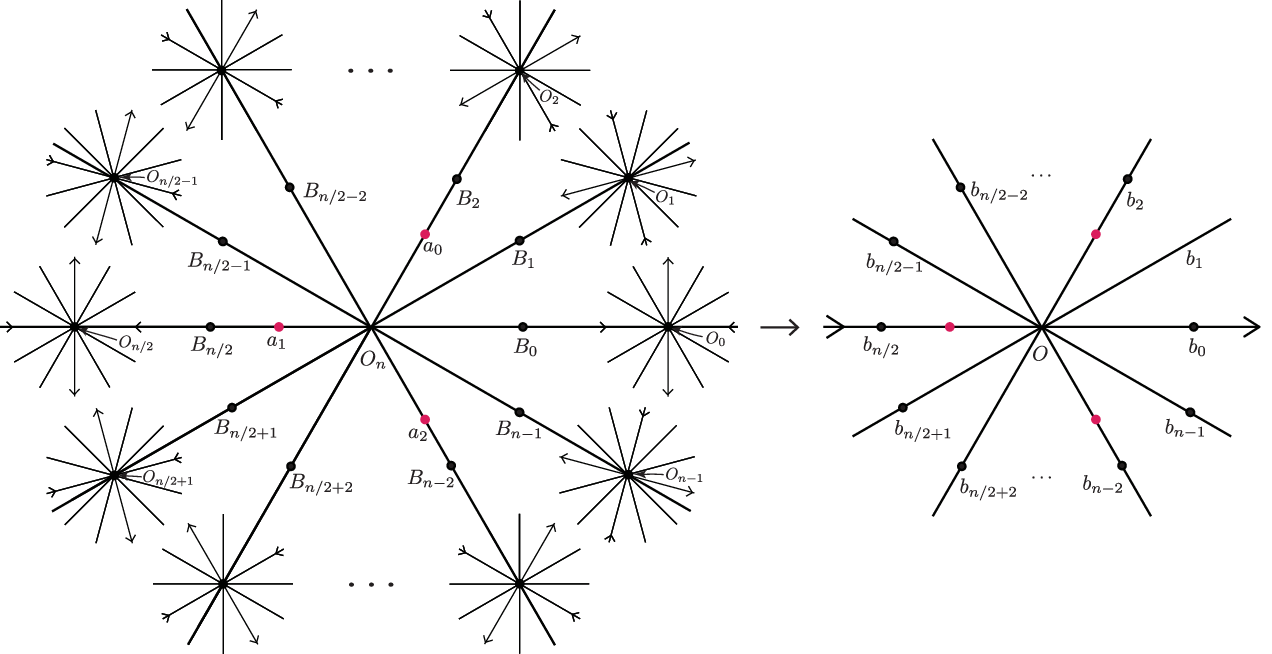}
\caption{Branched cover $\hat{\pi}_y: \hat{f}_n^{t_O} \arrowr \C_y$.}
\label{fig:branched_cover_n}
\end{figure}

Note that the curves $\tilde{c}_{i,j}$, $\tilde{c}_{i+k,j+2}$, $\tilde{c}_{i+2k,j+1}$, $\tilde{c}_{i+3k,j}$, $\tilde{c}_{i+4k,j+2}$, $\tilde{c}_{i+5k,j+1}$ 
are contained within a single disk containing 9 branch points. 
Therefore, they lift to vanishing cycles contained in a subsurface $\Sigma_{4,1}\subset \Sigma_{3n+1}$.
Also note that $l(0,j)$ is given by
\[l(0,j)=\begin{cases} \frac{1}{2}(1+\frac{n}{3}(2j+1))& (n \text{ odd}) \\ \frac{n}{6}(2j+1) & (n \text{ even})\end{cases}.\]
Thus, in Figure~\ref{fig:move_branchpt}, the branch point $b_{l(0,j)}(\overline{\gamma}_i)$ winds by $2\pi(2n-1)/(2n)$, or $2\pi$. 
As $i$ increases by 1, the rotation angle decreases by $2\pi/n$.
From this information, we can precisely draw the vanishing arcs $\tilde{c}_{i,j}$ in $\hat{f}_n^{t_O}$.

\begin{rmk}\label{lem:ppaction}
(Symmetry of curves $\tilde{c}_{i,j}$)

Let $\theta_{i,j}$ be the non-negative angle defined by \[\theta_{i,j}\equiv \left(\frac{2l(i,j)}{n}-\frac{2j+1}{3}\right)\pi \pmod{2\pi}.\]
For $0\le i \le n-1$ and $j=0,1,2$, the arc $\tilde{c}_{i,j}$ is isotopic to a simple curve starting at $a_{j}$, 
winding clockwise around $O_n$ by $2\pi- \theta_{i,j}$, passing through $B_{l(i,j)}$, 
winding counter-clockwise around $O_{l(i,j)}$ by $\pi-\theta_{i,j}/2$, and finally ending at $a_{i,j}$ (which maps to $y_{i,j}$ under $\hat{\pi}_y$).

For $n\le i \le 2n-1$, $j=0,1,2$, $\tilde{c}_{i,j}$ is isotopic to a simple curve starting at $a_{j}$, winding counter-clockwise around $O_n$ by $ \theta_{i,j}$ passing through $B_{l(i,j)}$, 
winding clockwise around $O_{l(i,j)}$ by $\theta_{i,j}/2$, and ending at $a_{i,j}$.

It is clear that a curve $\tilde{c}_{i,j}$ satisfying $l=l(i,j)$ passes through $B_l$.
For example, when $l=[\frac{n+1}{2}]$, since $l(0,j)=l$ implies $j=1$, the curves $\tilde{c}_{0,1}, \tilde{c}_{k,0}, \tilde{c}_{2k,2}, \tilde{c}_{3k,1}, \tilde{c}_{4k,0}$ and $\tilde{c}_{5k,2}$ all pass through $B_l$.
\end{rmk}

\begin{exam} \label{ex:n=3}
Let $n=3$. 
Then the vanishing arcs $\tilde{c}_{i,j}$ for $i=0,5$ and $j=0,1,2$ are drawn in the figures below.

\begin{figure}[ht]
\begin{center}
\includegraphics[scale=0.6]{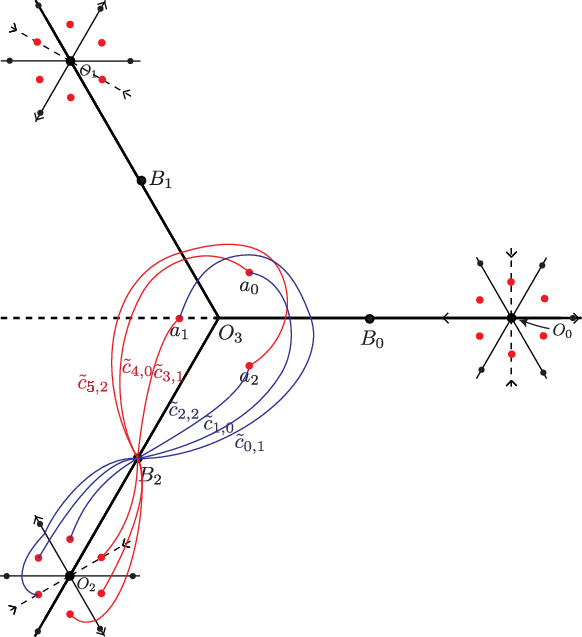}
\includegraphics[scale=0.6]{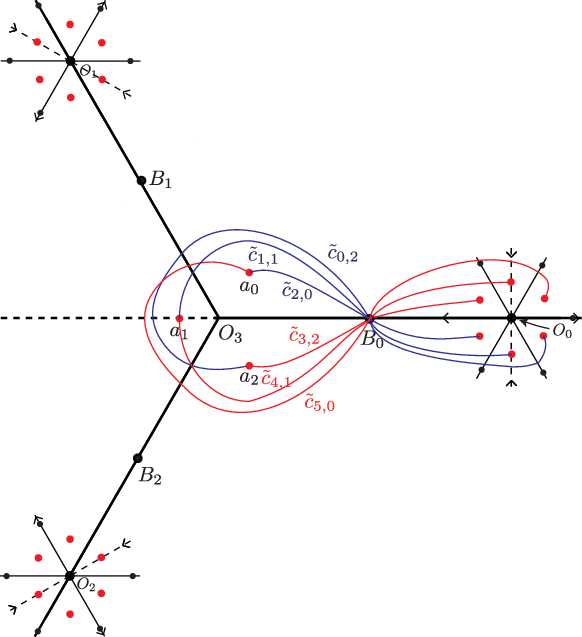}
\includegraphics[scale=0.6]{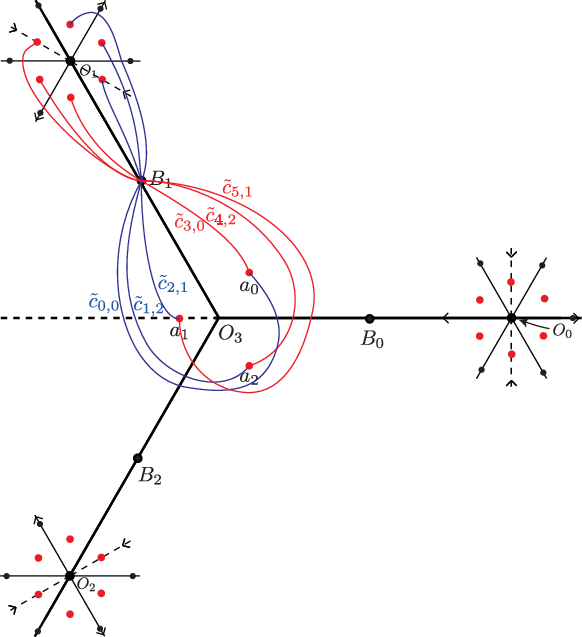}
\end{center}
\caption{Vanishing arcs $\tilde{c}_{i,j}$}
\end{figure}

To lift the arcs $\tilde{c}_{i,j}$ to vanishing cycles of the Lefschetz fibration, we align all red vertices along a single line.
We place the vertices near $O_2$ to the left of $O_3$, 
and we place the vertices near $O_0$ between $O_2$ and $O_3$ by moving them in the clockwise direction.
Similarly, we place the vertices near $O_1$ between $O_0$ and $O_3$ by moving them in the clockwise direction.
Finally, the vertices forming a circular arrangement around each $O_i$ are aligned in a line, drawn as the figure below.

\begin{figure}[ht]
\centering
\includegraphics[scale=1]{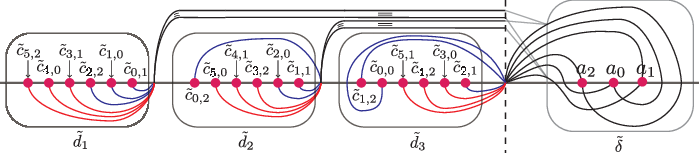}
\caption{$\tilde{c}_{i,j}$ after perturbation}
\end{figure}

For example, $a_2, a_0$, and $a_1$ are placed in a line from left to right, after moving $a_2, a_0$ in the clockwise direction.
Note that the monodromy around the origin $\Phi_O$ in the mapping class group $\mathcal{M}(D^2)$ is given by
\[t_{\tilde{d}_1} t_{\tilde{d}_2} t_{\tilde{d}_3} t_{\tilde{\delta}}^2,\]
which lifts to \[(t_{d_1}t_{d'_1})(t_{d_2}t_{d'_2})(t_{d_3}t_{d'_3})t_{\delta}\]
 in $\mathcal{M}(\Sigma_{10,1})$ via the hyperelliptic lift.
Also note that the monodromy $\Phi_i$ corresponding to the singular value $t_i$ is given by 
\[\Phi_i=t_{\tilde{c}_{i,0}}t_{\tilde{c}_{i,1}}t_{\tilde{c}_{i,2}},\]
where the arcs $\tilde{c}_{i,j}$ are mutually disjoint for a fixed $i$.
Thus, the monodromy factorization in $\mathcal{M}(D^2)$ is given by
\begin{align*}
&\Phi_O \cdot \Phi_5 \Phi_4\Phi_3\Phi_2\Phi_1\Phi_0 \\
&=\left(t_{\tilde{d}_1} t_{\tilde{d}_2} t_{\tilde{d}_3} t_{\tilde{\delta}}^2\right)\cdot  \prod_{i=5}^0 t_{\tilde{c}_{i,0}}t_{\tilde{c}_{i,1}}t_{\tilde{c}_{i,2}}.
\end{align*}
Since the arc $\tilde{c}_{1,2}$ is disjoint from $\tilde{c}_{0,j}$ for all $j\neq 2$,
we can commute the three words $t_{\tilde{c}_{1,2}}, t_{\tilde{c}_{0,2}},$ and $t_{\tilde{c}_{0,0}}$ to the front. 
For the remaining arcs, by commuting words appropriately, we can rewrite the word as
\begin{align*}
&\prod_{i=5}^0 t_{\tilde{c}_{i,0}}t_{\tilde{c}_{i,1}}t_{\tilde{c}_{i,2}} \\
=& (t_{\tilde{c}_{5,0}}t_{\tilde{c}_{5,1}}t_{\tilde{c}_{5,2}})(t_{\tilde{c}_{4,0}}t_{\tilde{c}_{4,1}}t_{\tilde{c}_{4,2}})  
(t_{\tilde{c}_{3,0}}t_{\tilde{c}_{3,1}}t_{\tilde{c}_{3,2}})(t_{\tilde{c}_{2,0}}t_{\tilde{c}_{2,1}}t_{\tilde{c}_{2,2}})
(t_{\tilde{c}_{1,0}}t_{\tilde{c}_{1,1}}t_{\tilde{c}_{1,2}})(t_{\tilde{c}_{0,0}}t_{\tilde{c}_{0,1}}t_{\tilde{c}_{0,2}})
\\
=& (t_{\tilde{c}_{5,2}}) (t_{\tilde{c}_{5,0}}t_{\tilde{c}_{4,0}}) ( t_{\tilde{c}_{5,1}}t_{\tilde{c}_{4,1}}t_{\tilde{c}_{3,1}})
( t_{\tilde{c}_{4,2}}t_{\tilde{c}_{3,2}}t_{\tilde{c}_{2,2}})( t_{\tilde{c}_{3,0}}t_{\tilde{c}_{2,0}}t_{\tilde{c}_{1,0}})
( t_{\tilde{c}_{2,1}}t_{\tilde{c}_{1,1}}t_{\tilde{c}_{0,1}}) ( t_{\tilde{c}_{1,2}}t_{\tilde{c}_{0,2}}t_{\tilde{c}_{0,0}})
\end{align*}
We now apply a Hurwitz move to the words $t_{\tilde{c}_{1,2}} t_{\tilde{c}_{0,2}}t_{\tilde{c}_{0,0}}$, which gives an equivalent Lefschetz fibration, so that they appear to the front of the factorization. 
Denote \[\tilde{c}'_{i,j}:= \Phi^{-1}_O (\tilde{c}_{i,j}).\]
Since $\tilde{c}'_{0,0}$ and $\tilde{c}_{5,2}$ are disjoint, their Dehn twists commute, yielding
\begin{align*}
&\Phi_O \cdot \prod_{i=5}^0 t_{\tilde{c}_{i,0}}t_{\tilde{c}_{i,1}}t_{\tilde{c}_{i,2}} \\
=& \Phi_O\cdot 
(t_{\tilde{c}'_{1,2}} t_{\tilde{c}'_{0,2}}t_{\tilde{c}_{5,2}}) (t_{\tilde{c}'_{0,0}} t_{\tilde{c}_{5,0}}t_{\tilde{c}_{4,0}}) ( t_{\tilde{c}_{5,1}}t_{\tilde{c}_{4,1}}t_{\tilde{c}_{3,1}})
( t_{\tilde{c}_{4,2}}t_{\tilde{c}_{3,2}}t_{\tilde{c}_{2,2}})( t_{\tilde{c}_{3,0}}t_{\tilde{c}_{2,0}}t_{\tilde{c}_{1,0}})
( t_{\tilde{c}_{2,1}}t_{\tilde{c}_{1,1}}t_{\tilde{c}_{0,1}})
\end{align*}
Note that the curves $\tilde{c}'_{0,0}$ and $\tilde{c}'_{1,2}$ are obtained from $\tilde{c}_{0,0}, \tilde{c}_{1,2}$ by applying the negative Dehn twist along $\tilde{d}_3$, 
and the curve $\tilde{c}'_{0,2}$ is obtained from $\tilde{c}_{0,2}$ by applying the negative Dehn twist along $\tilde{d}_2$.
As a result, the curves in each bracket exhibit a repetitive pattern:
the $i$-th curve in each bracket passes through $B_{i+1}$ (with indices taken modulo $3$).
After reindexing the curves, we obtain 
\begin{align*}
&
\Phi_O\cdot (t_{\tilde{c}_{3,6}}t_{\tilde{c}_{2,6}}t_{\tilde{c}_{1,6}})(t_{\tilde{c}_{3,5}}t_{\tilde{c}_{2,5}}t_{\tilde{c}_{1,5}})
(t_{\tilde{c}_{3,4}}t_{\tilde{c}_{2,4}}t_{\tilde{c}_{1,4}}) (t_{\tilde{c}_{3,3}}t_{\tilde{c}_{2,3}}t_{\tilde{c}_{1,3}})
(t_{\tilde{c}_{3,2}}t_{\tilde{c}_{2,2}}t_{\tilde{c}_{1,2}}) (t_{\tilde{c}_{3,1}}t_{\tilde{c}_{2,1}}t_{\tilde{c}_{1,1}}) \\
=&\prod_{j=6}^1 \prod_{i=3}^1 t_{\tilde{c}_{i,j}}.
\end{align*}
\end{exam}

Generalizing the $n=3$ example, we now prove the main theorem for the $n=3k$ case.

\begin{proof}[Proof of ~\eqref{eq:genfact_1}, $n=3k$ case]
First note that $\Phi_O$ consists of positive Dehn twists along curves along which the quotient is drawn in Figure ~\ref{fig:perturbfiber1}.
Since the torus corresponds to $f=0$ and each genus-$2$ subsurface corresponds to $yf^2=z_n^l \ep_n^2$,
each vanishing cycle corresponds to $\tilde{\delta}$, and $\tilde{d_i}$ in Figure~\ref{fig:vc}, respectively.
Therefore, we have
\[\Phi_O=t_{\tilde{d}_1}\cdots t_{\tilde{d}_n} \cdot t_{\tilde{\delta}}^2\]
in $\mathcal{M}(\hat{f}_n^{t_O})$.

Now, we need to calculate the remaining cycles.
Following a process similar to Example~\ref{ex:n=3}, 
we perturb $\hat{f}_n^{t_O}$ so that all $\Phi$-branch points are aligned in a straight line. 
We arrange words of Dehn twists by commuting with each other, and
after applying Hurwitz move to some words, we obtain the desired result.

Let $I_i=[\frac{n}{3} i,\frac{n}{3}(i+1))\cap \Z$, with the reverse order relative to the standard order in $\Z$.
The monodromy factorization we obtained can be written as
\[ \Phi_O \cdot 
\left(\prod_{i\in I_5} t_{\tilde{c}_{i,0}}t_{\tilde{c}_{i,1}}t_{\tilde{c}_{i,2}}\right)
\left(\prod_{i\in I_4} t_{\tilde{c}_{i,0}}t_{\tilde{c}_{i,1}}t_{\tilde{c}_{i,2}}\right)
\cdots
\left(\prod_{i\in I_0} t_{\tilde{c}_{i,0}}t_{\tilde{c}_{i,1}}t_{\tilde{c}_{i,2}}\right),\]
For example, in the first product, the $i=2n-1$ term appears first, and the $i=5n/3$ term appears last.

We now commute some words in this factorization.
One can verify the following properties of the vanishing cycles $\tilde{c}_{i,j}$.

\begin{enumerate}[(i)]
\item For a fixed $i$, the curves $\tilde{c}_{i,j}$ are mutually disjoint for $j=0,1,2$.
\item For each $i\in I_1$, $\tilde{c}_{i,2}$ and $\tilde{c}_{i',j'}$ are disjoint for $i' <i$ and $j'\neq 2$.
\item For each $\iota, \iota' \in I_i$, $\tilde{c}_{\iota,j}$ and $\tilde{c}_{\iota', j'}$ are disjoint when $\iota'<\iota, j\neq j'$.
\item For each $\iota\in I_i$ and $\iota'\in I_{i-1}$, $\tilde{c}_{\iota,j}$ is disjoint from $\tilde{c}_{\iota', j-1}$, where the index $j$ is taken modulo 3.
\end{enumerate}

Using these facts, we can reorder the products. Defining $\Phi_{i,j}$ as 
\begin{align}
\Phi_{i,j}:=\prod_{\iota\in I_i}t_{\tilde{c}_{\iota,j}}, 
\end{align}
we rewrite the factorization as
\[ \Phi_O  (\Phi_{5,2}\Phi_{5,0}\Phi_{5,1})(\Phi_{4,2}\Phi_{4,0}\Phi_{4,1})(\Phi_{3,2}\Phi_{3,0}\Phi_{3,1})
(\Phi_{2,2}\Phi_{2,0}\Phi_{2,1})(\Phi_{1,0}\Phi_{1,1})(\Phi_{0,1})(\Phi_{1,2}\Phi_{0,2}\Phi_{0,0}). \]
Since $\Phi_{i,j}$ and $\Phi_{i-1, j-1}$ commute, we can rearrange the terms to obtain:
\begin{align*}
&(\Phi_{5,2}\Phi_{5,0}\Phi_{5,1})(\Phi_{4,2}\Phi_{4,0}\Phi_{4,1})(\Phi_{3,2}\Phi_{3,0}\Phi_{3,1})
(\Phi_{2,2}\Phi_{2,0}\Phi_{2,1})(\Phi_{1,0}\Phi_{1,1})(\Phi_{0,1})\\
=&(\Phi_{5,2})(\Phi_{5,0}\Phi_{4,0})(\Phi_{5,1}\Phi_{4,1}\Phi_{3,1})(\Phi_{4,2}\Phi_{3,2}\Phi_{2,2})
(\Phi_{3,0}\Phi_{2,0}\Phi_{1,0})(\Phi_{2,1}\Phi_{1,1}\Phi_{0,1}).
\end{align*}
We now apply a Hurwitz move to $\Phi_{1,2}\Phi_{0,2}\Phi_{0,0}$.
Let $\Phi'_{i,j}=\Phi_O^{-1}\circ \Phi_{i,j}\circ \Phi_O$.
Then the factorization can be written as
\[\Phi_O (\Phi'_{1,2}\Phi'_{0,2}\Phi'_{0,0})(\Phi_{5,2})(\Phi_{5,0}\Phi_{4,0})(\Phi_{5,1}\Phi_{4,1}\Phi_{3,1})(\Phi_{4,2}\Phi_{3,2}\Phi_{2,2})
(\Phi_{3,0}\Phi_{2,0}\Phi_{1,0})(\Phi_{2,1}\Phi_{1,1}\Phi_{0,1}).\]
Since $\Phi'_{0,0}$ and $\Phi_{5,2}$ commute, we finally write
\begin{align}
&\Phi_O (\Phi'_{1,2}\Phi'_{0,2}\Phi_{5,2})(\Phi'_{0,0}\Phi_{5,0}\Phi_{4,0})(\Phi_{5,1}\Phi_{4,1}\Phi_{3,1})(\Phi_{4,2}\Phi_{3,2}\Phi_{2,2})
(\Phi_{3,0}\Phi_{2,0}\Phi_{1,0})(\Phi_{2,1}\Phi_{1,1}\Phi_{0,1}).
\end{align}

It can be shown that the curves $\tilde{c}_{i,j}$ and $\tilde{c}'_{i,j}$ are drawn in Figure~\ref{fig:vc}.
Taking into account Theorem~\ref{thm:comb_1}, we see that performing a negative Dehn twist along the curve $d$
yields Equation~\eqref{eq:genfact_1}.
\end{proof}

\subsection{Vanishing cycles, $n\neq3k$ case}
Many aspects of the construction are identical to the $n=3k$ case; however, the critical points $y_{i,j(i)}$ and singular values $t_{i}$ differ. We construct a path system $\gamma_i(s)$ as defined in \eqref{eq:path_system} for $i=0,\dots , 6n-1$. Note that the branch points of $\Phi_n^0$ are given by $y_j^0=z_6^{2j+1}$, and
\begin{align}
y_{i,j(i)}^0= \begin{cases} z_6^{2j(i)+1}(1+ \frac{1}{3} z_{12n}z_{6n}^i {\ep_n}) & (n \text{ odd}) \\
z_6^{2j(i)+1}(1+ \frac{1}{3} z_{6n}^i {\ep_n}) & (n \text{ even}). \end{cases}
\end{align}
These points are obtained by solving $y^{n}f^{2n}={\ep_n^{2n}}$, where $j(i)$ is the same function as defined in Lemma~\ref{lem:j(i)}.

\begin{claim} \label{claim:overlap}
For $n\neq3k$, $y_{j(i)}^0$ and $y_{i,j(i)}^0$ merge at $y_{i,j(i)}$ for each $i$, as they move in the fiber over the curve ${\gamma}_{i}$.
\end{claim}

\begin{proof}
For the case $n=2m+1$, define $y_i(s)$ by
\[y_i(s)=z_6^{2j(i)+1}(1+s \cdot \frac{c_n}{3} z_{12n}z_{6n}^i {\ep_n}),\]
where $c_n=(\frac{1}{2n+1})^{1/2n}$.
Then $y_i(0)\in (y^3=-1)$, and since $f(0,y_i(s))=s c_n z_{12n} z_{6n}^i {\ep_n}$, the corresponding value $t_i(s)$ is calculated as
\begin{align*}
f(y_i(s)^n f^{2n}-{\ep_n^{2n}})&=sc_n z_{12n}z_{6n}^i {\ep_n} (z_6^{2nj(i)+n}(1+ns \frac{c_n}{3}z_{12n}z_{6n}^i {\ep_n}) 
\cdot (2n+1)^{-1} s^{2n} z_6 z_3^i {\ep_n^{2n}} -{\ep_n^{2n}} )\\
&\approx sc_n z_{12n} z_{6n}^i (z_6^{2i+2nj(i)+2m+2} {s^{2n}}{(2n+1)}^{-1} -1){\ep_n^{2n+1}} \\
&= sc_n z_{6n}^i ( {s^{2n}}{(2n+1)}^{-1} -1){\ep_n^{2n+1}},
\end{align*}
where the last equality follows from the fact that $i+nj(i)+m+1\equiv 0 \pmod{3}$.
Thus, the endpoint of the path $y_{i}(s)$ is equal to $y_{i, j(i)}$.
Since these trajectories are nearly straight-line segments, they can be isotoped to $\overline{\gamma}_{i}$ simultaneously, without creating pairwise intersections in their interiors.
The case for $n=2m$ follows a similar movement of branch points in $\C_y$.
\end{proof}

\begin{claim} \label{claim:subsurface2}
For $0\le i \le n-1$, 
the points $y^0_{i,j(i)}$, $y^0_{i+n,j(i+n)}$, $y^0_{i+2n,j(i+2n)}$, $y^0_{i+3n,j(i+3n)}$, $y^0_{i+4n,j(i+4n)}$, $y^0_{i+5n,j(i+5n)}$
lie on the surface $yf^2=z_{n}^{l(i)} {\ep_n^{2}}$ for some $l(i)\in \{0,\dots, n-1\}$.
Hence we obtain a map \[l:\{0,\dots, 6n-1\} \arrowr \{0,\dots,n-1\}.\]
Moreover, the restriction of $l$ to $\{0,\dots, n-1\}$ is bijective.
\end{claim}

\begin{proof}
By calculating $y^0_{i,j(i)}f(0,y^0_{i,j(i)})^2$, we obtain
\begin{align*}
y^0_{i,j(i)}f(0,y^0_{i,j(i)})^2 \approx 
\begin{cases} z_{6n}^{2nj(i)+2m+2i}{\ep_n^{2}} & (n=2m ) \\
z_{6n}^{2nj(i)+2m+2i+2}{\ep_n^{2}} & (n=2m+1) \end{cases} \end{align*}
Since $2nj(i)+2m+2i$ ($2nj(i)+2m+2i+2$, resp.) is a multiple of 6, we have $6l \equiv 2nj(i)+2m+2i \pmod{6n}$ ($6l \equiv 2nj(i)+2m+2i+2 \pmod{6n}$, resp.), so that
\begin{align} \label{eq:l(i)}
l=l(i)\equiv \begin{cases} \frac{1}{3}(i+nj(i)+m) \pmod{n} & (n=2m) \\ \frac{1}{3}(i+nj(i)+m+1) \pmod{n}& (n=2m+1) \end{cases}
\end{align}
Note that $j(i+n)\equiv j(i)-1 \pmod{3}$ and $j(i+2n) \equiv j(i)+1 \pmod{3}$, since
$nj(i)\equiv 2i+n \pmod{3}$ for $n=2m$, and $nj(i) \equiv 2i+n+1 \pmod 3$ for $n=2m+1$.
Hence, we have $l(i)=l(i+n)=l(i+2n)=l(i+3n)=l(i+4n)=l(i+5n)$.

To prove the bijectivity of $l(i)$, we divide $n$ into four cases: $n=6k+1, 6k+5$ (odd cases), and $n=6k+2, 6k+4$ (even cases).
Define $\Delta j:=j(i+1)-j(i)$ and $\Delta l:= l(i+1)-l(i)$.
From the relations
\begin{align*}
&1+n \Delta j \equiv 0 \pmod{3} \\
&3\Delta l \equiv 1 +n\Delta j \pmod{3n}
\end{align*}
along with \eqref{eq:l(i)}, we obtain Table~\ref{tab:data}.
\begin{table}[ht] 
\centering
\begin{tabular}[t]{lcc|cc}
\hline
$n$		&$j(0) \mod3$	&$l(0)\mod n$	& $\Delta j \mod3$ 	&$\Delta l \mod n$\\
\hline
$6k+1$	& 2			&$5k+1$		&$-1$			&$-2k$		\\
$6k+5$	& 0			&$k+1$		&1				&$2k+2$		\\
\hline
$6k+2$	&1			&$3k+1$		&1				&$2k+1$		\\
$6k+4$	&1	 		&$3k+2$		&$-1$			&$-2k-1$		\\
\hline
\end{tabular}
\caption{Table of initial data $j(0), l(0)$ and the change of $j, l$ for each case}
\label{tab:data}
\end{table}

Using the Euclidean algorithm, we can show that in each case, $n$ and $\Delta l$ are relatively prime. 
Thus, the map 
$l: \{0,\dots, n-1\} \arrowr \{0,\dots, n-1\}$ is bijective.
\end{proof}
One remark in the proof is that $3 (\Delta l) \equiv 1 \pmod n$, a fact that will be utilized later.

Next, we consider how the movement of $y_i (s)$ in each fiber can be related to the reference fiber $\hat{f}_n^{t_O}$.
The method for finding the branch point $b_l(t_i)$ of $\pi_y$ is the same as in Section \ref{sec:vc_3k}:
find $l$ such that $b_{l}(t_i)=z_6^{2j(i)+1}$.
Since $b_{l}(t_i)$ is given by
\begin{align*}
b_{l}(t_i)=\begin{cases} z_n^l z_{6n}^{-1}z_{3n}^{-i} & (n=2m+1) \\ z_n^l z_{3n}^{-i}, & (n=2m) \end{cases}
\end{align*}
we have the equation \[3l=i+nj(i)+m \pmod{3n} \] for $n=2m$ and \[3l=i+nj(i)+m+1 \pmod{3n}\] for $n=2m+1$.
Therefore, $l=l(i)$.

Before proving the main theorem, let us see the example.

\begin{exam}
Assume $n=4$. 
For each singular value $t_i$ ($i=0,\dots, 23$), let $\tilde{c}_{\iota(i), j(i)}$ denote the corresponding vanishing curve,
where $\iota(i):=[\frac{i}{3}]$.
Since Table~\ref{tab:data} gives $j(0)=1$ and
\[j(i+1)\equiv j(i)+1 \pmod 3,\] we have the monodromy factorization
\[\Phi_O \cdot \prod_{\iota=7}^0(t_{\tilde{c}_{\iota,2}}t_{\tilde{c}_{\iota,0}}t_{\tilde{c}_{\iota,1}})\]
when we simply write $\iota(i)$ as $\iota$.
Note that the curves $\tilde{c}_{\iota,j}$ do not intersect for a fixed $\iota$.
Also, the curves $\tilde{c}_{\iota(i),j(i)}$ and $\tilde{c}_{\iota(i'),j(i')}$ do not intersect for $|i'-i|\le 2$.
Using the commutativity of Dehn twists along disjoint curves, 
we can therefore rewrite the factorization as
\begin{align*}
&\Phi_O \cdot (t_{\tilde{c}_{7,2}}t_{\tilde{c}_{6,2}})
(t_{\tilde{c}_{7,0}}t_{\tilde{c}_{6,0}}t_{\tilde{c}_{5,0}})
(t_{\tilde{c}_{7,1}}t_{\tilde{c}_{6,1}}t_{\tilde{c}_{5,1}}t_{\tilde{c}_{4,1}})
(t_{\tilde{c}_{5,2}}t_{\tilde{c}_{4,2}}t_{\tilde{c}_{3,2}}t_{\tilde{c}_{2,2}})\\
&(t_{\tilde{c}_{4,0}}t_{\tilde{c}_{3,0}}t_{\tilde{c}_{2,0}}t_{\tilde{c}_{1,0}})
(t_{\tilde{c}_{3,1}}t_{\tilde{c}_{2,1}}t_{\tilde{c}_{1,1}}t_{\tilde{c}_{0,1}})
(t_{\tilde{c}_{1,2}}t_{\tilde{c}_{0,2}}t_{\tilde{c}_{0,0}}).
\end{align*}
By applying a Hurwitz move to $t_{\tilde{c}_{1,2}}t_{\tilde{c}_{0,2}}t_{\tilde{c}_{0,0}}$, and letting
$\tilde{c}'_{i,j}:=\Phi_O^{-1}(\tilde{c}_{i,j})$,
we observe that since $\tilde{c}'_{0,0}$ does not intersect with $\tilde{c}_{7,2}$ and $\tilde{c}_{6,2}$, we have
\begin{align*}
&\Phi_O \cdot (t_{\tilde{c}'_{1,2}}t_{\tilde{c}'_{0,2}}t_{\tilde{c}_{7,2}}t_{\tilde{c}_{6,2}})
(t_{\tilde{c}'_{0,0}}t_{\tilde{c}_{7,0}}t_{\tilde{c}_{6,0}}t_{\tilde{c}_{5,0}})
(t_{\tilde{c}_{7,1}}t_{\tilde{c}_{6,1}}t_{\tilde{c}_{5,1}}t_{\tilde{c}_{4,1}})\\
&(t_{\tilde{c}_{5,2}}t_{\tilde{c}_{4,2}}t_{\tilde{c}_{3,2}}t_{\tilde{c}_{2,2}})
(t_{\tilde{c}_{4,0}}t_{\tilde{c}_{3,0}}t_{\tilde{c}_{2,0}}t_{\tilde{c}_{1,0}})
(t_{\tilde{c}_{3,1}}t_{\tilde{c}_{2,1}}t_{\tilde{c}_{1,1}}t_{\tilde{c}_{0,1}})
\end{align*}
After reindexing the curves so that each parenthesized block is labeled consecutively,
we recover the factorization \eqref{eq:genfact_1} for $n=4$.
\end{exam}

\begin{proof}[Proof of \eqref{eq:genfact_1}, $n\neq 3k$ case]
Let $\tilde{c}_i:=\tilde{c}_{\iota(i), j(i)}$, where $\iota(i)=[\frac{i}{3}]$.
The monodromy factorization is given by
\begin{align}
\Phi_O \cdot \prod_{\iota=2n-1}^0 (t_{\tilde{c}_{\iota,2}}t_{\tilde{c}_{\iota,0}}t_{\tilde{c}_{\iota,1}}).
\end{align}
Since $\tilde{c}_{\iota,0},\tilde{c}_{\iota,1},\tilde{c}_{\iota,2}$ do not intersect pairwise, their corresponding Dehn twists commute.
Let $n_0=\left[\frac{n}{3}\right], n_1=\left[\frac{n-2}{3}\right]$, and $n_2=\left[\frac{n}{3}\right]-1$.
Each integer $n_i+1$ counts the number of $l$ values with an angle $\theta$ such that $\frac{2i-3}{3} \pi \le \theta \le \frac{2i-1}{3} \pi$.
As in the $n=3k$ case, we can move words 
\[t_{\tilde{c}_{n_1+n_2+1,2}} \cdots t_{\tilde{c}_{0,2}} \cdot t_{\tilde{c}_{n_2,0}} \cdots t_{\tilde{c}_{0,0}}\]
to the front. Applying a Hurwitz move at those words, and using the commutativity of Dehn twists for non-intersecting curves, the expression can be rewritten as
\begin{align*}
&\Phi_O\cdot 
(t_{\tilde{c}'_{n_1+n_2+1,2}} \cdots t_{\tilde{c}'_{0,2}} \cdot 
t_{\tilde{c}'_{n_2,0}} \cdots t_{\tilde{c}'_{0,0}})\cdot 
(t_{\tilde{c}_{2n-1,2}}\cdots t_{\tilde{c}_{2n-n_0-1,2}})
(t_{\tilde{c}_{2n-1,0}}\cdots t_{\tilde{c}_{n+n_2+1,0}})\cdot \\
&(t_{\tilde{c}_{2n-1,1}}\cdots t_{\tilde{c}_{n,1}}) \cdot 
(t_{\tilde{c}_{2n-n_0-2,2}}\cdots t_{\tilde{c}_{n-n_0-1,2}})\cdot 
(t_{\tilde{c}_{n+n_2,0}}\cdots t_{\tilde{c}_{n_2+1,0}})\cdot (t_{\tilde{c}_{n-1,1}}\cdots t_{\tilde{c}_{0,1}}),
\end{align*}
where $\tilde{c}'_{i,j}=\Phi_O^{-1}(\tilde{c}_{i,j})$.
Since $t_{\tilde{c}'_{n_2,0}} \cdots t_{\tilde{c}'_{0,0}}$ commutes with $t_{\tilde{c}_{2n-1,2}} \cdots t_{\tilde{c}_{2n-n_0-1,2}}$, 
the factorization becomes
\begin{align*}
&\Phi_O\cdot 
(t_{\tilde{c}'_{n_1+n_2+1,2}} \cdots t_{\tilde{c}'_{0,2}} \cdot 
t_{\tilde{c}_{2n-1,2}}\cdots t_{\tilde{c}_{2n-n_0-1,2}})\cdot
(t_{\tilde{c}'_{n_2,0}} \cdots t_{\tilde{c}'_{0,0}}\cdot
t_{\tilde{c}_{2n-1,0}}\cdots t_{\tilde{c}_{n+n_2+1,0}})\cdot \\
&(t_{\tilde{c}_{2n-1,1}}\cdots t_{\tilde{c}_{n,1}}) \cdot 
(t_{\tilde{c}_{2n-n_0-2,2}}\cdots t_{\tilde{c}_{n-n_0-1,2}})\cdot 
(t_{\tilde{c}_{n+n_2,0}}\cdots t_{\tilde{c}_{n_2+1,0}})\cdot (t_{\tilde{c}_{n-1,1}}\cdots t_{\tilde{c}_{0,1}}).
\end{align*}
Note that the first two blocks have lengths
\[(n_1+n_2+2)+(n_0+1)=n\]
and
\[(n_2+1)+(n-n_2-1)=n,\]
and each of the remaining four blocks has length $n$.
After re-indexing curves and applying a negative Dehn twist along the curve $\delta$, we obtain ~\eqref{eq:genfact_1}.
\end{proof}

The following theorem determines the global conjugation in the geometrically obtained monodromy factorization, by using relations in the mapping class group to find a word factorization that is exactly equal to $\sigma_{2n+1}^{2n}$ and whose Dehn twist curves are related by conjugation.

\begin{thm} \label{thm:comb_1}
There is a word factorization of $\sigma_{2n+1}^{2n}$
\begin{align} \label{eq:monofacgen1}
\sigma_{2n+1}^{2n}=t_{\delta} \prod_{i=1}^n (t_{d_i} t_{d'_i}) \left(\prod_{j=6}^1 \prod_{i=n}^1 t_{c_{i,j}} \right) .
\end{align}
Moreover, after applying the left-handed Dehn twist along the curve $d$ on each vanishing cycles obtained in the geometric method,
the monodromy factorization agrees with \eqref{eq:monofacgen1}.
\end{thm}

\begin{proof}
From \cite{I}, a factorization of $\sigma_{2n+1}^{2n}$ is given as $\sigma_{2n+1}^{2n}=(t_1\cdots t_{6n+2})^6$.

Using braid relations and chain substitutions, we will obtain \eqref{eq:monofacgen1} as follows.
Let us write for $1\le i\le n$, $a_i=t_{6i-5} t_{6i-4} t_{6i-3} t_{6i-2} t_{6i-1}$, $b_i=t_{6i+1} t_{6i+2} \cdots t_{6n}$ $(i\neq n)$, $b_n$ is the identity map, 
and $c=t_{6n+1} t_{6n+2}$.
Also simply denote $t_i$ by $i$.
Note that $(6i)^{b_i^{-1}}$ commutes with $a_{i+1}, a_{i+2},\dots, a_{n}$.
Using braid relations and chain relations, we have
\begin{align*}
&(1\cdots (6n+2))^6 \\
=&(a_1 \cdots a_n \cdot (6n) (6n-6)^{b_{n-1}^{-1}} \cdots 6^{b_1^{-1}} c)^6 \\
=& (a_1^6 \cdots a_n^6 \cdot c^6) ((6n)^{c^{-6} a_n^{-5}} (6n-6)^{c^{-6}a_{n-1}^{-5} b_{n-1}^{-1}} \cdots (12)^{c^{-6} a_2^{-5}b_2^{-1}} 6^{c^{-6} a_1^{-5} b_1^{-1}}) \cdot \\
&((6n)^{c^{-5}a_n^{-4}} (6n-6)^{c^{-5}a_{n-1}^{-4}  b_{n-1}^{-1}} \cdots (12)^{c^{-5}a_2^{-4} b_2^{-1}} 6^{c^{-5}a_1^{-4}  b_1^{-1}}) \cdot \\
&\cdots \cdot ((6n)^{c^{-1}} (6n-6)^{c^{-1} b_{n-1}^{-1}} \cdots (12)^{c^{-1} b_2^{-1}} 6^{c^{-1} b_1^{-1}}) .
\end{align*}
Put 
\begin{equation}\label{eq:vc_comb1}
\begin{aligned}
a_i^6&=t_{d_i} t_{d'_i}\\
c^6&=t_{\delta}\\
c_{i,j}&=(c^{-j}a_{i}^{-j+1}b_{i}^{-1})(6i)
\end{aligned}
\end{equation}
for $i=1,2,\dots, n$, $j=1,\dots 6$.
\end{proof}
\section{Construction of $X_2$ and its perturbation}\label{sec:splitgen2}
\subsection{Construction of local charts}
The construction of local charts realizing a regular neighborhood of $F^2_{n}$ is similar to that in Section \ref{sec:splitgen1}.
Let $X_{n,2}= \phi^{-1}_{n}(D^2_-)$ be the preimage of the lower half-disk of $S^2$, and
let $\phi_{n,2}: X_{n,2} \arrowr D^2$ be the restriction of $\phi_{n}$.
We just describe the transition maps between charts, and $\phi_{n,2}$ on each chart.

\begin{equation} \label{eq:chart2}
    \begin{aligned}
    U_{i,1}\cap U&: (\mathfrak{s}_1, \mathfrak{t}_1)=(t^{-1}, t(st^2)), \\
    U_{i,2} \cap U_{i,1}&:(\mathfrak{s}_2, \mathfrak{t}_2)=(\mathfrak{t}_{1}^{-1}, \mathfrak{s}_{1}\mathfrak{t}_{1}^{n+1})\\
    U_{3,1} \cap U&: (\xi_1,\eta_1)=\left(v, \frac{\sqrt[2n+1]{1-(1+u^2)g(u,v)}^{n+1}}{v^2} g(u,v)\right),\\
    U_{3,j}\cap U_{3,j-1}&: (\xi_j, \eta_j)=(\eta_{j-1}^{-1}, \xi_{j-1} \eta_{j-1}^2) \quad (j=2,\dots, 2n+1)
    \end{aligned}
\end{equation}

\begin{equation} \label{eq:phi2}
    \begin{aligned}
    (\phi_{n,2})|_U&=y^{n+1} f(x,y)^{2n+1}\\
    (\phi_{n,2})|_{U_{i,1}}&=\mathfrak{s}_1\mathfrak{t}_1^{n+1},\\
    (\phi_{n,2})|_{U_{i,2}}&=\mathfrak{t}_2, \\
    (\phi_{n,2})|_{U_{3,j}}&=\xi_j^{2n+1-j} \eta_j^{2n+2-j}. \quad (j=1,\dots, 2n+1)
    \end{aligned}
\end{equation}

The perturbed transition maps and $(\phi_{n,2})_{\ep_n}$ on each chart are given as follows.
\begin{equation} \label{eq:charte2}
    \begin{aligned}
    U_{i,1}\cap U&: (\mathfrak{s}_1, \mathfrak{t}_1)=(t^{-1}, st^2), \\
    U_{i,2} \cap U_{i,1}&:(\mathfrak{s}_2, \mathfrak{t}_2)=(\mathfrak{t}_1^{-1}, \mathfrak{s}_1 \mathfrak{t}_1(\mathfrak{t}_1^n- {\ep_n^{2n}})),\\
    U_{3,1} \cap U&: (\xi_1,\eta_1)=(v, g(u,v) \sqrt[2n+1]{1-(1+u^2)g(u,v)}^{n+1}/v^2 ),\\
    U_{3,2}\cap U_{3,1}&:(\xi_2, \eta_2)=(\eta_1^{-1}, \eta_1(\xi_1\eta_1-{\ep_n})),\\
    U_{3,j+1}\cap U_{3,j}&: (\xi_{j+1}, \eta_{j+1})=(\eta_j^{-1}, \eta_j(\xi_j\eta_j+(z_{2n}^{j-2}-z_{2n}^{j-1}) {\ep_n})).\\
    &(j=2,\dots, 2n) 
    \end{aligned}
\end{equation}

\begin{equation}\label{eq:phie2}
    \begin{aligned}
    (\phi_{n,2})_{\ep_n}|_{U_{i,1}}&=\mathfrak{s}_1 \mathfrak{t}_1 (\mathfrak{t}_1^{n}-{\ep_n^{2n}}), &(i=1,2)\\
    (\phi_{n,2})_{\ep_n}|_{U_{i,2}}&=\mathfrak{t}_2, &(i=1,2)\\
    (\phi_{n,2})_{\ep_n}|_{U_{3,1}}&=\eta_1(\xi_1^{2n}\eta_1^{2n}-{\ep_n^{2n}}), &\\
    (\phi_{n,2})_{\ep_n}|_{U_{3,j}}&= \eta_j \prod_{i=j}^{2n-1}(\xi_j\eta_j -(z_{2n}^{j-1}-z_{2n}^j) {\ep_n}),& (j=2,\dots, 2n)\\
    (\phi_{n,2})_{\ep_n}|_{U_{3,2n+1}}&=\eta_{2n+1}.&
    \end{aligned}
\end{equation}

\begin{figure}[h]
\centering
\includegraphics[scale=0.8]{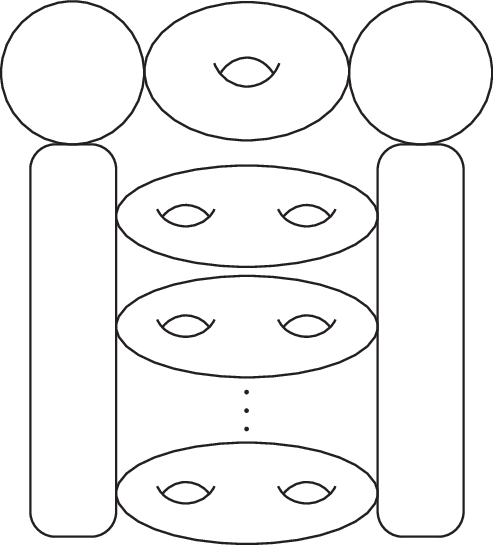}
\caption{Configuration of $((\phi_{n,2})_{\ep_n})^{-1}(0).$}
\label{fig:perturbfiber2}
\end{figure}

\subsection{Critical points and singular values}
For simplicity, we denote $\phi_n=(\phi_{n,2})_{\ep_n} |_U$ in this subsection.
The calculation is analogous to Section~\ref{sec:crit,sing}, so we just describe the conclusion.

\begin{prop}
There are $6n$ critical points in $U$ except the zero fiber, and all critical points are of the form $(0,y)$.
Every critical point lies on distinct singular fiber, except for those on the zero fiber.
For $n=3k$, approximated expressions of critical points $(0,y_{i,j})$ and singular values $t_i$ are given by
    \begin{align}\label{eq:y_i,j, 3}
    y_{i,j}=
    \begin{cases} z_6^{2j+1} (1+\frac{c_n}{3} z_{4n} z_{2n}^i {\ep_n}) & (k \text{ odd}) \\ 
    z_6^{2j+1} (1+\frac{c_n}{3} z_{2n}^i {\ep_n}) & (k \text{ even})
    \end{cases}
    \end{align}
and
    \begin{align}\label{eq:t_i, 3}
    t_{i,j}= -\frac{2nc_n}{2n+1} {\ep_n^{2n+1}}\cdot z_6^{2j+1} \cdot 
    \begin{cases} (z_{4n}z_{2n}^i +\frac{c_n{\ep_n}}{3}  z_{2n}z_n^{i})& (k \text{ odd}) \\
    ( z_{2n}^i+\frac{c_n{\ep_n}}{3} z_n^{i}), &(k \text{ even}) \end{cases}
    \end{align}
where \[c_n=\frac{1}{\sqrt[2n]{2n+1}}.\]

For $n \neq 3k$, approximated expressions of critical points
$(0,y_i)$ and singular values $t_i$, $i=0,\dots, 6n-1$, are given by
\begin{align}
y_i=
\begin{cases}
z_6^{2j(i)+1}(1+\frac{c_n}{3} z_{12n} z_{6n}^i \ep_n )& (n \text{ odd}) \\
z_6^{2j(i)+1}(1+\frac{c_n}{3} z_{6n}^i \ep_n), & (n \text{ even}) 
\end{cases}
\end{align}
where $j(i)\in\{0,1,2\}$ is the unique solution of the modular equation
\begin{align} 
2nj \equiv \begin{cases} i+m+1 \pmod 3 & \quad \text{if}  \quad n=2m+1\\
 i+m \pmod 3 & \quad \text{if}  \quad n=2m, \end{cases}
\end{align}
and
\begin{align} \label{eq:t_i, 2}
t_i =\frac{2nc_n}{2n+1} {\ep_n^{2n+1}} \cdot z_6^{2j(i)+1}\cdot 
\begin{cases} z_{12n}z_{6n}^i +\frac{c_n{\ep_n}}{3}  z_{6n}z_{3n}^{i}& (n \text{ odd}) \\
 z_{6n}^i+\frac{c_n{\ep_n}}{3} z_{3n}^{i}, &(n \text{ even}) \end{cases}
\end{align}
\end{prop}

\begin{rmk}
In contrast to the $\sigma_{2n+1}^{-1}$ case, the singular values $t_{i,j}$ corresponding to $y_{i,j}$ are all distinct for $n=3k$, resulting in a total of $6n$ values.

Since $f_{i+2k}=z_3 f_{i}$, it follows that $z_6^{2j+1}f_{i}=z_6^{2j-1}f_{i+2k}=z_6^{2j+3}f_{i+4k}$ and $f^2_{i+2k}=z_3^2 f_i^2$.
Writing \[t_{i,j}=A_{i,j}+\frac{c_n{\ep_n}}{3}B_{i,j}\] with $|B_{i,j}|=1$, we observe that 
\[t_{i+2k,j-1}=A_{i,j}+\frac{c_n{\ep_n}}{3}z_3 B_{i,j}\] and \[t_{i+4k,j+1}=A_{i,j}+\frac{c_n{\ep_n}}{3}z^2_3 B_{i,j}.\]
Thus, $t_{i,j}, t_{i+2k,j-1}$, and $t_{i+4k,j+1}$ lie on the circle $|t-A_{i,j}|=\frac{c_n{\ep_n}}{3}$, and are related by $2\pi/3$-rotations.

This observation may affect the choice of the Hurwitz system in the next subsection. 
However, by Remark~\ref{rmk:vc_disjoint}, the vanishing cycles corresponding to the clustered singular values are mutually disjoint. 
Therefore, when choosing Hurwitz paths, it is irrelevant whether the paths are taken to avoid the other singular values.
\end{rmk}

\begin{figure}[h]
\includegraphics[scale=0.65]{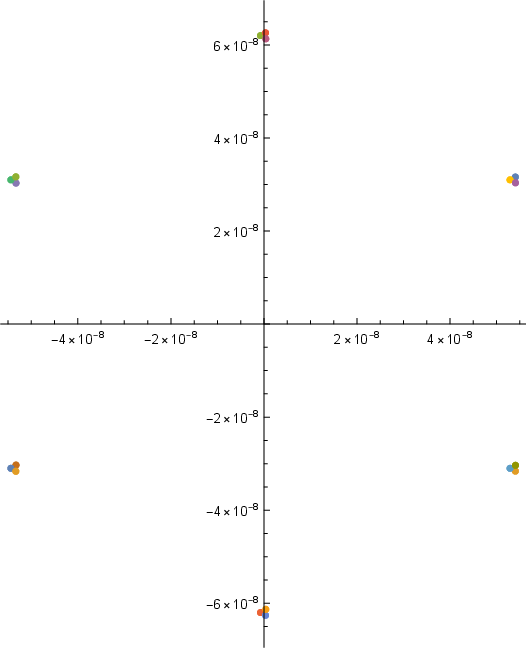}
\includegraphics[scale=0.73]{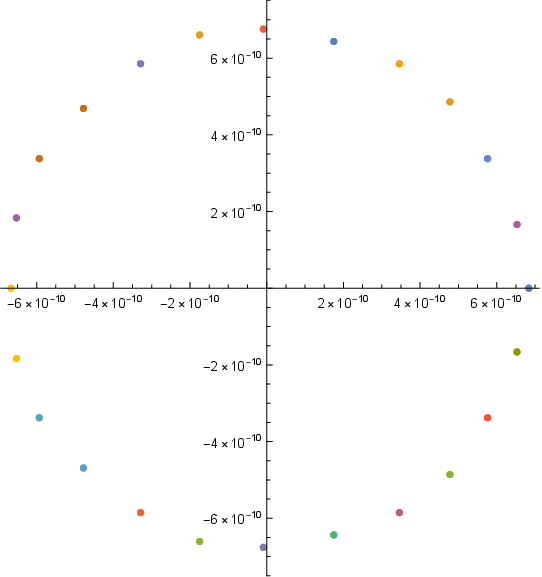}
\caption{Singular values $t_{i,j}$ for $n=3$ (left) and $n=4$ (right).}
\end{figure}

\begin{rmk}
All singular values $t_i$ in \eqref{eq:t_i, 2} are distinct.
If $t_{i}=t_{i'}$ for some $i'\neq i$, it would imply 
\begin{align}\label{eq:ij_condition}
2nj(i)+n+i \equiv 2nj(i')+n+i' \pmod{6n} \Leftrightarrow i'-i\equiv 2n(j(i)-j(i')), \pmod{6n}\end{align}
hence $i'=i\pm 2n$.
When $i'=i+2n$, $j(i+2n)\equiv j(i)+1 \pmod 3$.
Similarly, when $i'=i-2n\equiv i+4n \pmod{6n}$, then $j(i+4n) \equiv j(i)-1 \pmod 3$.
Both cases contradict ~\eqref{eq:ij_condition}.
Therefore, $t_i \neq t_{i'}$ for all $i\neq i'$.
\end{rmk}

\subsection{Vanishing cycles for $n=3k$}
The process is analogous to Section~\ref{sec:vc_3k}.
Consider the two projection maps
\[ \Phi : \C_{(x,y)}^2 \to \C_{(X,y)}^2 \quad\text{and}\quad \pi_y : \mathbb{C}_{(X,y)}^2 \to \mathbb{C}_y\]
defined by $\Phi(x,y)=(x^2,y)=:(X,y)$ and $\pi_y(X,y)=y$.
Then \[\phi_n:\C_{(x,y)}^2 \arrowr \C_t\] is transformed under the map $\Phi$ into 
\[\psi_n: \C_{(X,y)}^2 \arrowr \C_t,\]
that is, 
\[\psi_n(X,y)=y(X-y^3-1)(y^n (X-y^3-1)^{2n}-{\ep_n^{2n}}).\]
Denote \[F_n^t=\phi_n^{-1}(t)\cap U, \quad f_n^t=\psi_n^{-1}(t)\cap \Phi(U), \quad\text{and}\quad \Phi_n^t=\Phi|_{F_n^t}: F_n^t \arrowr f_n^t.\]

Since $\psi_n^{-1}(t)$ is smooth for $t\neq 0$, each fiber $\psi_n^{-1}(t_i)$ is tangent to the $y$-axis, and such tangent point lifts to a Lefschetz singularity $\phi_n^{-1}(t_i)$. 
That is, the two branch points of $\Phi_n^t$ merge into a single point which becomes a critical point, 
as $t$ moves to a singular value.
The monodromy around such a singular value is given by the Dehn twist along an appropriate path connecting these two points.

The behavior of the movement of branch points of $\Phi_n^t$ is the same as Claim~\ref{claim:branch_move}.
Fix a regular value ${t_O}={\ep_n^{4n}}$ and a Hurwitz system $\gamma_{i,j}$ in $\C_t$ as~\eqref{eq:path_system}.

\begin{claim} \label{claim:overlap_2}
Let $y_j^0$ and $y_{i,j}^0$ be the solutions of \[(-y^3-1)(y^n(-y^3-1)^{2n}-{\ep_n^{2n}})=0,\]
that is, \[y=y_j^0:=z_6^{2j+1},\] and
\begin{align}\label{eq:refpoints}
y \approx y_{i,j}^0:= \begin{cases} z_6^{2j+1}(1+\frac{1}{3}z_{4n}z_{2n}^i {\ep_n}) & (k \text{ odd}) \\
 z_6^{2j+1}(1+\frac{1}{3}z_{2n}^i {\ep_n})& (k \text{ even}). \end{cases}
\end{align}
Then the branch points $y_j^0$ and $y_{i,j}^0$ merge at $y_{i,j}$ for each $i,j$ as we move in the fibers over the curve ${\gamma}_{i,j}$.
\end{claim}

\begin{proof}
Define $y_{i,j}(s)$ by
\[ y_{i,j}(s)= \begin{cases} z_6^{2j+1}(1+s \frac{c_n}{3} z_{4n}z_{2n}^i {\ep_n}) & (k \text{ odd}) \\z_6^{2j+1}(1+s \frac{c_n}{3}z_{2n}^i {\ep_n}). & (k \text{ even}) \end{cases}\]
The path $y_{i,j}(s)$ connects $y_j^0$ to $y_{i,j}$.
We claim that the trajectory of $\phi_n(0,y_{i,j}(s))$ is sufficiently close to $\gamma_{i,j}$.
For simplicity of description, assume $k$ is even.
Since $-y_{i,j}(s)^3-1 \approx s \cdot cz_{2n}^i {\ep_n}$, we have
\begin{align*}
\phi_n(0,y_{i,j}(s)) &\approx z_{6}^{2j+1}(1+s \frac{c_n}{3}z_{2n}^i {\ep_n}) (s \cdot c_n z_{2n}^i {\ep_n}) \left( (1+sn \frac{c_n}{3}z_{2n}^i {\ep_n})s^{2n} \frac{1}{2n+1}{\ep_n^{2n}}-{\ep_n^{2n}}\right) \\
&\approx  -s \cdot c_n  z_{6}^{2j+1} z_{2n}^i{\ep_n^{2n+1}} \left( \frac{2n+(1-s^{2n})}{2n+1}+sc_n \frac{(2n+1)-(n+1)s^{2n}}{3(2n+1)} z_{2n}^i {\ep_n} \right).
\end{align*}
The value $\phi_n(0,y_{i,j}(s))$ is dominated by the ${\ep_n^{2n+1}}$- term, and hence sufficiently close to $\gamma_{i,j}$. 
This proves the claim.
\end{proof}

$\Phi_n^t$ is the hyperelliptic double branched cover.
We use one more projection map \[\pi_y:f_n^t \arrowr \C_y,\]
given by projection onto the $y$-coordinate.
By computing the discriminant of $\psi_n-t=0$, we can find the branch points of the map $\pi_y$.

\begin{lem}\label{lem:discriminant_2}
The discriminant $\Delta_y(t)$ of the polynomial equation $\psi_n(X,y)-t=0$ in $X$ is given by
\begin{align}
\Delta_y(t)=(-1)^{\frac{n(n-1)}{2}} y^{2n(n+1)}((2n)^{2n}{\ep_n^{2n(2n+1)}}y^n-(2n+1)^{2n+1} t^{2n}).
\end{align}
\end{lem}

\begin{proof}
The proof is analogous to that of Lemma~\ref{lem:discriminant_1}.
After a parallel transport along $X$, the equation $\psi_n(X,y)-t=0$ can be rewritten as
\[F(X):=y^{n+1}X^{2n+1}-byX-t=0,\]
where $b={\ep_n^{2n}}$.

Since the resultant $R(F,F')$ is
\[R(F,F')=y^{(n+1)(2n+1)}((2n)^{2n}b^{2n+1} y^{n}-(2n+1)^{2n+1}t^{2n}),\]
the discriminant polynomial 
$\Delta_y(t)=(-1)^{\frac{n(n-1)}{2}}y^{-n-1}R(F,F')$
yields the desired formula.
\end{proof}

From Lemma~\ref{lem:discriminant_2}, the branch points $b_l(t)$ of $\pi_y$ are
\begin{align}
b_l(t)=z_n^l D_n {\ep_n^{-2(2n+1)}}t^2,
\end{align}
where \[D_n=\left(\frac{2n+1}{2n}\right)^2 \sqrt[n]{2n+1}.\]

Note that $\pi_y$ is not a branched cover for $t\neq 0$, since $(y=0)\cap f_n^t =\varnothing$.
To solve this problem, we will extend the map $\pi_y$ to a larger fiber.
For convenience of description, we consider the map $\pi_{\hat{y}}: f_n^t \arrowr \C_{\hat{y}}$, where $\hat{y}=1/y$, defined by
\[\pi_{\hat{y}}(X,y)=1/y.\]
The branch points of $\pi_{1/y}$ are in one-to-one correspondence with those of $\pi_y$.
To make $\pi_y$ surjective, we extend the domain to a space corresponding to the fiber \[\phi_n^{-1}(t) \cap (U \cup U_3),\]
where $U_3=\cup_{j=1}^{2n+1} U_{3,j}$.

\begin{lem}
The map $\pi_{\hat{y}}$ extends to 
\[\hat{\pi}_{\hat{y}}: \hat{f}_n^t \arrowr \C_{\hat{y}},\]
where $\hat{f}_n^t \cong \C$ is an extension of $f_n^t$. Thus $\hat{\pi}_{\hat{y}}$ is a $(2n+1)$-fold branched cover, branched at $\hat{y}=1/b_l(t)$ and $\hat{y}=0$.
\end{lem}

The strategy of the proof is as follows. 
First, we extend the hyperelliptic map $\Phi$ over the fiber \[\hat{F}_n^t:=(\phi_n |_{U\cup U_{3}})^{-1}(t).\] 
This allows us to define the corresponding extended fiber $\hat{f}_n^t$, which is an extension of $f_n^t$ up to homeomorphism.
Since $\hat{f}_n^t$ cannot be expressed directly by equations on the original charts, 
we introduce a new chart system that describes $\hat{f}_n^t$ for each $t$ as the zero set of polynomials.
We then transform the value $1/y$ (or $z/y$ near infinity) into a function of the variables in these new charts.
By analyzing the behavior as these values approach zero, we show that the local multiplicities of the functions are either $1$ or $2$.
These points correspond to the ramification points of $\hat{\pi}_{\hat{y}}$, 
which implies that $\hat{y}=0$ is the branch point of $\hat{\pi}_{\hat{y}}$.

\begin{proof}
Note that since the branch points of $\pi_y$ and $\pi_{\hat{y}}$ coincide, $\hat{y}=1/b_l(t)$ are indeed the branch points of $\pi_{\hat{y}}$.

Before constructing the extension $\hat{f}_n^t$, we extend the hyperelliptic action $\Phi_n^t: F_n^t \arrowr f_n^t$ to the charts $U_{3,j}$.
Under $\Phi_n^t$, the coordinate $u=z/x$ is identified with $-u$, while $u/v=y/z$ is fixed. 
This implies that $v$ is identified with $-v$, and consequently, $\xi_1$ is identified with $-\xi_1$.
Since $g(-u,-v)=g(u,v)$, the coordinate $\eta_1$ is fixed under the map $\Phi_n^t$.
Thus, $\Phi_n^t$ extends to $U_{3,1}$ as the map $\hat{\Phi}_n^t|_{U_{3,1}}: (\xi_1, \eta_1) \mapsto (\Xi_1=\xi_1^2, \eta_1)$.

Consider the fiber defined on $U_{3,1}$: 
\[\phi_n^{-1}(t)|_{U_{3,1}}=\{ \eta_1 \prod_{i=0}^{n-1}(\xi_1^2 \eta_1^2- z_{n}^i {\ep_n^{2}})=t\} . \]
This fiber is clearly preserved under $\Phi_n^t$.
Note that $\phi_n^{-1}(t)|_{U_{3,1}} \cong \phi_n^{-1}(0)|_{U_{3,1}}$ consists of a disk $(\eta_1=0)$ and 
$2n$ annuli $A_i:=(\xi_1 \eta_1=z_{2n}^i {\ep_n})$ for $i=0,\dots, 2n-1$. 
The annuli $A_i$ and $A_{i+n}$ are identified under the map $\Phi_n^t$.
Each annulus $A_i$ ($A_{i+n}$, resp.) is connected to a disk $D_i$ ($D_{i+n}$, resp.) in $U_{3,i+2}$ ($U_{3,i+n+2}$, resp.).
Consequently, we extend $\Phi_n^t$ to $\hat{\Phi}_n^t$ on $\cup_{i=1}^{2n+1} U_{3,i}$ by identifying $D_i$ with $D_{i+n}$ for each $i=0,\dots, n-1$.
We denote the identified annulus and disk by $\hat{A}_i$ and $ \hat{D}_i$, respectively.


To define a map $\hat{\pi}_{\hat{y}}:\hat{f}_n^t \arrowr \C_{\hat{y}}$ as polynomial equations, 
we introduce adjusted charts on $\cup_{i=1}^{2n+1} U_{3,i}$.
Since
\[\hat{f}_n^t\cap U_{3,1}=\{(\Xi_1, \eta_1)\mid \eta_1 \prod_{i=0}^{n-1} (\Xi_1 \eta_1^2 -z_n^i {\ep_n^{2}})=t  \},\]
we set $U'_{3,2}=\{(\Xi_2, \eta_2)\}$ with the transition map \[(\Xi_2, \eta_2)=(\eta_1^{-1}, \eta_1(\Xi_1 \eta_1^2 -{\ep_n^{2}})).\]
Inductively, for $j=2,\dots, n$, we define $U'_{3,j+1}=\{(\Xi_{j+1}, \eta_{j+1})\}$ with
\[(\Xi_{j+1}, \eta_{j+1})=(\eta_j^{-1}, \eta_j(\Xi_j \eta_j +(z_n^{j-2}-z_n^{j-1}){\ep_n^{2}})).\]
On each $U'_{3,j}$, $j=2, \dots, n$, the map $\psi_n$ is given by
\[(\psi_n)|_{U'_{3,j}}=\eta_j \prod_{i=j-1}^{n-1} (\Xi_j \eta_j+(z_n^{j-2}-z_n^{i}){\ep_n^{2}}),\]
so the disk $(\eta_j=0)$ corresponds to the disk $\hat{D}_{j-2}$.
Finally, the map $\psi_n$ on the chart $U'_{3,n+1}$ is given by
\[(\psi_n)|_{U'_{3,n+1}}=\eta_{n+1}.\]

Next, we express $z/y=u/v$ in each $U'_{3,j}$.
Since the transition map $(u,v)\mapsto (\xi_1=v, \eta_1)$ is a diffeomorphism, 
there exists a $C^\infty$-function $h(\xi_1, \eta_1)$ such that $(u,v)=(h(\xi_1, \eta_1), \xi_1)$.
Observing that \[\xi_1^2 \eta_1 \approx g(u,v)=\frac{1}{u^2} \left(1-\frac{v}{u}v^2 \right)-1=\frac{1}{h(\xi_1, \eta_1)^2}\left(1-\frac{v}{u}\xi_1^2\right)-1, \]
we have
\begin{align} \label{eq:u/v}
\frac{u}{v}=\frac{\xi_1^2}{1-(\xi_1^2 \eta_1+1)h(\xi_1, \eta_1)^2}.
\end{align}
Using the transition map from $(X, y)$ to $(\Xi_1, \eta_1)$, and letting \[\hat{h}(\Xi_1, \eta_1)={h}(\xi_1^2, \eta_1)^2,\]
we rewrite \eqref{eq:u/v} as
\begin{align} \label{eq:u/v_revised}
\hat{\pi}_{1/y}=\frac{1}{y}=\frac{u}{v}=\frac{\Xi_1}{1-(\Xi_1 \eta_1+1)\hat{h}(\Xi_1, \eta_1)}.
\end{align}
As $u \arrowr 0$ when $\xi_1 \arrowr 0$, the denominator of \eqref{eq:u/v_revised} remains non-zero, 
ensuring that $\hat{\pi}_{1/y}$ goes to $0$ linearly.
Furthermore, since
\begin{align*}
\Xi_1&= \Xi_2^2( \Xi_2 \eta_2 +{\ep_n^{2}}) \\
=&\Xi_3^2 (\Xi_3 \eta_3 - (1-z_n){\ep_n^{2}})^2(\Xi_3 \eta_3 +z_n {\ep_n^{2}}) \\
=&\cdots \\
=& \Xi_{n+1}^2 (\Xi_{n+1}\eta_{n+1}-(z_n^{n-2}-z_n^{n-1}){\ep_n^{2}})^2  \cdots  (\Xi_{n+1}\eta_{n+1}-(1-z_n)\ep_n^2)^2 (\Xi_{n+1}\eta_{n+1}+z_n^{n-1}{\ep_n^{2}}),
\end{align*}
the ramification points $O_{j-1}$ of $\hat{\pi}_{1/y}$ over $\hat{y}=0$ are given by $\Xi_{j}=0$ with index $2$ for $j=2,\dots, n+1$, and $\Xi_1=0$ with index $1$.
\end{proof}

Write $\hat{b}_l(t)=1/b_l(t)$ as
\[\hat{b}_l(t)=z_n^l D_n^{-1} {\ep_n^{2(2n+1)}} t^{-2},\]
so that $\arg{\hat{b}_l(t_O)} \approx 2\pi l /n$.


The lifting of the path $1/y_{i,j}(s)$ in $\C_{\hat{y}}$ to the fiber $\hat{f}_n^t$ is analogous to the construction in the preceding section.
Since the two points $(0,1/y_{i,j}^0)$ and $(0,1/y_j^0)$ lie on different sheets of the branched cover $\hat{\pi}_{\hat{y}}$, 
the path $1/y_{i,j}(s)$ must intersect the branch point $\hat{b}_l(\gamma_{i,j}(s))$ for some $s$ in $\C_{\hat{y}}$. 
Furthermore, there exists a unique integer $l=l(i,j)$ such that the branch point $\hat{b}_{l(i,j)}(\gamma_{i,j}(s))$ approaches $1/y_{i,j}(1)$ as $s$ goes to $1$.

\begin{lem}
For each $i,j$, there exists a unique integer $l= l(i,j)$ such that $\hat{b}_l(\gamma_{i,j}(1))$ is sufficiently close to $1/y_{i,j}(1)$, 
where $l(i,j)$ is given by the modular equation
\begin{align}
l(i,j)\equiv i+kj+\left[\frac{k+1}{2}\right] \pmod{n}.
\end{align}
\end{lem}

\begin{proof}
We show this for the case where $k$ is even.
Since \[t_{i,j}^2=D_n^{-1}{\ep_n^{2(2n+1)}} z_3^{2j+1}(z_n^i +2c_n {\ep_n}/3 z_{2n}^{3i}+ c_n^2{\ep_n^{2}}/9 z_n^{2i}) \approx D_n^{-1}{\ep_n^{2(2n+1)}} z_3^{2j+1}z_n^i, \]
the value of $\hat{b}_l(t_{i,j})$ is approximately given by
\[\hat{b}_l(t_{i,j}) \approx z_n^{l-(i+2kj+k)}.\]
Comparing this with $1/y_{i,j}(1) \approx z_6^{-(2j+1)}$, we obtain the relation 
\[2l \equiv 2i +2kj +k \pmod{2n},\]
which implies that \[l \equiv i +kj + \frac{k}{2} \pmod{n},\]
thus proving the lemma.
\end{proof}

\begin{rmk} \label{rmk:vc_disjoint}
The singular values $t_{i,j}, t_{i+2k, j-1}$, and $t_{i+4k, j+1}$ lie on the circle $|t-A_{i,j}|=\frac{c_n{\ep_n}}{3}$, which means that they are closely positioned relative to the other singular values. 
Since $l(i+2k,j-1)=l(i,j)+k$ and $l(i+4k, j+1)=l(i,j)-k$, the corresponding vanishing cycles pass through distinct ramification points $B_{l}$ corresponding to $b_l$.
This implies that these vanishing cycles do not intersect.
\end{rmk}

Let $c_{i,j}$ denote the vanishing cycle in $\hat{F}_n^{t_O}$ corresponding to the singular value $t_{i,j}$, 
and let $\tilde{c}_{i,j}$ be the image of $c_{i,j}$ under the map $\hat{\Phi}_n^{t_O}$.
This curve $\tilde{c}_{i,j}$ connects the points $a_{i,j}:=(0,{y}_{i,j}^0)$ and $a_j:=(0,{y}_j^0)$. 
Since $\hat{\Phi}_n^{t_O}$ is a hyperelliptic map branched at $a_{i,j}, a_j$, and $O_n$, the curve $\tilde{c}_{i,j}$ can be uniquely lifted to $c_{i,j}$. 
Thus, the main task is to characterize the curve $\tilde{c}_{i,j}$.

The curve $\tilde{c}_{i,j}$ is determined by the trajectory of the branch points $\hat{b}_l(\gamma_{i,j})$. 
Figure~\ref{fig:move_branchpt2} illustrates the movement of $\hat{b}_{l(i,1)}(\overline{\gamma}_i)$ in the $\C_{\hat{y}}$-plane.

\begin{figure}[ht]
\centering
\includegraphics{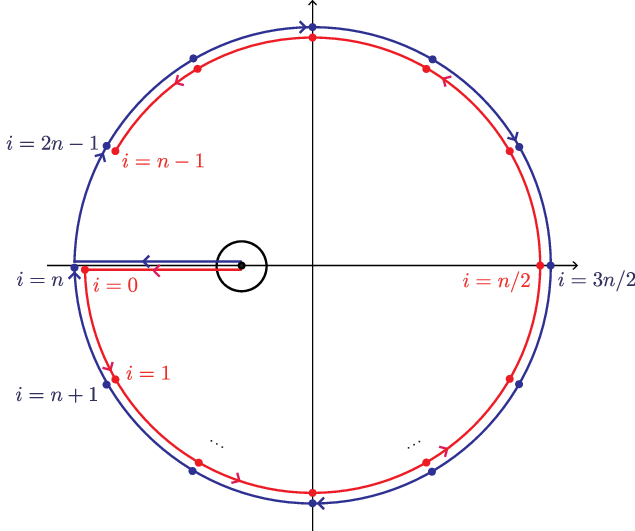}
\caption{Movement of $\hat{b}_{l(i,1)}(\overline{\gamma}_i)$. The blue arrow represents the curve for $n\le i\le 2n-1$, and the red for $0 \le i \le n-1$}
\label{fig:move_branchpt2}
\end{figure}

Note that the configuration of moving branch points is similar to that in Figure~\ref{fig:move_branchpt}, though the indices differ. 
Specifically, the winding angle of $b_{l(i,1)}$ has the maximum at $i=n$, and the minimum at $i=0$.
Considering the movement of $b_{l(i,j)}$, the branched cover structure $\hat{\pi}_{\hat{y}}:\hat{f}_n^{t_O} \arrowr \C_{\hat{y}}$ is depicted in Figure~\ref{fig:branched_cover_n_2}.

\begin{figure}[ht]
\includegraphics[scale=0.6]{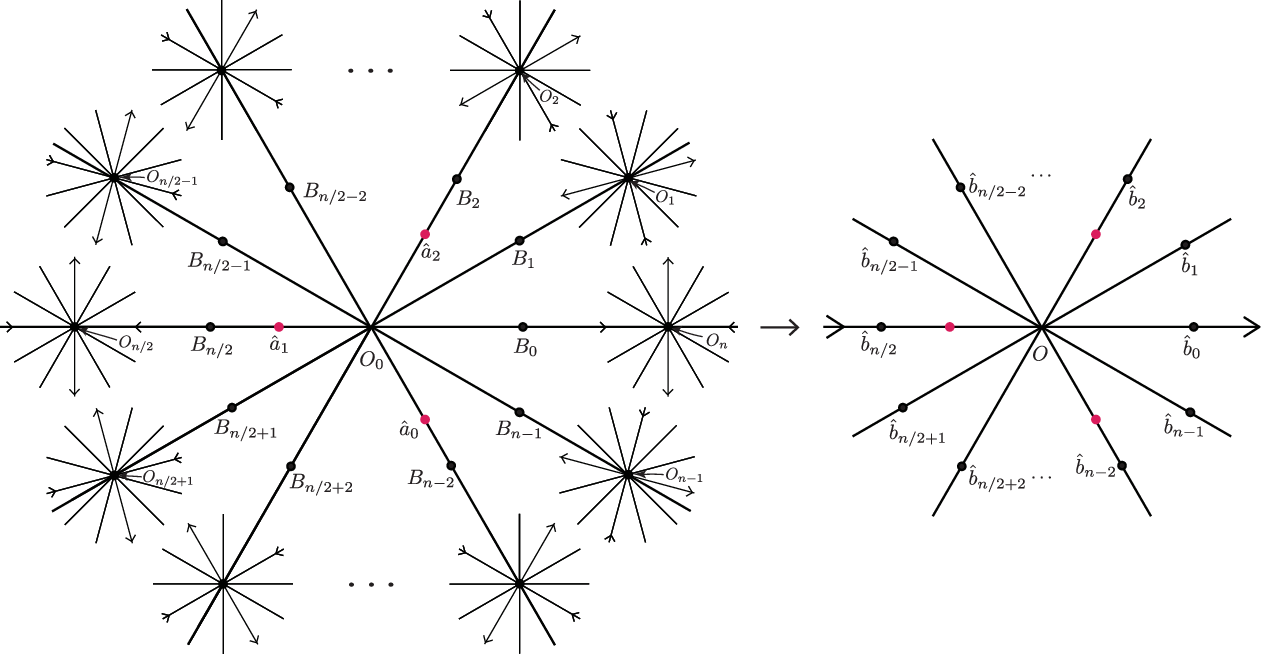}
\caption{Branched cover $\hat{\pi}_{\hat{y}}: \hat{f}_n^{t_O} \arrowr \C_{\hat{y}}$.}
\label{fig:branched_cover_n_2}
\end{figure}

Up to a re-indexing, the branched cover structure $\hat{\pi}_{\hat{y}}$ and the curves $\tilde{c}_{i,j}$ are identical to those in the $\sigma_n^{-1}$ case.
However, the vanishing cycles $c_{i,j}$ themselves are distinct, since $O_0$ is a branch point for $\hat{\Phi}_n^{t_O}: \hat{F}_n^{t_O} \arrowr \hat{f}_n^{t_O}$, 
though it is not involved in the $\sigma_n^{n-1}$-case.

Finally, we observe that $\tilde{c}_{i,j}, \tilde{c}_{i+2k, j-1}$, and $\tilde{c}_{i+4k, j+1}$ are mutually disjoint. 
Consequently, the Dehn twists along these curves commute with each other.

\begin{proof}[Proof of ~\eqref{eq:genfact_2}, $n=3k$ case]
Apply a perturbation to $\hat{f}_n^{t_O}$ such that all $\Phi_n^{t_O}$-branch points are aligned in a line.
Under this perturbation, $B_{l}$ moves clockwise around $O_0$
such that the points $B_{n/2-1}, B_{n/2-2}, \dots, B_0, B_{n-1},\dots B_{n/2}$ are aligned on the line, ordered from left to right.
Also, in this perturbation, the points $a_{i,1}, a_{i-k,2}, a_{i-2k,0}, a_{i-3k,1}, a_{i-4k,2}$, and $a_{i-5k, 0}$ (sharing the same $l(i,j)$) 
are aligned from left to right.

Considering the Hurwitz system $\gamma_{i,j}$ ordered counter-clockwise starting from $\gamma_{3k,1}$, 
we obtain the factorization
\begin{align*}
&\Phi_O \cdot \left( \prod_{i\in I_2} t_{\tilde{c}_{i,1}}t_{\tilde{c}_{i-2k,2}}t_{\tilde{c}_{i-4k,0}} \right) \cdots
\left( \prod_{i\in I_0} t_{\tilde{c}_{i,1}}t_{\tilde{c}_{i-2k,2}}t_{\tilde{c}_{i-4k,0}} \right)\cdot \\
&\left( \prod_{i\in I_5} t_{\tilde{c}_{i,1}}t_{\tilde{c}_{i-2k,2}}t_{\tilde{c}_{i-4k,0}} \right) \cdots
\left( \prod_{i\in I_3} t_{\tilde{c}_{i,1}}t_{\tilde{c}_{i-2k,2}}t_{\tilde{c}_{i-4k,0}} \right), 
\end{align*}
where $I_\iota=[\iota k,(\iota+1)k) \cap \Z$, and the order of each product is taken in the reverse order relative to the standard one.
Let 
\begin{align}\Phi_{\iota,j}:=\prod_{i \in I_\iota} t_{\tilde{c}_{i, j}}.
\end{align} 
Using the commutativity of Dehn twists along non-intersecting curves, we rearrange the factorization as
\begin{align*}
&\Phi_O \cdot (\Phi_{4,0} \Phi_{0,2}\Phi_{2,1})(\Phi_{3,0} \Phi_{5,2}\Phi_{1,1})(\Phi_{2,0} \Phi_{4,2}\Phi_{0,1})
(\Phi_{1,0} \Phi_{3,2}\Phi_{5,1})(\Phi_{0,0} \Phi_{2,2}\Phi_{4,1})(\Phi_{5,0} \Phi_{1,2}\Phi_{3,1})\\
=&\Phi_O \cdot (\Phi_{4,0} \Phi_{0,2}\Phi_{2,1})(\Phi_{3,0} \Phi_{5,2}\Phi_{1,1})(\Phi_{2,0} \Phi_{4,2}\Phi_{0,1})
(\Phi_{1,0} \Phi_{3,2}\Phi_{5,1})( \Phi_{2,2}\Phi_{4,1})(\Phi_{3,1})(\Phi_{0,0}\Phi_{5,0} \Phi_{1,2}).
\end{align*}
Since $\Phi_{i,j}$ commutes with $\Phi_{i-3, j-1}$, we rearrange it as
\[\Phi_O \cdot (\Phi_{4,0})( \Phi_{0,2}\Phi_{5,2})(\Phi_{2,1}\Phi_{1,1}\Phi_{0,1})(\Phi_{3,0} \Phi_{2,0}\Phi_{1,0} )
(\Phi_{4,2} \Phi_{3,2} \Phi_{2,2})(\Phi_{5,1}\Phi_{4,1}\Phi_{3,1})(\Phi_{0,0}\Phi_{5,0} \Phi_{1,2}).\]

Applying a Hurwitz move to $\Phi_{0,0}\Phi_{5,0} \Phi_{1,2}$, we obtain 
\[\Phi_O (\Phi'_{0,0}\Phi'_{5,0} \Phi'_{1,2}) \cdot (\Phi_{4,0})( \Phi_{0,2}\Phi_{5,2})(\Phi_{2,1}\Phi_{1,1}\Phi_{0,1})(\Phi_{3,0} \Phi_{2,0}\Phi_{1,0} )
(\Phi_{4,2} \Phi_{3,2} \Phi_{2,2})(\Phi_{5,1}\Phi_{4,1}\Phi_{3,1}),\]
where $\Phi'_{i,j}:=\Phi_O^{-1} \Phi_{i,j}\Phi_O$.
Since $\Phi'_{1,2}$ commutes with $\Phi_{4,0}$, we finally have
\begin{align}
\Phi_O (\Phi'_{0,0}\Phi'_{5,0} \Phi_{4,0})(\Phi'_{1,2} \Phi_{0,2}\Phi_{5,2})(\Phi_{2,1}\Phi_{1,1}\Phi_{0,1})(\Phi_{3,0} \Phi_{2,0}\Phi_{1,0} )
(\Phi_{4,2} \Phi_{3,2} \Phi_{2,2})(\Phi_{5,1}\Phi_{4,1}\Phi_{3,1}).
\end{align}
After reindexing the curves $\tilde{c}_{i,j}$ and $\tilde{c}'_{i,j}$, where $\tilde{c}'_{i,j}=\Phi_O^{-1}(\tilde{c}_{i,j})$, and 
taking the global conjugation, which is the negative Dehn twist along a curve $d$, we obtain \eqref{eq:genfact_2} with curves depicted in Figure~\ref{fig:vc}.
To find the global conjugation, see the proof of Theorem~\ref{thm:comb_2}.
\end{proof}

\subsection{Vanishing cycles for $n\neq 3k$}
As in the $n=3k$ case, we will show that along the path system $\gamma_i:=\gamma_{i, j(i)}$, defined in \eqref{eq:path_system}, 
the branch point $y_{j(i)}^0$ moves to $y_{i}$.

Recall that $t_i$ is given in \eqref{eq:t_i, 2}.
If we define $\theta_i=6n\arg(t_i)$, then we have 
\begin{align} \label{eq:arg_t_i}
\theta_i\equiv i+2nj(i)+n \pmod{6n}.
\end{align}
Note also that by choosing a reference fiber $f_n^{t_O}$, where $t_O={\ep_n^{4n}}$, sufficiently close to the fiber over $t=0$, 
so that the intersection $f_n^{t_O}\cap \{X=0\}$ is given as the solution to $\psi_n(0,y)=0$:
$y_j^0=z_6^{2j+1}$, and 
\[y_{i,j(i)}^0=z_6^{2j(i)+1} \cdot 
\begin{cases} 1+\frac{1}{3} z_{12n}z_{6n}^i {\ep_n} & (n: \text{odd}) \\ 1+\frac{1}{3}z_{6n}^i {\ep_n} & (n: \text{even}) 
\end{cases}\]

\begin{claim}
The points $y_{j(i)}^0$ and $y_{i,j(i)}^0$ merge as they move along the curve $\gamma_i$.
\end{claim}
\begin{proof}
Define a path $y_{i}(s)$ by
\[y_{i}(s)=z_6^{2j(i)+1} \cdot 
\begin{cases} 1+s \frac{c_n}{3} z_{12n}z_{6n}^i {\ep_n} & (n: \text{odd}) \\ 1+s \frac{c_n}{3}z_{6n}^i {\ep_n} & (n: \text{even}). \end{cases}\]
Then the path $y_{i}(s)$ connects $y_{j(i)}^0$ and $y_{i,j(i)}$.
A calculation analogous to those in Claim~\ref{claim:overlap} and Claim~\ref{claim:overlap_2} proves that the trajectory of $\phi_n(0,y_i(s))$ is sufficiently close to $\gamma_{i,j(i)}$.
\end{proof}

Since $y_{j(i)}^0$ and $y_{i,j(i)}^0$ lie on different sheets, there exists a branch point $b_l(\gamma_i(s))$ that passes through $y_i (s)$ for some $s$.
Thus, we must track the trajectory of $b_l$ to determine the vanishing cycles.
As in the preceding subsection, we analyze $\hat{b}_l$.
Note that 
\begin{align*}
\hat{b}_l(t_i)\approx \begin{cases} z_n^l z_{6n}^{-1} z_{3n}^{-i} z_3^{-(2j(i)+1)} & (n: \text{odd}) \\
z_n^l z_{3n}^{-i} z_3^{-(2j(i)+1)}. & (n: \text{even}) \end{cases}
\end{align*}
Since this value must be sufficiently close to $z_6^{-(2j(i)+1)}$, we obtain
\[6l\equiv 2i+ 2nj(i)+2\left[\frac{n+1}{2}\right] \pmod{6n},\]
which simplifies to
\begin{align} \label{eq:3l_equation}
3l \equiv i+nj(i)+\left[\frac{n+1}{2}\right] \pmod{3n}.
\end{align}
In Claim \ref{claim:subsurface2}, we have proved that such an integer $l$ is uniquely determined modulo $n$.
Consequently, there exists a well-defined function $l(i)\in \{0,\dots, n-1\}$ satisfying \eqref{eq:3l_equation}.

Since the functions $j(i), l(i)$ are identical to those in the $\sigma_{2n+1}^{2n}$ case, we refer to Table~\ref{tab:data} for the values of $j(0), l(0), \Delta j, \Delta l$.
However, a complicated behavior arises because the singular values $t_i$ are not arranged in cyclic order with respect to the index $i$.
In particular, when $n\equiv 1 \pmod 3$, $\theta_{i+4n+1}$ is given as
\begin{align*}
\theta_{i+4n+1} & = i+4n+1+2nj(i+4n+1)+n \\
			& \equiv i+4n+2n(j(i)-1-1)+n+1 \pmod{3n} \\
			& \equiv \theta_i +1 \pmod{3n}
\end{align*}
which implies that $t_{i+4n+1}$ is adjacent to $t_{i}$.
A similar calculation shows that when $n\equiv 2 \pmod 3$, $t_{i+2n+1}$ is adjacent to $t_i$.
Since we take a Hurwitz system on $\C_t$ in the counter-clockwise direction starting from $t_0$, 
the ordering of the indices of $t_i$ in this Hurwitz system should be considered when finding the corresponding monodromy factorization.

\begin{proof}[Proof of \eqref{eq:genfact_2}]

The proof is analogous to that of \eqref{eq:genfact_1}.
The only difference is that the singular values are not ordered cyclically according to the index $i$.

Define
\[
r=
\begin{cases}
4n+1 & (n\equiv 1 \pmod 3),\\
2n+1 & (n\equiv 2 \pmod 3).
\end{cases}
\]
As observed above, $t_{i+r}$ is adjacent to $t_i$ in the cyclic ordering of singular values.
Let
\[
i_q \equiv i_0+qr \pmod{6n},
\]
where $i_0$ is chosen so that $\arg t_{i_0}=-\pi$.
Re-index the vanishing cycles according to this cyclic order.

After this re-indexing, the local configuration of the vanishing arcs and the disjointness relations among the corresponding vanishing cycles are identical to those appearing in the proof of \eqref{eq:genfact_1}.
In particular, the same block decomposition of the monodromy factorization is obtained, and the same commutation relations between Dehn twists hold.

Therefore, the identical sequence of commutations and Hurwitz moves used in the proof of \eqref{eq:genfact_1} applies here as well.
This produces the desired factorization up to a global conjugation.

Finally, applying the global conjugation given by the left-handed Dehn twists along the curves $d$ and $d'$, we obtain \eqref{eq:genfact_2}.
\end{proof}

Finally, we can figure out the possibly involved global conjugation, using relations in the mapping class groups.
\begin{thm}\label{thm:comb_2}
There is a word factorization of $\sigma_{2n+1}$
\begin{align} \label{eq:comb_2}
\sigma_{2n+1}=(t_{d} t_{d'})^2 \left(\prod_{j=12}^7 \prod_{i=n}^1 t_{c_{i,j}} \right) \cdot \left( \prod_{i=1}^n t_{d_{i}}t_{d'_{i}} \right).
\end{align}
Moreover, after applying the left-handed Dehn twists along the curves $d$ and $d'$ to each vanishing cycle obtained in the geometric method,
the monodromy factorization agrees with \eqref{eq:comb_2}.
\end{thm}

\begin{proof}
From \cite{I}, a monodromy factorization of $\sigma_{2n+1}$ is known as \[\sigma_{2n+1}=(1\cdots (6n+2))^{12n}.\]
Write $C_n=(1\cdot 2 \cdots n)$, $\tilde{C}_n=(n \cdot (n-1) \cdots 1)$.
Then by the odd chain relations, we have 
\begin{equation}
\begin{aligned} \label{eq:dd'}
t_{d}t_{d'}&= C_{6n-1}^{6n}\\
t_{d_{i}}t_{d'_i}&=((6i-5)\cdots (6i-1))^6.
\end{aligned}
\end{equation}
First, consider $C_{6n+2}^{6n+2}=(1\cdots (6n+2))^{6n+2}$.
\begin{align*}
&(1\dots (6n+2))^{6n+2}\\
=&C_{6n+1}C_{6n}C_{6n-1}^{6n} \cdot \tilde{C}_{6n} \tilde{C}_{6n+1}\tilde{C}_{6n+2} \\
=&C_{6n+1}C_{6n} \cdot (t_{d}  t_{d'}) \cdot (6n) \cdot \tilde{C}_{6n-1} \cdot (6n+1)(6n)\cdot \tilde{C}_{6n-1} \cdot (6n+2)(6n+1)(6n)\cdot \tilde{C}_{6n-1} \\
=&C_{6n+1}C_{6n} \cdot (t_{d} t_{d'}) \cdot (6n+1)(6n+2)(6n+1)\cdot (t_{6n+2}^{-1}t_{6n+1}^{-1})(6n) \\
&\cdot \tilde{C}_{6n-1}  \cdot t_{6n+1}^{-1}(6n)\cdot \tilde{C}_{6n-1} \cdot (6n)\cdot \tilde{C}_{6n-1}\\
=&C_{6n+1}C_{6n} \cdot (t_{d} t_{d'}) \cdot (6n+2)(6n+1)(6n+2)\cdot (t_{6n+2}^{-1}t_{6n+1}^{-1})(6n)  \\
&\cdot \tilde{C}_{6n-1} \cdot t_{6n+1}^{-1}(6n)\cdot \tilde{C}_{6n-1} \cdot (6n) \cdot \tilde{C}_{6n-1}\\
=&C_{6n+2}^2 \cdot (t_{d} t_{d'})  \cdot t_{{c}'_{1,9}} \cdot \tilde{C}_{6n-1}\cdot  t_{{c}'_{1,8}}\cdot \tilde{C}_{6n-1} \cdot t_{{c}'_{1,7}} \cdot \tilde{C}_{6n-1},
\end{align*}
where ${c}'_{1,7},  {c}'_{1,8},$ and ${c}'_{1,9}$ are defined by
\begin{align}
{c}'_{1,7}&=6n,\\ 
{c}'_{1,8}&=t_{6n+1}^{-1}(6n), \\ 
{c}'_{1,9}&=(t_{6n+2}^{-1}t_{6n+1}^{-1})(6n). 
\end{align}

Also, $C_{6n+2}^{6n}$ is given as
\begin{align*}
&C_{6n+2}^{6n}\cdot (t_{d} t_{d'})\\
=&C_{6n-1}^{6n}(6n \cdots 1)((6n+1) \cdots 2)((6n+2)\cdots 3) \cdot (t_{d} t_{d'})\\
=&(t_{d} t_{d'}) \cdot (6n+2)^{t_{6n}t_{6n+1}}\cdot (6n \cdots 1) ((6n+1) \cdots 2)((6n+1)\cdots 3)\cdot (t_{d} t_{d'})\\
=&(t_{d} t_{d'}) \cdot (6n+2)^{t_{6n}t_{6n+1}}\cdot ((6n-1)\cdots 1) (6n \cdots 1) ((6n+1) \cdots 2)\cdot (t_{d} t_{d'})\\
=&(t_{d} t_{d'}) \cdot (6n+2)^{t_{6n}t_{6n+1}}\cdot \tilde{C}_{6n-1} (6n+1)^{t_{6n}}\cdot(6n \cdots 1) (6n \cdots 2)\cdot (t_{d} t_{d'})\\
=&(t_{d} t_{d'}) \cdot (6n+2)^{t_{6n}t_{6n+1}}\cdot \tilde{C}_{6n-1}  (6n+1)^{t_{6n}}\cdot((6n-1)\cdots 1) (6n)  ((6n-1) \cdots 1)\cdot (t_{d} t_{d'})\\
=&(t_{d} t_{d'})^2 \cdot  t_{{c}'_{1,12}}\cdot \tilde{C}_{6n-1} \cdot t_{{c}'_{1,11}}\cdot \tilde{C}_{6n-1} \cdot t_{{c}'_{1,10}} \cdot \tilde{C}_{6n-1},
\end{align*}
where ${c}'_{1,10},  {c}'_{1,11},$ and ${c}'_{1,12}$ are defined by
\begin{align}
{c}'_{1,10}&=(t_{d} t_{d'})^{-1}(6n), \\
{c}'_{1,11}&= ((t_{d} t_{d'})^{-1}t_{6n})(6n+1), \\
{c}'_{1,12}&=((t_{d} t_{d'})^{-1}t_{6n}t_{6n+1})(6n+2).
\end{align}

From the two equations above, there is a repetitive pattern of words $t_{{c}'_{n,j}}\tilde{C}_{6n-1}$.
To re-arrange a sequence of words, write $a_i:=((6i-1)\cdots (6i-5)), b_i:= ((6n-1) \cdots (6i+1))$.
\begin{align*}
&t_{{c}'_{1,j}}\tilde{C}_{6n-1} \\
=&6^{t_{{c}'_{1,j}} b_1} t_{{c}'_{1,j}}((6n-1)\cdots 7) \cdot a_1 \\
=&6^{t_{{c}'_{1,j}} b_1} (12)^{t_{{c}'_{1,j}} b_2}\cdot t_{{c}'_{1,j}}  ((6n-1)\cdots 13) \cdot  a_2 a_1 \\
=&\cdots \\
=&6^{t_{{c}'_{1,j}} b_1} (12)^{t_{{c}'_{1,j}} b_2}\cdots (6n-6)^{t_{{c}'_{1,j}}b_{n-1}}t_{{c}'_{1,j}} \cdot a_n \cdots a_1 \\
=&t_{{c}'_{n,j}} t_{{c}'_{n-1,j}} \cdots t_{{c}'_{2,j}}t_{{c}'_{1,j}} \cdot a_n \cdots a_1,
\end{align*}
where ${c}'_{i,j}$, $i\neq 1$, is defined by
\begin{align}
{c}'_{n-i+1,j}:={t_{{c}'_{1,j}} b_i}(6i) \text{ }(i=1,\dots, n-1)
\end{align}
Note that ${c}'_{i,j}$ commutes with $a_k$ for $k\neq i$.

Therefore, $(1\dots (6n+2))^{12n}$ is given as follows.
\begin{align*}
&(1\dots (6n+2))^{12n}\\
=&(1\dots (6n+2))^{6n-2}(1\dots (6n+2))^{6n+2}\\
=&C_{6n+2}^{6n} (t_{d} t_{d'}) \cdot t_{{c}'_{1,9}} \cdot \tilde{C}_{6n-1}\cdot  t_{{c}'_{1,8}}\cdot \tilde{C}_{6n-1} \cdot t_{{c}'_{1,7}} \tilde{C}_{6n-1} \\
=&(t_{d} t_{d'})^2 t_{{c}'_{1,12}} \tilde{C}_{6n-1} \cdot t_{{c}'_{1,11}} \tilde{C}_{6n-1} \cdot t_{{c}'_{1,10}}  \tilde{C}_{6n-1}
 \cdot t_{{c}'_{1,9}}  \tilde{C}_{6n-1}\cdot  t_{{c}'_{1,8}} \tilde{C}_{6n-1} \cdot t_{{c}'_{1,7}} \tilde{C}_{6n-1} \\
=&(t_{d} t_{d'})^2 \prod_{j=12}^7\left( (\prod_{i=n}^1 t_{{c}'_{i,j}}) (a_n \cdots a_1) \right) \\
=&(t_{d} t_{d'})^2 \left(\prod_{j=12}^7 \prod_{i=n}^1 t_{c_{i,j}} \right) \cdot (a_n\cdots a_1)^6 \\
=&(t_{d} t_{d'})^2 \left(\prod_{j=12}^7 \prod_{i=n}^1 t_{c_{i,j}} \right) \cdot \left( \prod_{i=1}^n t_{d_{i}}t_{d'_{i}} \right),
\end{align*}
where $c_{i,j}$ is defined by
\begin{align} \label{eq:vc_comb2}
c_{i,j}= (a_{n-i+1})^{12-j}({c}'_{i,j}).\end{align}
Comparing this with the calculation in Section~\ref{sec:splitgen2}, we obtain the same curves.
\end{proof}


\appendix

\section{Error Estimates for the Approximations}
In this section, we describe the justification of linear approximations in Sections 4 and 5.
Since the arguments of both sections are similar, we only explain for Section 4.
Assume $n=3k$ with $k$ even.

Recall that $f(x,y)=x^2-y^3-1$, $\ep_n=(10 n)^{-1}$, and write $f=f(0,y)$ in this context.
Then the solutions of 
\begin{align} 
f^{2n}=\frac{3{\ep_n^{2n}}(f+1)}{(f+1)^k ((7n+3)f+6n+3)} \label{eq:app1}
\end{align}
are given by \[f \approx f_i = c_n {\ep_n}z_{2n}^i ,\]
where $c_n=\sqrt[2n]{\frac{1}{2n+1}}$ and $z_m=\exp(2\pi i/m)$.
In this process, we assume that we can approximate Equation~\eqref{eq:app1} as
\[f^{2n} =\frac{1}{2n+1} {\ep_n^{2n}}.\]
Since $f=-y^3-1$, applying one more linear approximation, we have the final solution
\[y\approx y_{i,j}=z_6^{2j+1}(1+\frac{c_n}{3}z_{2n}^i {\ep_n}).\] 

The main goal of this appendix is to prove that \[|y-y_{i,j}|=O\left(\frac{1}{n^2}\right).\]
Since the minimum distance between two distinct approximated solutions $y_{i,j}$ is $O(1/n)$,
while the approximation error is $O(1/n^2)$,
each true solution is uniquely associated with its corresponding approximated solution for sufficiently large $n$.

To prove this, define a rational map $g(x)$ as
\[g(x):=\frac{3{\ep_n^{2n}}(x+1)^{1-k} }{(7n+3)x+6n+3}.\]
The map $g$ is holomorphic on $\C\setminus \{ -1, -\frac{6n+3}{7n+3}\}$ for $k\ge 2$.
We first need to approximate the roots of \[x^{2n}=g(x)\]
close to the origin.

\begin{lem}\label{lem:app1}
Choose $r_n=2c_n\ep_n.$
Then the number of solutions of \[x^{2n}=g(x)\] in the disk \[D_n=\{|x|\le r_n\}\] is $2n$.
Moreover, for each $i=0,\dots,2n-1$, there exists a holomorphic branch $h_i$ of $g^{1/2n}$ on $D_n$ such that
\[h_i(0)=f_i,\]
and the equation
\[x=h_i(x)\]
has a unique solution \(x_i\) in \(D_n\).
The \(2n\) solutions \(x_i\) are precisely the solutions of
\(x^{2n}=g(x)\) in \(D_n\).
\end{lem}

\begin{proof}
To prove the first claim, use Rouche's theorem: if $|g(x)|<|x^{2n}|$ for $|x|=r_n$, then $x^{2n}-g(x)$ and $x^{2n}$ have the same number of roots in $|x|\le r_n$, counted with multiplicity.

For $|x|=r_n$, \[|x^{2n}|= \frac{2^{2n}\ep_n^{2n}}{2n+1}.\]
On the other hand,
\[|g(x)|\le\frac{3\ep_n^{2n}(1-r_n)^{1-k}}{6n+3-(7n+3)r_n}.\]

Since \(r_n=O(1/n)\), the right-hand side is bounded by \(2 \ep_n^{2n}/(2n+1)\) for large $n$.
Therefore, $\frac{2 \ep_n^{2n}}{2n+1}<\frac{2^{2n}\ep_n^{2n}}{2n+1}$ implies 
\[|g(x)|<|x^{2n}|\]
for sufficiently large \(n\).
By Rouché's theorem, \(x^{2n}-g(x)\) has exactly \(2n\) zeros in \(D_n\), counted with multiplicity.

Since $D_n$ is simply connected and $g$ has no zeros on $D_n$, 
there exists $2n$ holomorphic branch $h_i$ ($i=0,\dots, 2n-1$) of \( g^{1/2n}\) on \(D_n\) such that $h_i^{2n}=g$ on $D_n$.
It is clear that \[h_i(0)=f_i=c_n \ep_n z_{2n}^i\]
and a solution $x_i$ of $x=h_i(x)$ is the solution of $x^{2n}=g(x)$.
Since $|h_i(x)|<|x|$ for $x\in \partial D_n$, by Rouche's theorem, $x-h_i(x)$ has a unique zero $x_i\in D_n$.
\end{proof}

Now we measure the distance between the true solution $x_i$ and the approximated solution $f_i$.
\begin{lem} \label{lem:app2}
$|x_i-f_i|=O\left(\frac{1}{n^2}\right)$ as $n\arrowr \infty$.
\end{lem}

\begin{proof}
We have used the estimation $x_i=h_i(x_i) \approx h_i(0)=f_i$.
The error of estimation is calculated as
\[|x_i-f_i|=|h_i(x_i)-h_i(0)|\le r_n \sup_{|z|\le r_n}|h_i'(z)|.\]

We claim that \[|h_i'(x)|=O\left(\frac{1}{n}\right)\] as $n \arrowr \infty$.
Note that 
\begin{align*}
|h_i(x)|&=c_n \ep_n |1+x|^{\frac{1-k}{2n}}|6n+3+(7n+3)x|^{-\frac{1}{2n}} \\
&\le c_n \ep_n (1-r_n)^{\frac{1}{2n}-\frac{1}{6}}(6n+3 -(7n+3)r_n)^{-\frac{1}{2n}}=O\left(\frac{1}{n}\right).
\end{align*}
Since
\begin{align*}
\left|\frac{h_i'(x)}{h_i(x)}\right|=\frac{1}{2n}\left|\frac{g'(x)}{g(x)}\right|&=\frac{1}{2n}\left| \frac{1-k}{1+x} - \frac{7n+3}{6n+3+(7n+3)x}\right| \\
&\le \frac{1}{2n}\left( \frac{k-1}{1-r_n} + \frac{7n+3}{6n+3-(7n+3)r_n}\right)=O(1),
\end{align*}
we have 
\begin{align*}
|h_i'(x)|\le |h_i(x)| O(1)=O\left(\frac{1}{n}\right).
\end{align*} 
Therefore
\[|x_i-f_i|\le r_n \sup_{D_n}|h_i'| =O\!\left(\frac1n\right)O\!\left(\frac1n\right)=O\!\left(\frac1{n^2}\right).\]
\end{proof}

Since $f-f_i=-(y^3-y_{i,j}^3)=-(y-y_{i,j})(y^2+y y_{i,j}+y_{i,j}^2)$
and \[m<|y^2+y\,y_{i,j}+y_{i,j}^2|<M\]
for some constants $m,M>0$, independent of $n$,
this implies
\[|y-y_{i,j}|=O(|f-f_i|).\]
Therefore
\[|y-y_{i,j}|=O\!\left(\frac1{n^2}\right),\]
which proves the claim.


\section*{Acknowledgements}
The author would like to thank Jongil Park, Ki-Heon Yun, and Ju A Lee 
for many helpful discussions and valuable suggestions. 
The author was supported by the National Research Foundation of Korea (NRF) grant funded by the Korea government (RS-2024-00392067).

\section*{Conflict of Interest}
The author declares that there is no conflict of interest.


\end{document}